\documentclass[12pt]{article}
\usepackage[usenames,dvipsnames]{color}
\usepackage{amsmath,amssymb,amsthm}
\usepackage{thm-restate} %% NEW
\usepackage{tikz}
\usetikzlibrary{shapes}
\usepackage{wasysym}

\usepackage[colorlinks=true,citecolor=black,linkcolor=black,urlcolor=blue]{hyperref}

\newcommand{\arxiv}[1]{\href{http://arxiv.org/abs/#1}{\texttt{arXiv:#1}}}

\usepackage{graphicx,color}

\def\nfrac#1#2{{\textstyle\frac{#1}{#2}}}
\def\dfrac#1#2{\lower0.15ex\hbox{\large$\frac{#1}{#2}$}}

\newtheorem{theorem}{Theorem}[section]
\newtheorem{lemma}[theorem]{Lemma}
\newtheorem{remark}[theorem]{Remark}

\newtheorem{corollary}[theorem]{Corollary}
\numberwithin{equation}{section}

\def\LB{{\underline m}}
\def\UB{{\overline m}}
\def\imax{i_1}
\def\algA{{\tt FactorEasy}}
\def\algB{{\tt FactorUniform}}
\def\algC{{\tt FactorApprox}}
\def\accept{{\mathcal A}_0}
\def\ex{{\mathbb E}}
\def\pr{{\mathbb P}}

\def\strata{{\mathcal S}}

\def\Bi{\text{B1}+}
\def\Bii{\text{B2}+}
\def\Bvi{\text{B2}-}
\def\Bvii{\text{B1}-}
\def\textBvi{B2$-$}
\def\textBvii{B1$-$}
\def\Ci{\text{C}+}
\def\Cii{\text{C}-}
\def\Bipm{\text{B1}\pm}
\def\Biipm{\text{B2}\pm}
\def\Cpm{\text{C}\pm}
\def\textBipm{B1$\pm$}
\def\textBiipm{B2$\pm$}
\def\textCpm{C$\pm$}
\def\IVa{\text{IIb}+}
\def\IVb{\text{IIa}+}
\def\IVc{\text{IIc}+}
\def\IVd{\text{IIb}-}
\def\IVe{\text{IIa}-}
\def\IVf{\text{IIc}-}
\def\Va{\text{III}+}
\def\Vb{\text{III}-}
\def\textVb{\text{III}$-$}
\def\IIapm{\text{IIa}\pm}
\def\IIbpm{\text{IIb}\pm}
\def\IIcpm{\text{IIc}\pm}
\def\IIIpm{\text{III}\pm}
\def\textIIapm{IIa$\pm$}
\def\textIIbpm{IIb$\pm$}
\def\textIIcpm{IIc$\pm$}
\def\textIIIpm{III$\pm$}

\def\Prob{\operatorname{Prob}}

\def\nfrac#1#2{{\textstyle\frac{#1}{#2}}}
\def\dfrac#1#2{\lower0.15ex\hbox{\large$\frac{#1}{#2}$}}

\title{Uniform generation of spanning regular subgraphs\\ of a dense graph}

\author{Pu Gao\thanks{Supported by the Australian Research Council Discovery Project DP160100835 and NSERC.}\\
\small Department of Combinatorics and Optimization\\[-0.8ex]
\small University of Waterloo\\[-0.8ex]
\small ON, N2L 3G1, Canada\\
\small \texttt{pu.gao@uwaterloo.ca}
\and 
Catherine Greenhill\thanks{Supported by the Australian Research Council Discovery Project DP140101519.}\\
\small School of Mathematics and Statistics\\[-0.8ex]
\small UNSW Sydney\\[-0.8ex]
\small NSW 2052, Australia\\[-0.3ex]
\small \texttt{c.greenhill@unsw.edu.au}
}

\begin{document}
\maketitle

\begin{abstract}
Let $H_n$ be a graph on $n$ vertices and let $\overline{H_n}$
denote the complement of $H_n$.  Suppose that $\Delta = \Delta(n)$ is the
maximum degree of $\overline{H_n}$.
We analyse three algorithms for sampling $d$-regular subgraphs 
($d$-factors) of
$H_n$.  This is equivalent to uniformly sampling
$d$-regular graphs which avoid a set $E(\overline{H_n})$ of forbidden edges.
Here $d=d(n)$ is a positive integer which may depend on $n$.

Two of these algorithms produce a uniformly random
$d$-factor of $H_n$ in expected runtime which is linear
in $n$ and low-degree polynomial in $d$ and $\Delta$.
The first algorithm applies when $(d+\Delta)d\Delta = o(n)$.
This improves on an earlier algorithm by the first author,
which required constant $d$ and at most a linear number of edges in
$\overline{H_n}$.
The second algorithm applies when $H_n$ is regular
and $d^2+\Delta^2 = o(n)$, adapting an approach developed by the first 
author together with Wormald.  The third algorithm is a simplification
of the second, and produces an approximately uniform $d$-factor
of $H_n$ in time $O(dn)$. 
Here the output distribution differs from uniform by $o(1)$ in
total variation distance, provided that $d^2+\Delta^2 = o(n)$.
\end{abstract}

\section{Introduction}

Enumeration and uniform generation of graphs of given degrees has been
an active research area for four decades, with many applications.
Research on graph enumeration started in 1978 when Bender and Canfield~\cite{BC} 
obtained the asymptotic number of graphs with bounded degrees. The constraint of bounded degrees was slightly relaxed by Bollob\'{a}s~\cite{B} to the order of $\sqrt{\log n}$. A further relaxation on the degree constraints was obtained by McKay~\cite{M} to $d=o(n^{1/3})$, and by McKay and Wormald~\cite{MW2} to $d=o(\sqrt{n})$, for the $d$-regular graph case. 
On the other hand, asymptotic enumeration for dense $d$-regular graphs 
was performed in 1990 by McKay and Wormald~\cite{MW3} for $d, n-d\ge n/\log n$,
leaving a gap in which no result was known: namely, when $\min\{d, n-d\}$ lies
between $n^{1/2}$ and $n/\log n$. This gap was filled recently
by Liebenau and Wormald~\cite{LW16}. 
It is not known whether the problem of counting
all graphs with a given degree sequence is \#P-complete. Erd{\H o}s et al.~\cite{EKMS}
proved that the problem is self-reducible, in the bipartite setting,
and hence approximately counting such graphs can be reduced to sampling.

Uniform generation of graphs with given degrees has an equally long history, 
and is closely related to asymptotic enumeration. Tinhofer~\cite{Tinhofer} 
is among the first who investigated algorithmic heuristics, and realised 
that efficient uniform generation, or approximately uniform generation, 
is difficult. 
Graph enumeration arguments can sometimes be adapted to provide algorithms 
for uniform generation of graphs with given degrees. The proofs in~\cite{BC,B} 
immediately yield a simple rejection algorithm, which uniformly generates 
graphs of degrees at most $O(\sqrt{\log n})$ 
in expected polynomial time
(polynomial as a function of $n$). 
The switching arguments in~\cite{M,MW2} were adapted to give a polynomial-time 
(in expectation) algorithm~\cite{MW} for uniformly 
generating random $d$-regular graphs when $d=O(n^{1/3})$. However, overcoming 
the $n^{1/3}$ barrier was extremely challenging, with no progress for more 
than two decades. A major breakthrough was obtained by Wormald and the first author~\cite{GWSIAM}, %\cite{GWFOCS,GWSIAM}, 
by modifying the McKay-Wormald algorithm~\cite{MW} to a more flexible form, which allows the use and classification of different 
types and classes of switchings. This new technique greatly extended the 
range
of degree sequences for which uniform generation is possible. 
The first application~\cite{GWSIAM} %\cite{GWFOCS,GWSIAM} 
of the new technique gave an algorithm for
uniform generation of $d$-regular graphs with expected runtime $O(nd^3)$
when $d=o(\sqrt{n})$. The second application~\cite{GW2SODA} gave an
algorithm for uniform generation of graphs with power-law degree sequences
with exponent slightly below 3, with expected runtime
$O(n^{2.107})$ with high probability. 

Although uniform generation of random regular graphs has been challenging and 
remains open for $d$ of order at least $\sqrt{n}$, various approximate samplers have been developed and proven to run in polynomial time. Jerrum and Sinclair gave a MCMC-based scheme~\cite{JS} which 
gives polynomial-time sampling of graphs with a given degree sequence, so
long as the degree sequence satisfies a condition called P-\emph{stability}.
No explicit bound on the runtime was given.
Since all regular sequences are P-stable, the Jerrrum--Sinclair algorithm
generates $d$-regular graphs approximately uniformly in polynomial time. 
%although no explicit bound on the runtime was given. 
Kannan, Tetali and Vempala used another Markov chain to sample random regular bipartite graphs. They proved that the mixing time is polynomial without giving an explicit bound. Later this chain was extended by Cooper, Dyer and Greenhill~\cite{CDG,CDG-corrigendum} to generate random regular graphs, with mixing time bounded by $d^{24}n^9\log n$. There are asymptotically approximate samplers~\cite{BKS,GWSIAM,KV,SW,Zhao} % ~\cite{BKS,GWFOCS,GWSIAM,KV,SW,Zhao} 
which generate $d$-regular graphs very fast, typically with linear or up to quadratic runtime, and with an output of total variation distance $o(1)$ from the uniform.

Jerrum and Sinclair's algorithm~\cite{JS} built on their earlier work on a Markov chain
for sampling perfect matchings (1-factors) in a given graph~\cite{JS2}. 
A natural generalisation of this problem is that of sampling random $d$-factors in a given
graph, which we call the \emph{host graph} and denote by $H_n$, where $n$ is the number of vertices.  

%
%Compared with generating random $d$-regular graphs with $n$ vertices, 
%it is naturally more challenging 
%to generate random $d$-factors of a given graph $H_n$ with $n$ vertices. 
%We call $H_n$ the \emph{host graph}, 
We call ${\overline{H_n}}$, the complement of  $H_n$, the {\em forbidden graph}. 
Let $\Delta$ denote the maximum degree of $\overline{H_n}$.

The computational complexity of counting $d$-factors in a given host graph is not
known in general, though the special case $d=1$ (counting perfect matchings in a given graph)
is \#P-complete. There are asymptotic enumeration results for graphs with given degrees
and a specified set of forbidden edges, see for example~\cite{M,McKay}.
However, there has not been much research in the direction of uniform generation of $d$-factors.
The first author gave a rejection algorithm in~\cite{G14} which has an
expected linear runtime when $d=O(1)$ and ${\overline{H_n}}$ contains at most 
a linear number of edges. Here $H_n$ is not 
necessarily regular, and the maximum degree of ${\overline{H_n}}$ can be linear. 
We are not aware of any other algorithms which have explicit polynomial bounds 
on the runtime and vanishing bounds on the approximation error.  
For approximate sampling, the Jerrum--Sinclair algorithm~\cite[Section~4]{JS} generates an
approximately uniform $d$-factor of $H_n$ in polynomial time, as long as  
$d+\Delta\leq n/2 + 1$.
It seems unlikely that the approach of Cooper et al~\cite{CDG} can
be adapted to the setting of sampling $d$-factors, due to the complexity of the
analysis. Erd{\H o}s et al.~\cite{EKMS} analysed a Markov chain
algorithm which uniformly generates bipartite graphs with a given half-regular 
degree sequence,
avoiding a set of edges which is the union of a 1-factor and a star.
Here ``half-regular'' means that the degrees on one side of the bipartition
are all the same, with the possible exception of the centre of the star.

The aim of this paper is to develop efficient algorithms that 
sample $d$-factors of $H_n$ uniformly or 
approximately uniformly.  We will describe and
analyse three different algorithms, which we call \algA, \algB\ and \algC.
Our main focus is \algB, which is an algorithm for uniformly generating 
$d$-factors of an $(n-1-\Delta)$-regular host graph $H_n$, when $d$ and
$\Delta$ are not too large.  For smaller $d$ and $\Delta$,
the simpler algorithm \algA\ is more efficient, and does not
require $H_n$ to be regular (here $\Delta$ is the maximum degree
of $\overline{H_n}$). Finally,
\algC\ is a linear-time algorithm for generating $d$-factors of $H_n$
asymptotically approximately uniformly, under the same conditions
as \algB.

Our results are stated formally below.
All asymptotics are as the number of vertices $n$ tends to infinity,
along even integers if $d$ is odd. 
Throughout, $d=d(n)$ and $\Delta=\Delta(n)$ are
positive integers which may depend on $n$.

%\newpage

% FactorSimple
\begin{restatable}{theorem}{thmA}\label{thm:A} 
Let $H_n$ be a graph on $n$ vertices such that
the maximum degree of $\overline{H_n}$ is $\Delta$. 
%an $(n-1-\Delta)$-regular graph. 
The algorithm \algA\ uniformly generates a $d$-factor of $H_n$. 
If $(d+\Delta)d\Delta=o(n)$ then \algA\ runs in time 
$O((d+\Delta)^3n)$ in expectation.  
%\end{theorem}
\end{restatable}

% FactorAdvance
\begin{restatable}{theorem}{thmB}\label{thm:B} 
Let $H_n$ be an $(n-\Delta-1)$-regular graph
on $n$ vertices.
The algorithm \algB\ uniformly generates a $d$-factor of $H_n$. 
%If $d^2+\Delta^2=o(n)$ then \algB\ runs in time 
If $d^2+\Delta^2=o(n)$ then the time complexity of \algB\  is
$O((d+\Delta)^4n+d^3n\log n)$ a.a.s., 
%and in time $O(M)$ in expectation, where
and is $O(M)$ in expectation, where
\[
M=O\Big((d+\Delta)^4(n+\Delta^3)+(d+\Delta)^8d^2\Delta^2/n +(d+\Delta)^{10} d^2\Delta^3 /n^2\Big).
\]
\end{restatable}

\noindent{\em Note added in proof:}  A new technique called {\em incremental relaxation} has recently been developed by Arman, Wormald and the first author~\cite{relaxation}. This technique significantly improves the run time of switching-based algorithms by incrementally performing rejections. It is very likely that the run time of \algA\ and \algB\ can be significantly improved by adapting the incremental relaxation scheme from~\cite{relaxation}.

% FactorApprox
\begin{restatable}{theorem}{thmC}\label{thm:C} 
%\begin{theorem}\label{thm:C} 
Let $H_n$ be an $(n-\Delta-1)$-regular graph on $n$ vertices.
Assume that $d^2+\Delta^2=o(n)$. The algorithm \algC\ approximately generates a
 uniformly random $d$-factor of $H_n$ in time $O(dn)$ in expectation. 
The distribution of the output of \algC\ differs from uniform by $o(1)$ in total variation distance.   
\end{restatable}

\noindent {\em Remark:} For simplicity we considered regular spanning subgraphs in this paper, and the host graph is assumed regular for \algB\ and \algC. However, all of these algorithms are flexible and can be modified to cope with more general degree sequences for the spanning subgraph and for the host graph. For instance, the McKay-Wormald algorithm~\cite{MW} can uniformly generate a subgraph of $K_n$ with a given degree sequence where the maximum degree is not too large. We believe that \algA\ can be easily modified for general degree sequences, by calling~\cite{MW}, and then slighly modifying the analysis, with essentially the same switching. To cope with denser irregular subgraphs or sparser irregular host graphs, \algB\ and \algC\ can be modified accordingly, 
by possibly introducing new types of switchings. We do not pursue this here.

\smallskip

In Section~\ref{sec:frame}, we describe the common framework of 
\algA\ and \algB, and define some key parameters that will appear in these 
two algorithms and in the approximate sampling algorithm \algC. 
More detail on the structure of the paper can be found at the end of
Section~\ref{sec:A}.

\section{The framework for uniform generation} \label{sec:frame}

The new approach of Gao and Wormald~\cite{GWSIAM} %~\cite{GWFOCS,GWSIAM} 
gives a common framework for Algorithms \algA\ and \algB, 
which we will now describe. The Gao--Wormald scheme reduces to the McKay--Wormald 
algorithm~\cite{MW} by setting certain parameters to some trivial values. 
Our algorithm \algA\ is indeed an adaptation of the simpler McKay-Wormald algorithm, 
whereas  \algB\ uses the full power of~\cite{GWSIAM}%~\cite{GWFOCS,GWSIAM}, 
which allows it to cope with a larger range of $d$ and $\Delta$.

It is convenient to think of the host graph $H_n$ as being
defined by a 2-colouring of the complete graph $K_n$ with the
colours red and black, where edges of $\overline{H}$ are coloured red,
while edges of $H_n$ are coloured black.
Then our aim is to uniformly sample $d$-factors of $K_n$ which
contain no red (forbidden) edges.

Both algorithms \algA\ and \algB\ begin by generating a uniformly 
random $d$-regular graph $G$ on $\{ 1,2,\ldots, n\}$.
For the range of $d$ which we consider, this can be done using the
Gao--Wormald algorithm REG~\cite{GWSIAM}. %~\cite{GWFOCS,GWSIAM}.
Typically this initial graph $G$ will contain some 
red edges. Let $\strata_i$ denote the set of all $d$-regular graphs containing
precisely $i$ red edges. The sets $\strata_0,\strata_1,\ldots$ are called  {\em strata}. 
For some positive integer parameter $\imax$, which we must 
define for each 
algorithm, let
\[
\accept = \bigcup_{i=0}^{\imax} \, {\strata}_i.
\] 
If the initial graph $G$ does not belong to $\accept$ then the algorithm
will reject $G$ and restart.
Otherwise, the initial graph $G$ belongs to $\accept$ and so it does not contain
too many red edges.  Then the algorithm will perform a sequence of switching 
operations (which we must define), starting from $G$, until it reaches a 
$d$-regular graph with no red edges.  At each switching step there is a chance of a
\emph{rejection}: if a rejection occurs then the algorithm will restart.
This rejection scheme must also be defined, for each algorithm.

In \algA, only one type of switching (Type I) is used. Each such switching 
reduces the number of red edges by exactly one. As soon as \algA\ 
reaches a $d$-regular graph with no red edges, it outputs that graph,
provided that no rejection has occurred. 
When $(d+\Delta)d\Delta$ is of order $n$
or greater, the probability of a rejection occurring in \algA\ is very close to 1 and \algA\ becomes inefficient. 

\algB\ reduces the probability of rejection by permitting some switchings 
that are invalid in \algA, as well as introducing other types of switchings. 
Switchings which typically reduce the number of red edges by exactly one will 
still be the most frequently applied switchings, although in \algB\ we relax these 
switchings slightly so that certain operations which were forbidden
in \algA\ will be permitted in \algB. The other types of switchings 
do not necessarily reduce the number of red edges. 
Rather counterintuitively, some switchings will create more red edges. 
The new types of switchings are introduced for the same reason as the use
of the rejection scheme: to remedy the distortion of the distribution which
arises by merely applying Type I switchings. 

We will specify parameters 
$\rho_{\tau}(i)$, for $0\le i\le \imax$ and $\tau\in \Gamma$, where 
$\Gamma$ is the set of the types of switchings to be applied in \algB. 
In each step of \algB, ignoring rejections that may occur with a small 
probability, a switching type $\tau\in\Gamma$ is chosen with probability 
$\rho_{\tau}(i)$ if the current graph $G$ is in $\strata_i$. Then, 
given $\tau$, a random switching of type $\tau$ is performed. 
As mentioned before, the Type I switchings are the most common: in fact, 
we set $\rho_{\text{I}}(i)$ close to 1 for each $0\le i\le \imax$. 
If the current graph lies in $\strata_0$ then a Type I switching applied to
that graph will simply output the graph, but any other type of switching
will not.  Hence \algB\ does not always
immediately produce output as soon as it reaches a graph in $\strata_0$,
unlike \algA.
%This is another property that 
%differentiates \algB\ from \algA: the algorithm does not always output the resulting graph immediately if it reaches a graph in $\strata_0$.

As a preparation, we compute the expected number of red edges in a 
random $d$-regular subgraph of $K_n$.

\begin{lemma} \label{lem:expectation}
Let $G$ be a uniformly random $d$-regular graph on $\{1,2,\ldots, n\}$. The expected number of red edges in $G$ is $|E(\overline{H_n})|d/(n-1)$.
\end{lemma}

\begin{proof} Let $uv$ be a red edge in $K_n$ (that is, an edge in ${\overline{H_n}}$). 
We know that $u$ is incident with $d$ edges in $G$, and by symmetry each of the $n-1$ edges incident with $u$ in $K_n$ is equally likely to be in $G$. Thus, the probability that $uv\in G$ is $d/(n-1)$. There are exactly $|E(\overline{H_n})|$ red edges in $K_n$. By linearity of expectation, the expected number of red edges in $G$ is 
$|E(\overline{H_n})| d/(n-1)$.
\end{proof}

\section{The algorithm \algA} \label{sec:small}

Define
\begin{eqnarray}
\imax &=& 2|E(\overline{H_n})|d/n; \label{def-imax} \label{imax}\\
\accept &=& \cup_{i=0}^{\imax} \, {\strata}_i.
\end{eqnarray}
The following is a direct corollary of Lemma~\ref{lem:expectation},
using Markov's inequality.
\begin{corollary}\label{cor1:imax}
With probability at least $1/2+o(1)$, a uniformly random $d$-regular graph on 
$\{ 1,2,\ldots, n\}$ contains at most $\imax$ red edges.
\end{corollary}

We now define the switching operation which we use in \algA,
called a \emph{3-edge-switching}. 
To define a 3-edge-switching from the current graph $G$, 
choose a sequence of vertices $(v_0,\ldots, v_5)$ such that 
$v_0v_1$ is a red edge in~$G$, 
$v_2v_3$ and $v_4v_5$ are edges in $G$ (with repetitions allowed), 
and the choice satisfies the following conditions:
\begin{itemize}
\item $v_2v_3$ and $v_4v_5$ are black edges in $G$, 
\item $v_0v_5$, $v_1v_2$ and $v_3v_4$ are all absent in $G$, and are all black in $K_n$;
\item The vertices $v_0,\ldots, v_5$ are distinct, except that $v_2=v_5$ is permitted.
\end{itemize}
We say that the 6-tuple $\boldsymbol{v}=(v_0,v_1,\ldots, v_5)$ is \emph{valid} if it 
satisfies these conditions.
Given a valid 6-tuple $\boldsymbol{v}$, the 3-edge-switching determined
by $\boldsymbol{v}$ deletes
the three edges $v_0v_1$, $v_2v_3$, $v_4v5$ 
and replaces them with the edges $v_1v_2$, $v_3v_4$, $v_0v_5$, producing
a new graph $G'$.  This switching operation is denoted by
$(G,{\boldsymbol v})\mapsto G'$, and is illustrated in  
Figure~\ref{3-switch}.  Red edges and red non-edges are also labelled `r', to ensure visibility.
\begin{figure}[ht!]
\begin{center}
\begin{tikzpicture}[scale=1.2]
\draw [-,very thick,red] (0,1.732) -- (1,1.732);
\node [above] at (0.5,1.8) {r}; %% NEW
\draw [-,very thick] (0,0) -- (-0.5,0.866);
\draw [-,very thick] (1,0) -- (1.5,0.866);
\draw [-,very thick,dashed] (0,0) -- (1,0);
\draw [-,very thick,dashed] (1.5,0.866) -- (1,1.732);
\draw [-,very thick,dashed] (0,1.732) -- (-0.5,0.866);
\draw [fill] (0,0) circle (0.1);
\draw [fill] (1,0) circle (0.1);
\draw [fill] (1.5,0.866) circle (0.1);
\draw [fill] (1,1.732) circle (0.1);
\draw [fill] (0,1.732) circle (0.1);
\draw [fill] (-0.5,0.866) circle (0.1);
\node [left] at (-1.5,0.366) {$G\in\mathcal{S}_i$};
%\node [left] at (-1.5,0.866) {$\mathcal{S}_{i}$};
\node [above] at (0.0,1.832) {$v_0$};
\node [above] at (1,1.832) {$v_1$};
\node [right] at (1.6,0.866) {$v_2$};
\node [below] at (1,-0.1) {$v_3$};
\node [below] at (0,-0.1) {$v_4$};
\node [left] at (-0.6,0.866) {$v_5$};
\draw [->,line width = 1mm] (3,0.866) -- (4,0.866);
\begin{scope}[shift={(6,0)}]
\draw [-,very thick,red,dashed] (0,1.732) -- (1,1.732);
\node [above] at (0.5,1.8) {r}; %% NEW
\draw [-,very thick,dashed] (0,0) -- (-0.5,0.866);
\draw [-,very thick,dashed] (1,0) -- (1.5,0.866);
\draw [-,very thick] (0,0) -- (1,0);
\draw [-,very thick] (1.5,0.866) -- (1,1.732);
\draw [-,very thick] (0,1.732) -- (-0.5,0.866);
\draw [fill] (0,0) circle (0.1);
\draw [fill] (1,0) circle (0.1);
\draw [fill] (1.5,0.866) circle (0.1);
\draw [fill] (1,1.732) circle (0.1);
\draw [fill] (0,1.732) circle (0.1);
\draw [fill] (-0.5,0.866) circle (0.1);
%\node [right] at (2.5,0.866) {$\mathcal{S}_{i-1}$};
\node [right] at (2.5,0.366) {$G'\in\mathcal{S}_{i-1}$};
\node [above] at (0.0,1.832) {$v_0$};
\node [above] at (1,1.832) {$v_1$};
\node [right] at (1.6,0.866) {$v_2$};
\node [below] at (1,-0.1) {$v_3$};
\node [below] at (0,-0.1) {$v_4$};
\node [left] at (-0.6,0.866) {$v_5$};
\end{scope}
\end{tikzpicture}
\caption{A 3-edge switching}
\label{3-switch}
\end{center}
\end{figure}
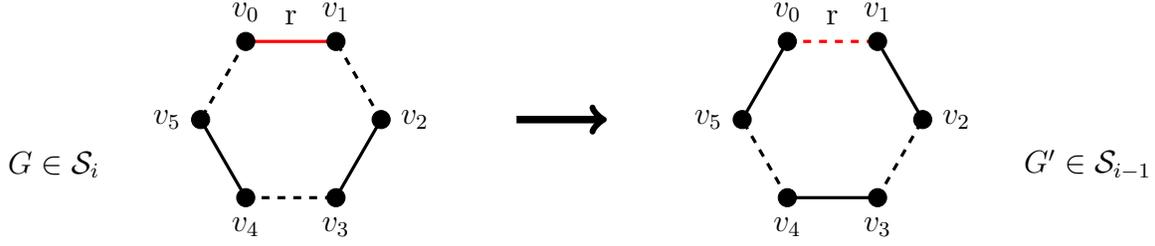

Each 
\emph{3-edge switching} reduces the number of red edges by exactly one. 
The inverse operation, obtained by reversing the arrow in Figure~\ref{3-switch},
is called an {\em inverse 3-edge switching}.  The 6-tuple 
$(v_0,v_1,\ldots, v_5)$ is valid
for the inverse 3-edge switching if exactly one new red edge $v_0v_1$ is introduced,
no multiple edges are introduced, and all vertices are distinct
except possibly $v_2=v_5$.

Define
\begin{align*}
\UB(i)&=2i(dn)^2,\\ 
\LB(i) &=
  (2|E({\overline{H_n}})|-2i) d^2 (dn-2i-8d)-4|E({\overline{H_n}})| d^3 (d+\Delta) - 4i \Delta d^2n.
\end{align*}
Let $f(G)$ denote the number of valid 6-tuples $(v_0,v_1,\ldots, v_5)$ which
determine a 3-edge switchings that can be applied to $G$, and let $b(G)$ 
denote the 
number of valid 6-tuples $(v_0,v_1,\ldots, v_5)$ which determine a 
inverse 3-edge switchings that can be applied to $G$.
In the following lemma, we show that $f(G)$ is approximately $\UB(i)$ and $b(G)$ is approximately $\LB(i)$ for $G\in\strata_i$. 

\begin{lemma}\label{lem:fb}
Suppose that $(d+\Delta)d\Delta = o(n)$.
Let $i\in \{1,\ldots, i_1\}$. For all $G\in\strata_i$ we have
\begin{align*}
 2i(dn-2i-4d)(dn-2i-7d)-6id^2(d+\Delta)n&\le f(G)\le \UB(i),\\
 \LB(i)&\le b(G) \le 2|E({\overline{H_n}})| d^3n.
 \end{align*}
Hence 
\[ f(G) = \UB(i)\left(1 +O\left(\frac{d+\Delta}{n}\right)\right),\quad  b(G) 
 = 2|E({\overline{H_n}})| d^3n\left(1 + O\left(\frac{d+\Delta}{n}
  \right)\right).\]
\end{lemma}

%The proof of Lemma~\ref{lem:fb} is deferred to Section~\ref{sec:DeferredProofLemma}.

\begin{proof}
For the upper bound of $f(G)$, note that there are exactly $2i$ ways to choose 
$(v_0, v_1)$, and then at most $dn$ ways to choose $(v_2,v_3)$ and at most 
$dn$ ways to choose $(v_4,v_5)$.

For the lower bound of $f(G)$, there are $2i$ ways to choose $(v_0, v_1)$. 
Given that, there are at least $(dn-2i-4d)$ ways to choose $(v_2, v_3)$ 
such that $v_2v_3$ is a black edge in $G$, and $v_0$, $v_1$, $v_2$ and $v_3$ 
are all distinct. Then there are at least $(dn-2i-7d)$ ways to choose $(v_4,v_5)$ 
such that $v_4v_5$ is a black edge in $G$ and 
\[ v_4\not\in \{ v_0,v_1,v_2,v_3\},\qquad v_5\not\in\{ v_0,v_1,v_3\}\]
(as $v_5$ is allowed to coincide with $v_2$).
This gives at least $2i(dn-2i-4d)(dn-2i-7d)$ choices of $(v_0,\ldots, v_5)$,
but some of these choices are not valid: specifically, we must subtract
the number of choices such that one of $v_0v_5$, $v_1v_2$ or $v_3v_4$ is 
either an edge in $G$ (either black or red), or is a red non-edge 
(that is, an edge in ${\overline{H_n}}\cap {\overline G}$).
\begin{itemize}
\item There are at most
$3\cdot 2id^2dn$ choices such that $v_0v_5$ or $v_1v_2$ or $v_3v_4$ 
is an edge in $G$;
\item There are at most $3\cdot 2idn\Delta d$ choices such that $v_0v_5$ or $v_1v_2$ or $v_3v_4$ is a red non-edge.
\end{itemize}
Subtracting these yields the desired lower bound for $f(G)$.

Next we consider $b(G)$. For the upper bound, there are at most 
$2|E(\overline{H})|$ ways to choose $(v_0, v_1)$, since  $|E(\overline{H})|$
is an upper bound on the number of red non-edges in $G$.
Then there are at most $d$ 
ways to choose $v_2$, at most $d$ ways to choose $v_5$ and at most $dn$ ways to 
choose $(v_3,v_4)$. This yields the required upper bound. 

For the lower bound, there are exactly $2(|E(\overline{H})|-i)$ ways to choose
$(v_0,v_1)$, as the number of red non-edges in $G$ is exactly 
$|E(\overline{H})|-i$. Then there are exactly $d^2$ ways to choose $v_2$ and $v_5$ 
such that $v_0v_5$ and $v_1v_2$ are edges in $G$ (not necessarily black).
Note here that $v_2=v_5$ is permitted. The number of ways to choose $(v_3,v_4)$ 
such that $v_3v_4$ is a black edge in $G$, and 
$\{v_3,v_4\}\cap \{v_0,v_1,v_2,v_5\}=\emptyset$ is at least $dn-2i-8d$. 
This gives at least $2(|E(\overline{H})|-i)d^2(dn-2i-8d)$ 
choices for $(v_0,\ldots, v_5)$, but some of these choices are not
valid: specifically, we must subtract
the number of choices where one of $v_2v_3$, $v_4v_5$ is either an edge 
in $G$ or a red non-edge, or one of $v_0v_5$, $v_1v_2$ is a red edge.  
\begin{itemize}
\item There are at most $2\cdot 2|E(\overline{H})|d^3(d+\Delta)$ choices
such that $v_2v_3$ or $v_4v_5$ is either an edge in $G$ or a red non-edge;
\item There are at most $2\cdot 2i\Delta d\cdot d n$ choices such that
 $v_0v_5$ or $v_1v_2$ is a red edge.
\end{itemize}
Subtracting these gives the stated lower bound for $b(G)$. 

Finally, the last statement %of the Lemma~\ref{lem:fb}
follows immediately from the 
above bounds by~(\ref{imax}) and noting that $|E({\overline{H_n}})|\le \Delta n/2$.  % \hfill \qed
\end{proof}

\subsection{Algorithm \algA: definition and analysis} \label{sec:A}

First, \algA\ calls the Gao--Wormald algorithm REG~\cite{GWSIAM} %\cite{GWFOCS,GWSIAM}
 to generate a uniformly random $d$-regular graph $G$ on $\{ 1,2,\ldots, n\}$. 
If $G$ contains more than $\imax$ red edges then \algA\ restarts. 
Otherwise, \algA\ iteratively performs a sequence of switching steps. 
At each step, there is a chance that the current graph, $G$, might be rejected
(this is called \emph{f-rejection}) and there is a chance that the graph $G'$ selected
as the ``next'' graph might be rejected (this is called \emph{b-rejection}).
Here ``f'' is short for ``forward'' and ``b'' is short for ``backward''.
The probability of f-rejection and b-rejection is carefully chosen to maintain uniformity.

Let $G_t$ be the graph obtained after $t$ switching steps, and assume $G_t=G\in\strata_i$. 
The $(t+1)$-th switching step is composed of the following substeps: 

\begin{enumerate}
\item[(i)] If $i=0$ then output $G$.
\item[(ii)] If $i>0$ then uniformly at random choose 
a red edge in $G$ and choose two further edges in $G$ (of any colour), 
with repetition allowed. 
Randomly label the two endvertices of the red edge as $v_0$ and $v_1$, 
and the endvertices of the other two edges as $v_2$ and 
$v_3$, and $v_4$ and $v_5$ respectively. 
If $\boldsymbol{v}=(v_0,\ldots, v_5)$ is a valid 6-tuple then let $(G,\boldsymbol{v})\mapsto G'$ 
be the 3-edge switching induced by this 6-tuple, as in Figure~\ref{3-switch}. 
Otherwise (when the 6-tuple is not valid), perform an f-rejection.
\item[(iii)] If no f-rejection is performed then perform a b-rejection with 
probability 
\[1-\frac{b(G')}{\LB(i-1)}.\] 
\item[(iv)] If no b-rejection is performed then set $G_{t+1}=G'$.
\end{enumerate} 
If any rejection occurs then \algA\ restarts.  

\medskip

There is no deterministic
upper bound on the running time of the algorithm, due to the chance of rejections.
That is, \algA\ is a \emph{Las Vegas} algorithm.  But to prove
Theorem~\ref{thm:A} we will show that the expected number of restarts is $O(1)$.

The proof of the following lemma is deferred to Section~\ref{sec:runtime}.

\begin{restatable}{lemma}{Arun}\label{algA-runtime}
\algA\ can be implemented so that if there are $O(1)$ restarts during its run,
its time complexity is $O((d+\Delta)^3 n)$.
\end{restatable}

We now prove  Theorem~\ref{thm:A}, restated below for convenience. 

\thmA*
\begin{proof}%[Proof of Theorem~\ref{thm:A}] 
First we prove uniformity by induction. Recall that $G_0$ is a uniformly random 
$d$-regular graph in $\accept$ if no initial rejection occurs. 
If $G_0\in \strata_i$ then clearly $G_0$ is uniformly distributed over $\strata_i$. 

Next we prove that if $G_t$ is uniformly distributed over $\strata_i$ then $G_{t+1}$ 
is uniformly distributed over $\strata_{i-1}$, assuming that no rejection
occurs.  
For every $G\in\strata_i$ and $G'\in\strata_{i-1}$, let $\Psi(G,G')$ 
denote the set of valid 6-tuples $\boldsymbol{v}=(v_0,\ldots, v_5)$ 
such that $(G,\boldsymbol{v})\mapsto G'$ is a 3-edge-switching,
and let 
\begin{equation}
\label{psi-def}
 \Psi(G') = \bigcup_{G\in\strata_i} \Psi(G,G').
\end{equation}
Given $G_t=G$, note that $\UB(i)$ is exactly the 
number of choices of a 6-tuple of vertices $(v_0,\ldots, v_5)$, with repetition 
allowed, such that $v_0v_1$ is a red edge in $G$, and $v_2v_3$ and $v_4v_5$ are edges 
of $G$.
Thus, the probability that $G$ is converted to $G'$ without f-rejection in step $t+1$ 
is equal to 
\[ \frac{|\Psi(G,G')|}{\UB(i)}.
\] 
The probability that no b-rejection occurs is equal to 
$\LB(i-1)/b(G')=\LB(i-1)/|\Psi(G')|$.
Write $\rho_t=\pr(G_t=G\mid G_t\in\strata_i)$, 
which is invariant over all graphs $G\in \strata_i$ by the inductive hypothesis. 
Then
\begin{align*}
\pr(G_{t+1}=G'\mid G_t\in\strata_i) &= \sum_{G\in \strata_i}\,
\pr(G_t=G\mid G_t\in\strata_i)\cdot 
    \frac{|\Psi(G,G')|}{\UB(i)}\cdot
   \frac{\LB(i-1)}{|\Psi(G')|}\\
  &= \rho_t \,\frac{\LB(i-1)}{\UB(i)},
\end{align*}
using the fact that the union in (\ref{psi-def}) is disjoint.
Hence $\pr(G_{t+1}=G'\mid G_t\in\strata_i)$ does not depend on $G'$, which implies
that $G_{t+1}$ is uniformly distributed over $\strata_{j-1}$ if no rejection occurs.

Next, we prove that if $(d+\Delta)d\Delta=o(n)$ then \algA\ runs in 
$O((d+ \Delta)^3 n)$ time in expectation. The runtime for generating a random $d$-regular graph on $\{ 1,2,\ldots, n\}$ using the Gao--Wormald algorithm~\cite{GWSIAM} is $O(d^3n)$ in expectation. 
By Lemma~\ref{lem:fb}, the probability of an f-rejection or a b-rejection in each step is $O((d+\Delta)/n )$. By definition of $i_1$ and $\accept$, the algorithm
 \algA\ performs at most $\imax=O(|E({\overline{H_n}})|d/n)=O(d\Delta)$ 
switching steps. Thus, the overall probability of any 
f-rejection or b-rejection is at most
\[
 O\left(\frac{d+\Delta}{n} \right)\imax = O\left(\frac{(d+\Delta)d\Delta}{n} \right)
\]
which is $o(1)$ because $(d+\Delta)d\Delta=o(n)$. It follows 
from this and from Corollary~\ref{cor1:imax}, that in expectation \algA\ restarts $O(1)$ times. 
Therefore, by Lemma~\ref{algA-runtime},
%It only remains to show that in each run of \algA\ without rejection, 
the time complexity of \algA\ is $O((d+\Delta)^3n)$ in expectation. %We will prove this part in Section~\ref{sec:runtime}.
\end{proof}

We close this section by proving the following lemma, which follows easily 
from Lemma~\ref{lem:fb}. This result, which will be useful later, only
requires a rather weak condition on $d$ and $\Delta$.

\begin{lemma}\label{lem:sizeRatio}
Assume that $d+\Delta=o(n)$. Then for any $i=O(d\Delta)$, 
\[
\frac{|\strata_{i-1}|}{|\strata_i|}=\frac{i\, n}{|E({\overline{H_n}})|d}\left(1 + O\left(\frac{d+\Delta}{n}
  \right)\right).
\]
\end{lemma}
\begin{proof}
Let $\ex\, f(G)$ be the expected value of $f(G)$ when $G\in \strata_i$ is chosen
uniformly at random, and let
$\ex\, b(G')$ be the expected value of $b(G')$ when $G'\in\strata_{i-1}$ is chosen uniformly 
at random.
Then
\[
\frac{|\strata_{i-1}|}{|\strata_i|}=\frac{\ex\, f(G)}{\ex\,  b(G')}
\]
and
the result follows by Lemma~\ref{lem:fb}. 
\end{proof}

\medskip

When $(d+\Delta)d\Delta$ is no longer negligible compared to $n$, 
the probability that an f-rejection or b-rejection occurs in \algA\
before reaching $\strata_0$ is very close to 1.  In this case,
\algA\ becomes very inefficient as it has to restart many times. 
In Section~\ref{sec:B} we define \algB, which will use 
4-edge switchings instead of 3-edge switchings, giving more room 
for performing valid operations.  In addition, various new ideas 
will be incorporated into the design of \algB\ to achieve uniformity in 
the output and efficiency when $d^2+\Delta^2=o(n)$. 
Since the uniform sampler \algB\ can be treated as an extension of the
approximate sampler \algC, we will introduce \algC\ first, in
Section~\ref{sec:approx} below.  
The runtime analysis of \algA\ and \algB, and the proof that the output of \algC\ is
sufficiently close to uniform, are deferred
to Section~\ref{sec:runtime}. %where the proofs of our main theorems are completed.

For the algorithms \algB\ and \algC, we restrict the host graph $H_n$ to be regular
(specifically, $(n-\Delta-1)$-regular), to simplify the analysis. 
However, we believe that
\algB\ (and \algC) can be modified to work for irregular host graphs~$H_n$.

%Before moving on to these more advanced algorithms, we first complete
%the analysis of \algA\ by proving Lemma~\ref{lem:fb}.

%\subsection{Proof of Lemma~\ref{lem:fb}} \label{sec:DeferredProofLemma}

\section{The approximate algorithm: \algC}\label{sec:approx}

We now assume that $H_n$ is $(n-\Delta-1)$-regular, which implies that
$|E(\overline{H_n})|=\Delta n/2$. 
Define
\begin{align}
\imax &= \dfrac23 d\Delta; \label{def-imaxb}\\
\accept &= \cup_{i=0}^{\imax} \, {\strata}_i. \label{def-A0}
\end{align}

\smallskip

\begin{corollary}\label{cor2:imax}
With probability at least $1/4+o(1)$, a uniformly random $d$-regular graph on 
$\{ 1,2,\ldots, n\}$ contains at most $\imax$ red edges.
\end{corollary}

\begin{proof}
By Lemma~\ref{lem:expectation}, the expected number of red edges in a 
random $d$-factor of $G$ is asymptotic to $\Delta d/2$. The result follows by Markov's inequality.
\end{proof}

Gao and Wormald~\cite{GWSIAM} %\cite{GWFOCS,GWSIAM} 
gave an algorithm called REG* for
generating regular graphs asymptotically approximately uniformly in runtime $O(dn)$. 
\algC\ uses REG* to generate a random $d$-regular graph on $K_n$. 

\begin{corollary}\label{cor3:imax}
With probability at least $1/4+o(1)$, the output of REG* contains at most $\imax$ red edges.
\end{corollary}

\begin{proof}
When $d^2=o(n)$, the distribution of the output of REG* differs from the uniform by 
$o(1)$ in total variation distance, see~\cite[Section~10]{GWSIAM}. 
Combining this with Corollary~\ref{cor2:imax} completes the proof.
\end{proof}

Both \algB\ and \algC\ use \emph{4-edge-switchings}, which we now define.
To define a 4-edge-switching from the current graph $G$,
choose a sequence of 8 vertices $(v_0,v_1,\ldots, v_7)$ such that $v_0v_1$ is a red
edge in $G$ and $v_2v_3$, $v_4v_5$, $v_6v_7$ are other edges in $G$,
and such that 
\begin{itemize}
\item $v_0v_7$, $v_1v_2$, $v_3v_4$, $v_5v_6$ are \emph{not present} in $G$;
\item none of $v_3v_4,\,\, v_4v_5,\,\, v_5v_6$ is red in $K_n$; 
\item no two of the eight vertices are equal except for possibly $v_2=v_7$;
\item Either none of $v_1v_2$, $v_2v_3$, $v_0v_7$ and $v_6v_7$ is red; \\
or exactly one of $v_1v_2$, $v_2v_3$, $v_0v_7$ and $v_6v_7$ is red;\\
or both $v_1v_2$ and $v_2v_3$ are red, and both $v_0v_7$ and $v_6v_7$ are black;\\
or both $v_1v_2$ and $v_2v_3$ are black, and both $v_0v_7$ and $v_6v_7$ are red.
\end{itemize}
The sequence of vertices $(v_0,\ldots, v_7)$ is said to be
\emph{valid} if it satisfies the above properties.
Given a valid 8-tuple of vertices, the switching operation deletes the
four edges $v_0v_1$, $v_2v_3$, $v_4v_5$, $v_6v_7$,
and replaces them with the edges $v_0v_7$, $v_1v_2$, $v_3v_4$ and $v_5v_6$,
as in Figure~\ref{fig:I-switch}. The resulting graph is denoted by $G'$. 
This operation is called a 4-edge-switching.
We say that the 4-edge-switching is \emph{valid} if it arises from a valid 8-tuple.

Under the switching, the colour of each edge or 
non-edge stays the same, but edges become non-edges
and vice-versa.
%\end{itemize}
In Figure~\ref{fig:I-switch}, 
a solid line indicates an edge of $G$ and a dashed line represents
an edge of $\overline{G}$; that is, a non-edge in $G_t$. 
We label each edge (or non-edge) by the allowed colour,
where `b', `r' and `b/r' denote `black', `red', and `black or red',
respectively.   Edges or non-edges which are definitely red are also 
shown coloured red in the figure.

\begin{figure}[ht!]
\begin{center}
\begin{tikzpicture}[scale=1.4]
\draw [-,very thick,red] (0,2.414) -- (1,2.414);  % v_0v_1 %% red
\node [above] at (0.5, 2.514) {r};
\draw [-,very thick] (1.707,1.707) -- (1.707,0.707);  % v_2v_3
\node [right] at (1.807,1.207) {b/r};
\draw [-,very thick] (-0.707,1.707) -- (-0.707,0.707); % v_6v_7
\node [left] at (-0.807,1.207) {b/r};
\draw [-,very thick] (0,0) -- (1,0);   % v_4v_5
\node [below] at (0.5, -0.1) {b};
\draw [-,very thick,dashed] (0,0) -- (-0.707,0.707);  % v_5v_6
\node [left] at (-0.3,0.2) {b};
\draw [-,very thick,dashed] (-0.707,1.707) -- (0,2.414); % v_7v_0
\node [left] at (-0.3,2.2) {b/r};
\draw [-,very thick,dashed] (1,2.414) -- (1.707,1.707);  %v_1v_2
\node [right] at (1.3,2.2) {b/r};
\draw [-,very thick,dashed] (1.707,0.707) -- (1,0); % v_3v_4
\node [right] at (1.3,0.15) {b};
\draw [fill] (0,0) circle (0.1);
\draw [fill] (1,0) circle (0.1);
\draw [fill] (1.707,0.707) circle (0.1);
\draw [fill] (1.707,1.707) circle (0.1);
\draw [fill] (1,2.414) circle (0.1);
\draw [fill] (0,2.414) circle (0.1);
\draw [fill] (-0.707,1.707) circle (0.1);
\draw [fill] (-0.707,0.707) circle (0.1);
%\node [left] at (-1.707,1.207) {$\mathcal{S}_{i}$};
\node [above] at (0.0,2.514) {$v_0$};
\node [above] at (1,2.514) {$v_1$};
\node [right] at (1.807,1.707) {$v_2$};
\node [right] at (1.807,0.707) {$v_3$};
\node [below] at (1,-0.1) {$v_4$};
\node [below] at (0,-0.1) {$v_5$};
\node [left] at (-0.807,0.707) {$v_6$};
\node [left] at (-0.807,1.707) {$v_7$};
\draw [->,line width = 1mm] (3,0.866) -- (4,0.866);
\begin{scope}[shift={(6,0)}]
\draw [-,very thick,dashed,red] (0,2.414) -- (1,2.414); % v_0v_1 %% NEW: red
\node [above] at (0.5, 2.514) {r};   
\draw [-,very thick,dashed] (1.707,1.707) -- (1.707,0.707); % v_2v_3
\node [right] at (1.807,1.207) {b/r};
\draw [-,very thick,dashed] (-0.707,1.707) -- (-0.707,0.707); % v_6v_7
\node [left] at (-0.807,1.207) {b/r};
\draw [-,very thick,dashed] (0,0) -- (1,0);  % v_4v_5
\node [below] at (0.5, -0.1) {b};
\draw [-,very thick] (0,0) -- (-0.707,0.707);  % v_5v_6
\node [left] at (-0.3,0.2) {b};
\draw [-,very thick] (-0.707,1.707) -- (0,2.414);   % v_7v_0
\node [left] at (-0.3,2.2) {b/r};
\draw [-,very thick] (1,2.414) -- (1.707,1.707);   % v_1v_2
\node [right] at (1.3,2.2) {b/r};
\draw [-,very thick] (1.707,0.707) -- (1,0);   % v_3v_4
\node [right] at (1.3,0.15) {b};
\draw [fill] (0,0) circle (0.1);
\draw [fill] (1,0) circle (0.1);
\draw [fill] (1.707,0.707) circle (0.1);
\draw [fill] (1.707,1.707) circle (0.1);
\draw [fill] (1,2.414) circle (0.1);
\draw [fill] (0,2.414) circle (0.1);
\draw [fill] (-0.707,1.707) circle (0.1);
\draw [fill] (-0.707,0.707) circle (0.1);
%\node [right] at (2.707,1.207) {$\mathcal{S}_{i-1}$};
\node [above] at (0.0,2.514) {$v_0$};
\node [above] at (1,2.514) {$v_1$};
\node [right] at (1.807,1.707) {$v_2$};
\node [right] at (1.807,0.707) {$v_3$};
\node [below] at (1,-0.1) {$v_4$};
\node [below] at (0,-0.1) {$v_5$};
\node [left] at (-0.807,0.707) {$v_6$};
\node [left] at (-0.807,1.707) {$v_7$};
\end{scope}
\end{tikzpicture}
\caption{A 4-edge-switching}
\label{fig:I-switch}
\end{center}
\end{figure}
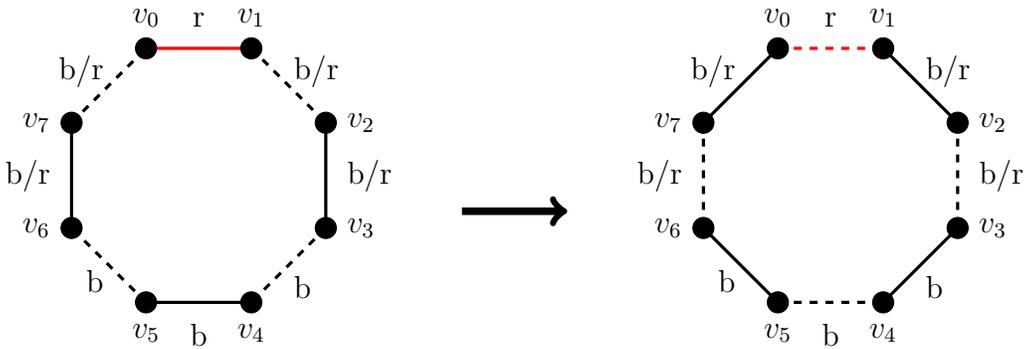

%\subsection{Algorithm \algC: definition and analysis} \label{sec:C}

The structure of \algC\ is similar to that of \algA, except that
there is no rejection after the first step.

First, \algC\ uses REG* to repeatedly generate a random $d$-regular graph on 
the vertex set $\{ 1,2,\ldots, n \}$ until the resulting graph belongs to
$\accept$. Next, \algC\ repeatedly applies random
4-edge switchings from the current graph,
until a graph without red edges is produced. 
Finally, \algC\ outputs this graph, which is a $d$-factor of $H_n$. 

To choose a uniformly random 4-edge switching from a current graph
$G=G_t$, we can uniformly at random choose a red edge and label its end 
vertices by $v_0$ and $v_1$. Then uniformly choose three edges of $G$
(repetition allowed) and label their endvertices. 
If the resulting 8-tuple is valid then perform the corresponding
4-edge-switching to produce
the new graph $G'=G_{t+1}$. 
Otherwise, repeat until a valid 4-edge-switching is obtained. 

\bigskip

The proof of the following lemma is given in Section~\ref{sec:runtime}.

\begin{restatable}{lemma}{Cdist}\label{lem:algC-distribution}
Under the conditions of Theorem~\ref{thm:C}, 
the output of \algC\ is a random $d$-factor of $H_n$ whose distribution differs
from the uniform distribution by $o(1)$ in total variation distance.
\end{restatable}

Using this lemma, we can prove Theorem~\ref{thm:C}, restated here for convenience.

\thmC*
\begin{proof}
The runtime of REG* is $O(dn)$ and by Corollary~\ref{cor3:imax},
a constant number of attempts will be sufficient to generate a random
$d$-regular graph with at most $\imax$ red edges
in expectation.  Now 
a given graph $G$, consider choosing $v_0v_1$ to
be a uniformly random red edge of $G$ and choosing each of 
$v_2v_3$, $v_4v_5$, $v_6v_7$ to be a uniformly random edge of
$G$ (with repetition allowed). It is easy to see that with high
probability, this random choice of edges $v_0v_1$, $v_2v_3$, $v_4v_5$,
$v_6v_7$ defines a valid 4-edge switching which reduces the number of
red edges by exactly one.
(The argument is very similar to the proof of Lemma~\ref{lem:iterations}.)
Hence, the cost of time in performing one 4-edge 
switching is $O(1)$, and \algC\ consists of performing $O(\imax)$ 
switching steps in expectation and with high probability. 
Since $\imax = O(d\Delta)$, it follows that the runtime of \algC\ is
$O(dn+\imax)=O(dn)$.   This completes the proof, by Lemma~\ref{lem:algC-distribution}.
\end{proof}

\section{The exactly uniform sampler: \algB}\label{sec:B}

In this section, our aim is to define and analyse an algorithm
for uniform generation of $d$-factors of $H_n$, which is efficient
for larger values of $d$ and $\Delta$ than \algA.
Specifically, we assume that $d^2+\Delta^2 = o(n)$.
As in Section~\ref{sec:approx}, we assume that $H_n$ is $(n-\Delta-1)$-regular
and define $\imax$ and $\accept$ as in (\ref{def-imaxb}), (\ref{def-A0}).

The analysis of \algA\ given in Section~\ref{sec:A} shows that
the probability of an \mbox{f-rejection} or a b-rejection depends on the gap between 
the upper and lower bounds of $f(G)$ and $b(G)$.  To obtain an algorithm
which is still efficient for larger values of $d$ and $\Delta$, we must
reduce the variation of $f(G)$ and $b(G)$ among
$G\in {\strata}_i$, for some family of switchings.  
We use several techniques to reduce this variation: 

\begin{itemize}
\item We use 4-edge-switchings instead of 3-edge switchings.
\item We will perform more careful counting than the analysis of
Lemma~\ref{lem:fb}.
\item We will occasionally allow switchings which create new red edges.
\item We will introduce other types of switchings to ``boost'' the probability
of graphs which are otherwise not created sufficiently often. 
These types of switchings are called {\em boosters}.
\end{itemize}
The last two of these techniques follow the approach of~\cite{GWSIAM}. %\cite{GWFOCS,GWSIAM}.
In particular, every switching introduced will have a \emph{type} 
and a \emph{class}.  The number of switchings of type $\tau$ that
can be performed on a given graph $G$ is denoted by $f_\tau(G)$,
and the number of switchings of class $\alpha$ that can be applied to
other graphs to produce $G$ is denoted by $b_\alpha(G)$.
We will also need parameters $\UB_\tau(i)$ and $\LB_\alpha(i)$ which
satisfy
\[ \UB_\tau(i)\geq \max_{G\in\strata_i} f_\tau(G),\qquad
  \LB_\alpha(i)\leq \min_{G\in\strata_i} b_\alpha(G).
\]
Further details of the structure of algorithm \algB\ will be explained
in Section~\ref{sec:alg}.

\subsection{Type I switchings}\label{s:forward}

\algB\ will mainly use the 4-edge-switching shown in Figure~\ref{fig:I-switch}, 
which replaces four edges by four new edges. We will call this the
\emph{Type I} switching, as we will define other types of switchings for use
in \algB\ later.

Suppose that a Type I switching transforms a graph $G$ into a graph
$G'\in\strata_i$.  
Then the initial graph $G$ can be in different strata, depending 
on the colour of $v_1v_2$, $v_2v_3$, $v_6v_7$ and $v_0v_7$. If these edges are 
all black then we say that the Type I switching is in \emph{Class A}.
In this case, $G$ has exactly one more red edge than $G'$.  
Other Type I switchings are categorised into 
different classes, as shown in Table~\ref{tab:TypeI}. 

%%%%%%%%%%%%%%%%%%%%%%%%%%%%%%%%%%%%%%%%%%%%%%%%%%%%%%%%%%
%%  BIG TABLE!!!
%%%%%%%%%%%%%%%%%%%%%%%%%%%%%%%%%%%%%%%%%%%%%%%%%%%%%%%%%%
\begin{table}[ht!]
\begin{center}
\renewcommand{\arraystretch}{1.2}
\begin{tabular}{|c|c|c|}
\hline
class &  action & the switching  \\
\hline
%%%%%%%% Type I, Class A
A & $\mathcal{S}_{i+1} \rightarrow \mathcal{S}_i$ &
\begin{tikzpicture}[scale=0.7]
\draw [-,very thick,red] (0,2.414) -- (1,2.414); %% RED
\node [above] at (0.5, 2.514) {r};
\draw [-,very thick] (1.707,1.707) -- (1.707,0.707);
\draw [-,very thick] (-0.707,1.707) -- (-0.707,0.707);
\draw [-,very thick] (0,0) -- (1,0);
\draw [-,very thick,dashed] (0,0) -- (-0.707,0.707);
\draw [-,very thick,dashed] (-0.707,1.707) -- (0,2.414);
\draw [-,very thick,dashed] (1,2.414) -- (1.707,1.707);
\draw [-,very thick,dashed] (1.707,0.707) -- (1,0);
\draw [-,very thick,dashed] (1.707,0.707) -- (1,0);
\draw [fill] (0,0) circle (0.1); \draw [fill] (1,0) circle (0.1);
\draw [fill] (1.707,0.707) circle (0.1); \draw [fill] (1.707,1.707) circle (0.1);
\draw [fill] (1,2.414) circle (0.1); \draw [fill] (0,2.414) circle (0.1);
\draw [fill] (-0.707,1.707) circle (0.1); \draw [fill] (-0.707,0.707) circle (0.1);
\draw [->,line width = 1mm] (3,0.866) -- (4,0.866);
\begin{scope}[shift={(6,0)}]
\draw [-,very thick,dashed,red] (0,2.414) -- (1,2.414); %% RED
\node [above] at (0.5, 2.514) {r};
\draw [-,very thick,dashed] (1.707,1.707) -- (1.707,0.707);
\draw [-,very thick,dashed] (-0.707,1.707) -- (-0.707,0.707);
\draw [-,very thick,dashed] (0,0) -- (1,0);
\draw [-,very thick] (0,0) -- (-0.707,0.707);
\draw [-,very thick] (-0.707,1.707) -- (0,2.414);
\draw [-,very thick] (1,2.414) -- (1.707,1.707);
\draw [-,very thick] (1.707,0.707) -- (1,0);
\draw [fill] (0,0) circle (0.1); \draw [fill] (1,0) circle (0.1);
\draw [fill] (1.707,0.707) circle (0.1); \draw [fill] (1.707,1.707) circle (0.1);
\draw [fill] (1,2.414) circle (0.1); \draw [fill] (0,2.414) circle (0.1);
\draw [fill] (-0.707,1.707) circle (0.1); \draw [fill] (-0.707,0.707) circle (0.1);
\end{scope}
\end{tikzpicture}
\\
\hline
%%%%%%%%  Type I, Class B1, B7; Type IIa (IIg), Type A
\textBipm & $\mathcal{S}_{i} \rightarrow \mathcal{S}_i$ &
\begin{tikzpicture}[scale=0.7]
\draw [-,very thick,red] (0,2.414) -- (1,2.414); %% RED
\node [above] at (0.5, 2.514) {r};
\draw [-,very thick] (1.707,1.707) -- (1.707,0.707);
\node [right] at (1.3,2.2) {r};  % v_1v_2
\draw [-,very thick] (-0.707,1.707) -- (-0.707,0.707);
\draw [-,very thick] (0,0) -- (1,0);
\draw [-,very thick,dashed] (0,0) -- (-0.707,0.707);
\draw [-,very thick,dashed] (-0.707,1.707) -- (0,2.414);
\draw [-,very thick,dashed,red] (1,2.414) -- (1.707,1.707); %% RED
\draw [-,very thick,dashed] (1.707,0.707) -- (1,0);
\draw [-,very thick,dashed] (1.707,0.707) -- (1,0);
\draw [fill] (0,0) circle (0.1); \draw [fill] (1,0) circle (0.1);
\draw [fill] (1.707,0.707) circle (0.1); \draw [fill] (1.707,1.707) circle (0.1);
\draw [fill] (1,2.414) circle (0.1); \draw [fill] (0,2.414) circle (0.1);
\draw [fill] (-0.707,1.707) circle (0.1); \draw [fill] (-0.707,0.707) circle (0.1);
\draw [->,line width = 1mm] (3,0.866) -- (4,0.866);
\begin{scope}[shift={(6,0)}]
\draw [-,very thick,dashed,red] (0,2.414) -- (1,2.414); %% RED
\node [above] at (0.5, 2.514) {r};
\draw [-,very thick,dashed] (1.707,1.707) -- (1.707,0.707); 
\node [right] at (1.3,2.2) {r};  % v_1v_2
\draw [-,very thick,dashed] (-0.707,1.707) -- (-0.707,0.707);
\draw [-,very thick,dashed] (0,0) -- (1,0);
\draw [-,very thick] (0,0) -- (-0.707,0.707);
\draw [-,very thick] (-0.707,1.707) -- (0,2.414);
\draw [-,very thick,red] (1,2.414) -- (1.707,1.707); %% RED
\draw [-,very thick] (1.707,0.707) -- (1,0);
\draw [fill] (0,0) circle (0.1); \draw [fill] (1,0) circle (0.1);
\draw [fill] (1.707,0.707) circle (0.1); \draw [fill] (1.707,1.707) circle (0.1);
\draw [fill] (1,2.414) circle (0.1); \draw [fill] (0,2.414) circle (0.1);
\draw [fill] (-0.707,1.707) circle (0.1); \draw [fill] (-0.707,0.707) circle (0.1);
\end{scope}
\end{tikzpicture}
\\
\hline
%%  Type I, Class B2, B6; Type IIb (IIf), Class A
\textBiipm & $\mathcal{S}_{i+2} \rightarrow \mathcal{S}_i$ &
\begin{tikzpicture}[scale=0.7]
\draw [-,very thick,red] (0,2.414) -- (1,2.414); %% RED
\node [above] at (0.5, 2.514) {r};
\draw [-,very thick,red] (1.707,1.707) -- (1.707,0.707); %% RED
\draw [-,very thick] (-0.707,1.707) -- (-0.707,0.707);
\draw [-,very thick] (0,0) -- (1,0);
\node [right] at (1.807,1.207) {r}; % v_2v_3
\draw [-,very thick,dashed] (0,0) -- (-0.707,0.707);
\draw [-,very thick,dashed] (-0.707,1.707) -- (0,2.414);
\draw [-,very thick,dashed] (1,2.414) -- (1.707,1.707);
\draw [-,very thick,dashed] (1.707,0.707) -- (1,0);
\draw [-,very thick,dashed] (1.707,0.707) -- (1,0);
\draw [fill] (0,0) circle (0.1); \draw [fill] (1,0) circle (0.1);
\draw [fill] (1.707,0.707) circle (0.1); \draw [fill] (1.707,1.707) circle (0.1);
\draw [fill] (1,2.414) circle (0.1); \draw [fill] (0,2.414) circle (0.1);
\draw [fill] (-0.707,1.707) circle (0.1); \draw [fill] (-0.707,0.707) circle (0.1);
\draw [->,line width = 1mm] (3,0.866) -- (4,0.866);
\begin{scope}[shift={(6,0)}]
\node [right] at (1.807,1.207) {r}; % v_2v_3
\draw [-,very thick,dashed,red] (0,2.414) -- (1,2.414); %% RED
\node [above] at (0.5, 2.514) {r};
\draw [-,very thick,dashed,red] (1.707,1.707) -- (1.707,0.707); %% RED
\draw [-,very thick,dashed] (-0.707,1.707) -- (-0.707,0.707);
\draw [-,very thick,dashed] (0,0) -- (1,0);
\draw [-,very thick] (0,0) -- (-0.707,0.707);
\draw [-,very thick] (-0.707,1.707) -- (0,2.414);
\draw [-,very thick] (1,2.414) -- (1.707,1.707);
\draw [-,very thick] (1.707,0.707) -- (1,0);
\draw [fill] (0,0) circle (0.1); \draw [fill] (1,0) circle (0.1);
\draw [fill] (1.707,0.707) circle (0.1); \draw [fill] (1.707,1.707) circle (0.1);
\draw [fill] (1,2.414) circle (0.1); \draw [fill] (0,2.414) circle (0.1);
\draw [fill] (-0.707,1.707) circle (0.1); \draw [fill] (-0.707,0.707) circle (0.1);
\end{scope}
\end{tikzpicture}\\
\hline
%%  Type I, Class C1/C2
\textCpm  & $\mathcal{S}_{i+1} \rightarrow \mathcal{S}_i$ &
\begin{tikzpicture}[scale=0.7]
\node [right] at (1.807,1.207) {r}; % v_2v_3
\node [right] at (1.3,2.2) {r};  % v_1v_2
\draw [-,very thick,red] (0,2.414) -- (1,2.414); %% RED
\node [above] at (0.5, 2.514) {r};
\draw [-,very thick,red] (1.707,1.707) -- (1.707,0.707); %% RED
\draw [-,very thick] (-0.707,1.707) -- (-0.707,0.707);
\draw [-,very thick] (0,0) -- (1,0);
\draw [-,very thick,dashed] (0,0) -- (-0.707,0.707);
\draw [-,very thick,dashed] (-0.707,1.707) -- (0,2.414);
\draw [-,very thick,dashed,red] (1,2.414) -- (1.707,1.707); %% RED
\draw [-,very thick,dashed] (1.707,0.707) -- (1,0);
\draw [-,very thick,dashed] (1.707,0.707) -- (1,0);
\draw [fill] (0,0) circle (0.1); \draw [fill] (1,0) circle (0.1);
\draw [fill] (1.707,0.707) circle (0.1); \draw [fill] (1.707,1.707) circle (0.1);
\draw [fill] (1,2.414) circle (0.1); \draw [fill] (0,2.414) circle (0.1);
\draw [fill] (-0.707,1.707) circle (0.1); \draw [fill] (-0.707,0.707) circle (0.1);
\draw [->,line width = 1mm] (3,0.866) -- (4,0.866);
\begin{scope}[shift={(6,0)}]
\draw [-,very thick,dashed,red] (0,2.414) -- (1,2.414); %% RED
\node [above] at (0.5, 2.514) {r};
\node [right] at (1.807,1.207) {r}; % v_2v_3
\node [right] at (1.3,2.2) {r};  % v_1v_2
\draw [-,very thick,dashed,red] (1.707,1.707) -- (1.707,0.707); %% RED
\draw [-,very thick,dashed] (-0.707,1.707) -- (-0.707,0.707);
\draw [-,very thick,dashed] (0,0) -- (1,0);
\draw [-,very thick] (0,0) -- (-0.707,0.707);
\draw [-,very thick] (-0.707,1.707) -- (0,2.414);
\draw [-,very thick,red] (1,2.414) -- (1.707,1.707); %% RED
\draw [-,very thick] (1.707,0.707) -- (1,0);
\draw [fill] (0,0) circle (0.1); \draw [fill] (1,0) circle (0.1);
\draw [fill] (1.707,0.707) circle (0.1); \draw [fill] (1.707,1.707) circle (0.1);
\draw [fill] (1,2.414) circle (0.1); \draw [fill] (0,2.414) circle (0.1);
\draw [fill] (-0.707,1.707) circle (0.1); \draw [fill] (-0.707,0.707) circle (0.1);
\end{scope}
\end{tikzpicture}\\
\hline
\end{tabular}
\caption{The different classes of Type I switchings}
\label{tab:TypeI}
\end{center}
\end{table}
%%%%%%%%%%%%%%%%%%%%%%%%%%%%%%%%%%%%%%%%%%%%%%%%%
%%  End of the big table
%%%%%%%%%%%%%%%%%%%%%%%%%%%%%%%%%%%%%%%%%%%%%%%%%

Each row of Table~\ref{tab:TypeI}, other than the first row, defines two
new classes of switching, depending on how the vertices are
labelled.  For example, the second row defines Classes \Bi\ and \textBvii,
which we refer to collectively as \textBipm.
Every class with a name ending ``+'' arises from using the vertex
labelling shown on the left of Figure~\ref{fig:vertex-labels},
while those
classes ending in ``$-$'' arise from using the vertex labelling
shown on the right of Figure~\ref{fig:vertex-labels}.
For example, if $v_1v_2$ (respectively, $v_0v_7$) is red but not present in 
$G$, and all the other edges and non-edges involved in the switching, 
except for $v_0v_1$, are black, then this switching is in Class \Bi\ 
(respectively, \textBvii) and both $G$ and $G'$ belong to $\strata_i$. 

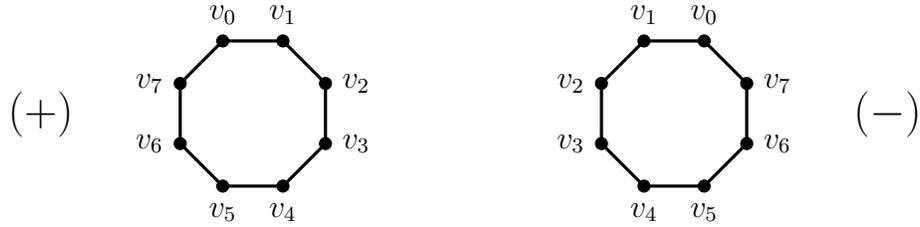
\begin{figure}[ht!]
\begin{center}
\begin{tikzpicture}[scale=0.8]
\draw [-,very thick] (0,2.414) -- (1,2.414);  % v_0v_1
\draw [-,very thick] (1.707,1.707) -- (1.707,0.707);  % v_2v_3
\draw [-,very thick] (-0.707,1.707) -- (-0.707,0.707); % v_6v_7
\draw [-,very thick] (0,0) -- (1,0);   % v_4v_5
\draw [-,very thick] (0,0) -- (-0.707,0.707);  % v_5v_6
\draw [-,very thick] (-0.707,1.707) -- (0,2.414); % v_7v_0
\draw [-,very thick] (1,2.414) -- (1.707,1.707);  %v_1v_2
\draw [-,very thick] (1.707,0.707) -- (1,0); % v_3v_4
\draw [fill] (0,0) circle (0.1);
\draw [fill] (1,0) circle (0.1);
\draw [fill] (1.707,0.707) circle (0.1);
\draw [fill] (1.707,1.707) circle (0.1);
\draw [fill] (1,2.414) circle (0.1);
\draw [fill] (0,2.414) circle (0.1);
\draw [fill] (-0.707,1.707) circle (0.1);
\draw [fill] (-0.707,0.707) circle (0.1);
%\node [left] at (-1.707,1.207) {$\mathcal{S}_{i}$};
\node [above] at (0.0,2.514) {$v_0$};
\node [above] at (1,2.514) {$v_1$};
\node [right] at (1.807,1.707) {$v_2$};
\node [right] at (1.807,0.707) {$v_3$};
\node [below] at (1,-0.1) {$v_4$};
\node [below] at (0,-0.1) {$v_5$};
\node [left] at (-0.807,0.707) {$v_6$};
\node [left] at (-0.807,1.707) {$v_7$};
\node [left] at (-2.3,1.2) {\Large{$(+)$}};
%%%
\begin{scope}[shift={(7,0)}]
\draw [-,very thick] (0,2.414) -- (1,2.414);  % v_0v_1
\draw [-,very thick] (1.707,1.707) -- (1.707,0.707);  % v_2v_3
\draw [-,very thick] (-0.707,1.707) -- (-0.707,0.707); % v_6v_7
\draw [-,very thick] (0,0) -- (1,0);   % v_4v_5
\draw [-,very thick] (0,0) -- (-0.707,0.707);  % v_5v_6
\draw [-,very thick] (-0.707,1.707) -- (0,2.414); % v_7v_0
\draw [-,very thick] (1,2.414) -- (1.707,1.707);  %v_1v_2
\draw [-,very thick] (1.707,0.707) -- (1,0); % v_3v_4
\draw [fill] (0,0) circle (0.1);
\draw [fill] (1,0) circle (0.1);
\draw [fill] (1.707,0.707) circle (0.1);
\draw [fill] (1.707,1.707) circle (0.1);
\draw [fill] (1,2.414) circle (0.1);
\draw [fill] (0,2.414) circle (0.1);
\draw [fill] (-0.707,1.707) circle (0.1);
\draw [fill] (-0.707,0.707) circle (0.1);
%\node [left] at (-1.707,1.207) {$\mathcal{S}_{i}$};
\node [above] at (0.0,2.514) {$v_1$};
\node [above] at (1,2.514) {$v_0$};
\node [right] at (1.807,1.707) {$v_7$};
\node [right] at (1.807,0.707) {$v_6$};
\node [below] at (1,-0.1) {$v_5$};
\node [below] at (0,-0.1) {$v_4$};
\node [left] at (-0.807,0.707) {$v_3$};
\node [left] at (-0.807,1.707) {$v_2$};
\node [right] at (3.3,1.2) {\Large{$(-)$}};
\end{scope}
\end{tikzpicture}
\caption{The vertex labellings for ``$+$'' and ``$-$'' classes,
respectively}
\label{fig:vertex-labels}
\end{center}
\end{figure}

\bigskip
%\newpage  %% Just to avoid a bad page break
Next we bound the number of Type I switchings which can be
performed in an arbitrary $G\in\strata_i$.
Define
\begin{align}
& \UB_{\text{I}}(i) \nonumber\\
&= 2i(dn)^3\left(1 + 28\left(\frac{(\Delta+d)^2}{n^2} + \frac{1}{n}\right)\right)  - 8i(d-1)^2d^2n^2 - 4i\Delta d^3 n^2 - 4i^2(dn)^2. \label{UBI}
\end{align}
The proof of the following lemma is similar to that of Lemma~\ref{lem:fb}.

\begin{lemma}\label{lem:ff}  Suppose that $i\in \{1,\ldots, i_1\}$.
Then for any $G\in{\strata}_i$,
\begin{eqnarray}
f_{\mathrm{I}}(G)&\le& \UB_{\text{I}}(i), \label{never-referred-to}\\
f_{\mathrm{I}}(G)&=& \UB_{\text{I}}(i)(1 + O((d^2+\Delta^2)/n^2+1/n)). \label{eq:lowerI}
\end{eqnarray}
\end{lemma}

\begin{proof}
Recall that $f_{\text{I}}(G)$ is the number of ways that a Type~I switching can be applied
to a given graph $G\in\mathcal{S}_i$. There are 
\[ 2i(dn-O(1))^3 \leq 2i(dn)^3 \] 
ways to choose the 8-tuple of vertices $(v_0,\ldots, v_7)$
so that $v_0v_1$ is a red edge in $G$, and $v_2v_3$,
$v_4v_5$, $v_6v_7$ are all edges in $G$, and these four edges are distinct.
From this we will subtract the following terms:
\begin{itemize}
\item Those with an unwanted vertex
coincidence, of which there are $O(id^3 n^2)$.
We will ignore these cases for the upper bound (\ref{never-referred-to}), 
but include them in the error term in (\ref{eq:lowerI}).
\item Those in which one of the dashed edges $v_1v_2$, $v_3v_4$, $v_5v_6$, $v_0v_7$
is present, either black or red, in $G$.  Using inclusion-exclusion, there are at least
\[ 4 \times 2i(d-1)^2(dn-6)(dn-8) - X\]
of these, where $X$ accounts for choices where at least two of them are present. 
It is easy to see that $X\leq \binom{4}{2} \cdot 2id^5 n= 12id^5n$.
So this expression is bounded below by
\[ 8i(d-1)^2 d^2 n^2 - 2i(dn)^3\, \left(56/n^2  + 6 d^2/n^2 \right).\]
(Here, and below, we treat some negligible terms as errors relative to the main term.)
\item Those for which at least one of $v_3v_4$, $v_5v_6$ or $v_4v_5$ is red. 
Suppose first that $v_3v_4$ is a red non-edge.  There are 
$2i$ ways to choose 
$v_0v_1$, then $\Delta n-2i$ ways to choose $v_3v_4$ to be red and $d^2$ ways to choose $v_2$ and $v_5$,
then at least $dn-6$ ways to choose $v_6v_7$ to be distinct from all chosen 
edges.   The number of choices where $v_5v_6$ is a red non-edge
is identical,
and the number of choices such that both $v_3v_4$ and $v_5v_6$ 
are red non-edges is at most $2id^3\Delta^2 n$.
There are at most $16id^4\Delta n$ choices of 8-tuple such that 
one of $v_3v_4$, $v_5v_6$ is a red non-edge and one of the 
dashed edges is present in $G$.
This gives the expression 
\begin{align*}
 & 4i(\Delta n - 2i)d^2(dn-6)  - 2id^3\Delta^2 n - 16id^4\Delta n\\
  &\geq  4i\Delta d^3 n^2 - 8i^2 d^3 n - 24 i \Delta d^2 n - 2id^3\Delta^2 n - 16id^4\Delta n\\
  &= 4i\Delta d^3 n^2 - 2i(dn)^3\,\left( 4i/n^2 + 12 \Delta/(dn^2)  + \Delta^2/n^2 + 8d\Delta/n^2
    \right)\\
  &\geq 4i\Delta d^3 n^2 - 2i(dn)^3\,\left( 11 d\Delta/n^2  + \Delta^2/n^2 + 12/n\right)
\end{align*}
since $i\leq i_1 = \dfrac{2}{3} d\Delta$ and $\Delta < n$.

Continuing, the number of choices of 8-tuple such that $v_4v_5$ is a red edge
is $2i(2i-2)(dn-4)(dn-6)$. 
From this we remove those where also one of the four dashed edges is present
(at most $16 i^2 d^3 n$ choices) or where also one of $v_3v_4$ or $v_5v_6$ is 
a red non-edge
(at most $8 i^2 \Delta d^2 n$ choices).  
This gives
\begin{align*}
&4i^2 (dn)^2 - 40 i^2 dn - 4 i (dn)^2 - 16 i^2 d^3 n - 8 i^2 \Delta d^2 n\\
&= 4i^2(dn)^2 - 2i(dn)^3\left( 20 i/(d^2n^2) + 2/(dn) + 8i/n^2 + 4i\Delta/(dn^2)\right)\\
&\geq 4i^2(dn)^2 - 2i(dn)^3\left( 13 \Delta^2/n^2 + 2/n + 
 6 d\Delta/n^2 \right).
\end{align*}
\item  As a final adjustment, which will appear with positive sign in the inclusion-exclusion,
we must consider choices where one of $v_1v_2$, $v_2v_3$ is red and one of
$v_6v_7$, $v_7v_0$ is red: there are at most 
$2idn(2i + \Delta d)^2 \leq 2i(dn)^3\cdot 9\Delta^2/n^2$ 
of these, since $i\leq i_1$.
\end{itemize}
Putting this together, we have proved that (\ref{eq:lowerI}) holds, and
that
\begin{align*}
&f_{\text{I}}(G)\\
 &\leq 2i(dn)^3\left(1 + \frac{56}{n^2} + \frac{14}{n} +
  \frac{14\Delta^2}{n^2} + \frac{8d\Delta}{n} + \frac{6d^2}{n^2}\right)
  - 8i(d-1)^2d^2n^2 - 4i\Delta d^3 n^2 - 4i^2(dn)^2\\
 &\leq 2i(dn)^3\left(1 + 28\left(\frac{(\Delta+d)^2}{n^2} + \frac{1}{n}\right)\right)  - 8i(d-1)^2d^2n^2 - 4i\Delta d^3 n^2 - 4i^2(dn)^2\\
 &= \UB_{\text{I}}(i).
\end{align*}
This completes the proof.
\end{proof}

\begin{remark}\emph{
If we did not allow the creation of structures in classes \textBipm, \textBiipm\ and 
\textCpm\ then the f-rejection probability would be too big. For instance, suppose 
that we did not allow Class \Bi\ (and that the other Class B and C switchings 
are permitted).  Then for most graphs in ${\strata}_i$, the typical number of 
forward switchings would be approximately 
\[ 2i(dn)^3 - 8i(d-1)^2 (dn)^2-4i\Delta d^3n^2-4i^2(dn)^2-2i\Delta d (dn)^2.\]
(This is approximately $\UB_I(i)$ with the term $2i\Delta d(dn)^2$ subtracted, which is
the typical number of 8-tuples corresponding to a Class \Bi\ switching.)
However, consider an extreme ``worst case'',
when $d=\Delta$ and $d$ divides $2i$,
and the $d$-factor $G$ is composed of two $d$-regular graphs, 
one with only red edges and the other with only black edges.  
Then there are no red edges in $G$ that are incident with a red dashed edge,
so no Class \Bi\ switchings need to be ruled out 
(as none are possible). 
In this case, the number of forward switchings in $G$ is
\[ 2i(dn)^3 - 8i(d-1)^2 (dn)^2-4i\Delta d^3n^2-4i^2(dn)^2.\]
Hence, $f_{\text{I}}(G)$ differs from $\max_{G\in\strata_i} f_{\text{I}}(G)$ by 
$\Omega(i\Delta d(dn)^2)$ for most graphs $G\in {\strata}_i$, causing an 
\mbox{f-rejection} probability of $\Omega(i\Delta d(dn)^2/i(dn)^3)=\Omega(\Delta/n)$ 
in a single step. But then the overall rejection probability will be too big, 
since there will be up to $i_1 = \Theta(d\Delta)$ steps,
and $d\Delta^2/n$ may not be $o(1)$ under the conditions of
Theorem~\ref{thm:B}. 
To reduce the probability of an f-rejection we must allow these Class \Bi\ 
switchings to proceed (and Class \textBvii\ switchings too, by symmetry). 
Similar arguments explain the introduction of Classes
\textBiipm\ and \textCpm. } 
\label{bad-example}
\end{remark}

\subsection{New switching types and counting inverse switchings}\label{s:backwards}

For each $\alpha\in\{\text{A}, \Bipm, \Biipm, \Cpm\}$ we count the inverse switchings of Class~$\alpha$.

\bigskip

\subsubsection{Class A}

Class A switchings are all of Type I.
In this section we will obtain a lower bound for $b_A(G)$ and an upper bound for 
the average of $b_A(G)$ over all $G\in\strata_i$.
%$\ex b_A(G)$. 
The following lemma will be useful.

\begin{lemma}\label{lem:average}
Assume that $d+\Delta=o(n)$ and $i\le\imax$.
Let $G$ be a $d$-factor chosen uniformly at random from ${\strata}_i$. Then the expected number of red 2-paths in $G$ is $O(i^2/n)$ and the expected number of pairs of red edges $\{\{u_1,v_1\}, \{u_2,v_2\}\}$ in $G$ such that $u_1u_2$ is either a red edge in $G$ or a red dashed edge is $O(i^2\Delta/n)$.
\end{lemma}

\begin{proof} 
Let $uvw$ be a red 2-path in $K_n$. We now bound the probability that $uvw$ is contained in a random $G\in\strata_i$. Let ${\mathcal W}$ denote the set of graphs in $\strata_i$ which contain $uvw$ and let ${\mathcal W}'$ denote the set of graphs in $\strata_{i-2}$ which do not contain neither of $uv, vw$. Consider the switching as shown in Figure~\ref{f:2-path}, where the 7 vertices must be distinct and all edges shown in the figure other than $uv$, $vw$
must be black.

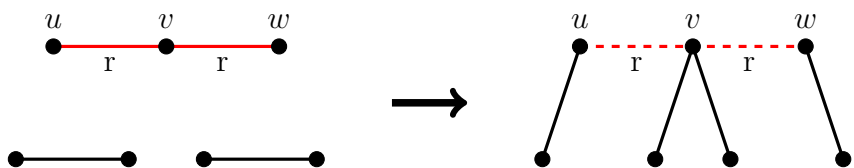
\begin{figure}[ht!]
\begin{center}
\begin{tikzpicture}[scale=1.0]
\draw [-,very thick,red] (1,1.5) -- (4,1.5);
\node [below] at (1.75,1.5) {r}; %% NEW
\node [below] at (3.25,1.5) {r}; %% NEW
\draw [-,very thick] (0.5,0) -- (2,0);
\draw [-,very thick] (3,0) -- (4.5,0);
\draw [fill] (0.5,0) circle (0.1);
\draw [fill] (2,0) circle (0.1);
\draw [fill] (3,0) circle (0.1);
\draw [fill] (4.5,0) circle (0.1);
\draw [fill] (1,1.5) circle (0.1);
\draw [fill] (2.5,1.5) circle (0.1);
\draw [fill] (4,1.5) circle (0.1);
\node [above] at (1,1.6) {$u$};
\node [above] at (2.5,1.6) {$v$};
\node [above] at (4,1.6) {$w$};
\draw [->,line width = 1mm] (5.5,0.75) -- (6.5,0.75);
\begin{scope}[shift={(7,0)}]
\draw [-,very thick,red,dashed] (1,1.5) -- (4,1.5);
\node [below] at (1.75,1.5) {r}; %% NEW
\node [below] at (3.25,1.5) {r}; %% NEW
\draw [-,very thick] (0.5,0) -- (1,1.5);
\draw [-,very thick] (2,0) -- (2.5,1.5) -- (3,0);
\draw [-,very thick] (4.5,0) -- (4,1.5);
\draw [fill] (0.5,0) circle (0.1);
\draw [fill] (2,0) circle (0.1);
\draw [fill] (3,0) circle (0.1);
\draw [fill] (4.5,0) circle (0.1);
\draw [fill] (1,1.5) circle (0.1);
\draw [fill] (2.5,1.5) circle (0.1);
\draw [fill] (4,1.5) circle (0.1);
\node [above] at (1,1.6) {$u$};
\node [above] at (2.5,1.6) {$v$};
\node [above] at (4,1.6) {$w$};
\end{scope}
\end{tikzpicture}
\caption{A switching for red 2-paths}
\label{f:2-path}
\end{center}
\end{figure}

Such a switching switches a graph in ${\mathcal W}$ to ${\mathcal W'}$. It is easy to see that the number of forward switchings is at least $(dn)^2(1-O((d+\Delta)/n))=\Omega((dn)^2)$, whereas the number of inverse switchings is at most $d^4$. Hence,
\[
\frac{|{\mathcal W}|}{|{\mathcal W'}|}=O(d^4/(dn)^2)=O(d^2/n^2).
\]
Moreover, $|{\mathcal W'}|\le |\strata_{i-2}|$ as ${\mathcal W'}\subseteq \strata_{i-2}$ and $|\strata_{i-2}|/|\strata_i|=O(i^2/(d\Delta)^2)$ by Lemma~\ref{lem:sizeRatio}.
Thus,
\[
\Prob(uvw\in G)=\frac{|{\mathcal W}|}{|\strata_i|}=
  \frac{|{\mathcal W}|}{|{\mathcal W'}|} \cdot
  \frac{|{\mathcal W'}|}{|\strata_{i-2}|} \cdot
  \frac{|\strata_{i-2}|}{|\strata_i|}
  =O\left(\frac{d^2}{n^2}\frac{i^2}{(d\Delta)^2}\right)=O(i^2/\Delta^2n^2).
\]
The total number of red 2-paths in $K_n$ is $O(\Delta^2 n)$. Thus, by linearity of expectation, the expected number of red 2-paths contained in a random $G\in\strata_i$ is 
\[ O(\Delta^2 n)\cdot O(i^2/\Delta^2 n^2)=O(i^2/n).\]

 The proof for the second claim is similar. 
Fix a red 3-path $v_1u_1u_2v_2$ in $K_n$. Let $G$ be a uniformly random graph in $\strata_i$. Using another switching (see Figure~\ref{f:3-path}) and a similar argument as above, we can bound the probability that $u_1v_1$ and $u_2v_2$ are edges in $G$ by $O(i^2/\Delta^2 n^2)$.   The red dotted edge marked ``?'' denotes a red edge
which may be either present or absent in $G$.
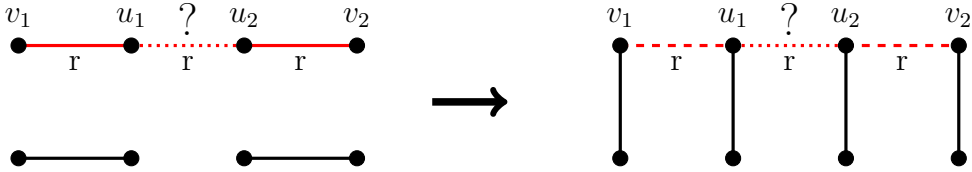
\begin{figure}[ht!]
\begin{center}
\begin{tikzpicture}[scale=1.0]
\draw [-,very thick,red] (1,1.5) -- (2.5,1.5); % NEW
\node [below] at (1.75,1.5) {r}; %% NEW
\node [below] at (4.75,1.5) {r}; %% NEW
\node [below] at (3.25,1.5) {r}; %% NEW
\draw [-,very thick,dotted,red] (2.5,1.5) -- (4.0,1.5);
\node [above] at (3.25,1.5) {\Large{?}};
\draw [-,very thick,red] (4,1.5) -- (5.5,1.5);
\draw [-,very thick] (1.0,0) -- (2.5,0);
\draw [-,very thick] (4,0) -- (5.5,0);
\draw [fill] (1.0,0) circle (0.1);
\draw [fill] (2.5,0) circle (0.1);
\draw [fill] (4,0) circle (0.1);
\draw [fill] (5.5,0) circle (0.1);
\draw [fill] (1,1.5) circle (0.1);
\draw [fill] (2.5,1.5) circle (0.1);
\draw [fill] (4,1.5) circle (0.1);
\draw [fill] (5.5,1.5) circle (0.1);
\node [above] at (1,1.6) {$v_1$};
\node [above] at (2.5,1.6) {$u_1$};
\node [above] at (4,1.6) {$u_2$};
\node [above] at (5.5,1.6) {$v_2$};
\draw [->,line width = 1mm] (6.5,0.75) -- (7.5,0.75);
\begin{scope}[shift={(8,0)}]
\draw [-,very thick,red,dashed] (1,1.5) -- (2.5,1.5) (4,1.5) -- (5.5,1.5); % NEW
\node [below] at (1.75,1.5) {r}; %% NEW
\node [below] at (4.75,1.5) {r}; %% NEW
\node [below] at (3.25,1.5) {r}; %% NEW
\draw [-,very thick,dotted,red] (2.5,1.5) -- (4.0,1.5);
\node [above] at (3.25,1.5) {\Large{?}};
\draw [-,very thick] (1.0,0) -- (1.0,1.5);
\draw [-,very thick] (2.5,0) -- (2.5,1.5);
\draw [-,very thick] (4.0,0) -- (4.0,1.5);
\draw [-,very thick] (5.5,0) -- (5.5,1.5);
\draw [fill] (1.0,0) circle (0.1);
\draw [fill] (2.5,0) circle (0.1);
\draw [fill] (4,0) circle (0.1);
\draw [fill] (5.5,0) circle (0.1);
\draw [fill] (1,1.5) circle (0.1);
\draw [fill] (2.5,1.5) circle (0.1);
\draw [fill] (4,1.5) circle (0.1);
\draw [fill] (5.5,1.5) circle (0.1);
\node [above] at (1,1.6) {$v_1$};
\node [above] at (2.5,1.6) {$u_1$};
\node [above] at (4,1.6) {$u_2$};
\node [above] at (5.5,1.6) {$v_2$};
\end{scope}
\end{tikzpicture}
\caption{A switching for red edges joined by a red edge or red non-edge}
\label{f:3-path}
\end{center}
\end{figure}

The total number of choices for $u_1,u_2,v_1,v_2$ in $K_n$ is $O(\Delta^3 n)$. Thus, by linearity of expectation, the expected number of pairs of red edges $\{u_1,u_2\}$ and $\{v_1,v_2\}$ as in the lemma is $O(\Delta^3 n)\cdot O(i^2/\Delta^2 n^2) = O(i^2\Delta/n)$.
 \end{proof}

Define
\begin{align}
\LB_A(i) =(\Delta n - 2i) d^2 (dn)^2\left(1 - \frac{30}{n}\right) - 3 d^5 \Delta n^2-8i\Delta d^3 n^2-3 \Delta^2 d^4 n^2.\label{LBA}
\end{align}

%The following lemma is proved in the appendix.  

\begin{lemma} \label{lem:bb}
For any $G\in\strata_i$ we have $b_A(G)\geq \LB_A(i)$.
Furthermore, if $G$ is chosen uniformly at random from $\strata_i$ then
\[ \ex b_A(G)=  \LB_A(i)(1+O((d+\Delta)^2/n^2+1/n)).\]
\end{lemma}

\begin{proof}
Firstly, note that all Class A switchings are of Type I.
(See Table~\ref{tab:type-class}.)
Thus we only need to count inverse Type I Class A switchings.

Consider the right hand side of Figure~\ref{fig:I-switch}. 
First we find a lower bound for the number of ways to select 
an 8-tuple $(v_0,\ldots, v_7)$ such that $v_0v_1$ is a red
non-edge, $v_1v_2$, $v_3v_4$, $v_5v_6$ are all edges and $v_2v_3$, $v_4v_5$, $v_6v_7$
are non-edges, with all vertices distinct except possibly $v_2$ and $v_7$.
There are $(\Delta n-2i)$ choices for $(v_0, v_1)$, then $d^2$
choices for $(v_2,v_7)$, then $dn-4$ choices for $v_2v_3$ avoiding the two chosen edges, and then $dn-6$ choices for $v_4v_5$ avoiding the three chosen edges.  
This gives the expression $(\Delta n - 2i) d^2(dn-4)(dn-6)$.

For the lower bound, we must subtract from this expression the
number of choices of 8-tuple with at least one defect.
The possible defects are: vertex coincidence; 
a dashed edge (other than $v_0v_1$) is present in $G$; 
a dashed edge (other than $v_0v_1$) is a red non-edge; 
or a chosen edge is a red edge in $G$.  
We now give upper bounds on the number of choices with particular defects.
\begin{itemize}
\item Vertex coincidences:
Out of $\binom{8}{2}=28$ possible vertex coincidences, 
one is allowed and 7 are impossible, leaving 20 vertex coincidences
that must be explicitly ruled out: at most $O(d^4 \Delta n^2)$ choices.
\item
One of the dashed edges 
(other than $v_0v_1$) is present in $G$:\  at most 
\[ 3d^2(d-1)^2 (\Delta n - 2i)(dn-8) \leq 3d^5 \Delta n^2\] choices.
\item $v_1v_2$ is red, or $v_0v_7$ is red:\ at most $2\cdot 2i\Delta d (dn)^2$ choices.
For later use, we remark that this upper bound includes cases where the edge $v_0v_1$ is red
and present in $G$.
\item $v_3v_4$ is red, or $v_5v_6$ is red:\ at most 
$2\cdot 2i(\Delta n-2i)d^2 (dn) = 2\cdot 2i \Delta  d^3 n^2$ choices;
\item $v_2v_3$ is a red non-edge, or $v_6v_7$ is a red non-edge:\ at most $2\cdot  \Delta^2 d^4 n^2$ choices;
\item $v_4v_5$ is a red non-edge:\ at most $(\Delta n-2i)^2 d^4
\leq \Delta^2 d^4 n^2$ choices.
\end{itemize}
Subtracting these choices leads to the inequality
\begin{align*}
 b_A(G) &\geq (\Delta n - 2i) d^4 n^2 \left(1 - \frac{30}{n}\right) - 3 d^5 \Delta n^2-8i\Delta d^3 n^2-
   3 \Delta^2 d^4 n^2,
\end{align*}
proving the first statement of the lemma.

To prove the second statement, we must investigate the average value of 
$b_A(G)$ over all $G\in\strata_{i}$.  
We continue inclusion-exclusion, calculating the number (or,
in two cases, the expected number) of 8-tuples containing two defects.
\begin{itemize}
\item
Two or more dashed edges (other than $v_0v_1$) are present: at most
$O(d^6\, \Delta n)$ choices, giving a relative error of $O(d^2/n^2)$.
\item One of the dashed edges (other than $v_0v_1$)
is a red non-edge and one of the chosen edges is red: a
t most $O(i\Delta^2 d^3 n)$
such choices, giving a relative error of $O(\Delta^2/n^2)$, since $i\leq i_1$.
\item
One of the dashed edges (other than $v_0v_1$) is present in $G$,
and one of the chosen edges is red: at most $O(i \Delta d^4 n)$ such choices,
giving a relative error of $O(d\Delta/n^2)$.
\item 
Two of the chosen edges are red: at most
$O(i^2 d^2 \Delta n)$ for all $G\in \strata_i$, giving a relative error of
$O(i^2/(d^2 n^2)) = O(\Delta^2/n^2)$, 
unless the two chosen edges are $v_1v_2$ and $v_0v_7$.  
The number of choices such that
$v_1v_2$ and $v_0v_7$ are red, with $v_0v_1$ a red non-edge, can vary a 
lot across 
$\strata_i$, and here we will need to calculate the average.  
(We come back to this, below.)
\item
One dashed edge is a red non-edge and another dashed edge is present in $G$
(neither edge is $v_0v_1$): at most $O(\Delta^2 d^5 n)$ such
choices, giving a relative error of $O(d\Delta/n^2)$.  
\item
Two dashed edges (other than $v_0v_1$)
are both red non-edges: at most $O(d^4\Delta^3 n)$ choices, 
giving a relative error of $O(\Delta^2/n^2)$.
\end{itemize}

There are two cases that must be considered further.  
\begin{itemize}
\item[$\ast$] Recall that in the lower 
bound, we subtracted some ``illegal'' cases where $v_0v_1$ is red and present.  
For the average-case expression we must add these cases back in. The term 
$4i\Delta d(dn)^2$, which we subtracted to obtain the lower bound, was an upper 
bound for the number of 8-tuples in which $v_1v_2$ or $v_7v_0$ is red. 
To obtain this bound, we first choose a red edge $v_0v_7$, say, in $2i$ ways, 
and then are at most $\Delta$ choices for $v_1$.  This upper bound of $\Delta$ 
includes the possibility that the edge $v_0v_1$ is present in $G$.  But such 
choices are not valid inverse Type I switchings, and so we must undo this 
subtraction by adding them back in now.
The number of choices of $(v_7,v_0,v_1,v_2)$ such that $v_0v_7$ and $v_0v_1$ are 
red edges in $G$, and $v_1v_2$ is an edge in $G$ (of any colour), varies quite 
widely among different $G\in\strata_i$.
By Lemma~\ref{lem:average}, the expected number of choices for this 4-tuple for 
a uniformly random $G\in \strata_i$ is $O(i^2\, d/n)$.  
\item[$\ast$]
The second thing we must consider is the
choices for the 8-tuple switching in which $v_0v_7$ and $v_1v_2$ are both red 
edges in $G$.
By Lemma~\ref{lem:average}, the expected number of 4-tuples $(v_7,v_0,v_1,v_2)$ with this property
(and with $v_0v_1$ a red non-edge in $G$) is $O(i^2 \Delta/n)$. 
\end{itemize}
Adding these counts together and multiplying by $(dn)^2$,
the expected number
of choices of $(v_0,v_1,\ldots, v_7)$ which must be added to the lower bound
\[ (dn)^2\, O((d + \Delta) i^2/n) = O(i^2 d^2(d+\Delta)n),\]
leading to a relative error of $O((d+\Delta)\Delta/n^2)$.

This completes the proof of the second statement of the lemma.
\end{proof}

\subsubsection{\bf Classes \textBiipm}
\bigskip

Classes \textBiipm\ are easy to handle, so we discuss them before
Classes \textBipm.
Define 
\begin{equation}
\LB_{\alpha}(i)=(\Delta n-2i)\Delta d^4n - 8i\Delta^2 d^3 n -2i \Delta d^4 n -
  2 \Delta^2 d^5 n -12\Delta^2d^4n 
  \quad\mbox{for $\alpha\in\{\Biipm\}$.} \label{LBB2} 
\end{equation}

\begin{lemma}\label{lem:bB2} 
For any $G\in {\strata}_i$ and for \emph{$\alpha\in\{\Biipm\}$},
\[
\LB_{\alpha}(i)\le b_{\alpha}(G) \le (\Delta n-2i) \Delta d^4n,
\]
and thus
\[
b_{\alpha}(G)=\Delta^2d^4n^2\left(1+O\left(\frac{d+\Delta}{n}\right)\right).
\]
\end{lemma}

\begin{proof}
Observe that all Class~\textBiipm\ switchings are of Type~I. 
(See Table~\ref{tab:type-class}.)  Thus we only need to count inverse
Type~I Class~\Bii\ switchings, say, and the same bounds will hold for
Class \textBvi, by symmetry.

There are $\Delta n-2i$ ways to choose $v_0$ and $v_1$. Then $d^2$ ways to fix $v_7$ and $v_2$. Then there are at most $\Delta d$ ways to choose $v_3$ and $v_4$ and finally at most $dn$ ways to choose $v_5$ and $v_6$. So the total number of inverse switchings is at most $(\Delta n-2i)\Delta d^4n$, giving the upper bound as desired. To deduce a lower bound, we subtract the number of the following structures, for which we only need an upper bound:
\begin{itemize}
\item $v_0v_7$ or $v_1v_2$ is red:\ at most $2\cdot 2i \Delta^2 d^2 \cdot dn=4i\Delta^2 d^3 n$ choices.
\item $v_2v_3$ is red and present:\ at most $2i\Delta d^3 \cdot dn=2i\Delta d^4n$ choices.
\item $v_3v_4$ is red:\ at most $2i \Delta^2 d^2 (dn)=2i\Delta^2d^3n$ choices.
\item $v_5v_6$ is red:\  at most $2i (\Delta n-2i)\Delta d^3\le 2i\Delta^2d^3n$ choices.
\item $v_6v_7$ or $v_4v_5$ is present:\  at most $2\cdot (\Delta n-2i)\Delta d^5\le 2\Delta^2d^5n$ choices. 
\item vertex coincidence, other than $v_2=v_7$:\ at most $12\cdot (\Delta n-2i) d^4 \Delta \le 12\Delta^2 d^4n$.
\end{itemize}
This immediately gives the required lower bound on the number of available
inverse Class \Bii\ switchings, completing the proof.
\end{proof}

\bigskip

\subsubsection{Classes \textBipm}\label{ss:boostB1}

The inverse switching of Type~I Class~\textBipm\ is indeed the same as the forward switching, up to a permutation of the labelling of the vertices involved in the switching. Recall the example discussed in Remark~\ref{bad-example}, where the $d$-factor is composed of a union of a $d$-regular graph with only red edges and a $d$-regular graph with only black edges. It is easy to see that in such a graph, the number of inverse Type~I Class~\textBipm\ switchings is zero. 
In general, the number of the following structure in $G$ can vary a lot among 
$G\in{\strata}_i$:
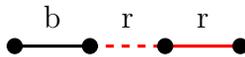
\begin{center}
\begin{tikzpicture}
\draw [very thick,-] (0,0) -- (1,0); 
\draw [very thick,-,red] (2,0) -- (3,0);  \draw [very thick,dashed,-,red] (1,0) -- (2,0); %% RED
\node [above] at (0.5,0.1) {b};  
\node [above] at (1.5,0.1) {r};  
\node [above] at (2.5,0.1) {r};  
\draw [fill] (0,0) circle (0.1); \draw [fill] (1,0) circle (0.1);
\draw [fill] (2,0) circle (0.1); \draw [fill] (3,0) circle (0.1);
\end{tikzpicture}
\end{center}
However, we do know that the sum of the number of the following structures in any 
$d$-factor $G\in {\strata}_i$ is between $2i (\Delta-1)(d-1)$ and $2i(\Delta-1) d$:
\begin{center}
\begin{tikzpicture}
\draw [very thick,-] (0,0) -- (1,0);
\draw [very thick,-,red] (2,0) -- (3,0); \draw [very thick,dashed,-,red] (1,0) -- (2,0); %% RED
\node [above] at (0.5,0.1) {b};  
\node [above] at (1.5,0.1) {r};  
\node [above] at (2.5,0.1) {r};  
\node [below] at (3.5,0.3) {+};  
\draw [fill] (0,0) circle (0.1); \draw [fill] (1,0) circle (0.1);
\draw [fill] (2,0) circle (0.1); \draw [fill] (3,0) circle (0.1);
%%%%%%%%%%%%
\begin{scope}[shift={(4,0)}]
\draw [very thick,-] (0,0) -- (1,0); 
\draw [very thick,-,red] (2,0) -- (3,0); \draw [very thick,-,red] (1,0) -- (2,0); %% RED
\node [above] at (0.5,0.1) {b};  
\node [above] at (1.5,0.1) {r};  
\node [above] at (2.5,0.1) {r};  
\node [below] at (3.5,0.3) {+};  
\draw [fill] (0,0) circle (0.1); \draw [fill] (1,0) circle (0.1);
\draw [fill] (2,0) circle (0.1); \draw [fill] (3,0) circle (0.1);
\end{scope}
%%%%%%%%%%%%
\begin{scope}[shift={(8,0)}]
\draw [very thick,-,red] (0,0) -- (1,0); 
\draw [very thick,-,red] (2,0) -- (3,0); \draw [very thick,dashed,-,red] (1,0) -- (2,0); %% RED
\node [above] at (0.5,0.1) {r};  
\node [above] at (1.5,0.1) {r};  
\node [above] at (2.5,0.1) {r};  
\node [below] at (3.5,0.3) {+};  
\draw [fill] (0,0) circle (0.1); \draw [fill] (1,0) circle (0.1);
\draw [fill] (2,0) circle (0.1); \draw [fill] (3,0) circle (0.1);
\end{scope}
%%%%%%%%%%%%
\begin{scope}[shift={(12,0)}]
\draw [very thick,-,red] (0,0) -- (1,0); 
\draw [very thick,-,red] (2,0) -- (3,0); \draw [very thick,-,red] (1,0) -- (2,0); %% RED
\node [above] at (0.5,0.1) {r};  
\node [above] at (1.5,0.1) {r};  
\node [above] at (2.5,0.1) {r};  
\draw [fill] (0,0) circle (0.1); \draw [fill] (1,0) circle (0.1);
\draw [fill] (2,0) circle (0.1); \draw [fill] (3,0) circle (0.1);
\end{scope}
%%%%%%%%%%%%
\end{tikzpicture}
\end{center}

This motivates the introduction of switchings of other types than Type I. 
We display these new switchings in Table~\ref{tab:TypeIV}. Here, the colours 
of the edges and non-edges must be black unless specified as red. 
These new types of switchings are categorised into Class \Bi\ or \textBvii,
under the rule that if the type ends with ``+'' then the class also ends with ``+'',
and similarly for ``$-$''. Note that due to symmetry, some switchings of different 
types have the same definition. For instance, type IIa$+$ and type IIa$-$ switchings 
are defined in the same way. However, they are introduced as booster switchings for 
different classes, and thus are categorised into different types.
Again, if the class name ends with a ``+'' then the vertices are labelled
as shown on the left of Figure~\ref{fig:vertex-labels}, while if the class name ends 
with a ``$-$'' then the vertices are labelled as shown on the right of 
Figure~\ref{fig:vertex-labels}.

We will show that for any $G\in{\strata}_i$, the number of inverse Class~\Bi\ 
(or \textBvii) switchings does not vary much, even though the number can be zero if restricted to inverse Type I Class \Bi\ (respectively, Class~\textBvii) switchings only.

%%%%%%%%%%%%%%%%%%%%%%%%%%%%%%%%%%%%%%%%%%%%%%%%%%%%%%%%%%
%%  BIG TABLE!!!
%%%%%%%%%%%%%%%%%%%%%%%%%%%%%%%%%%%%%%%%%%%%%%%%%%%%%%%%%%
\begin{table}[ht!]
\begin{center}
\renewcommand{\arraystretch}{1.2}
\begin{tabular}{|c|c|c|}
\hline
type, class &  action &  the switching \\
\hline & & \\
%%  Type IVb, Class B1 (Type IVd, Class B7)
\textIIapm,\, \textBipm  & $\mathcal{S}_{i-1} \rightarrow \mathcal{S}_i$ &
\begin{tikzpicture}[scale=0.7]
\node [right] at (1.3,2.2) {r};  % v_1v_2
\node [left] at (-0.3,2.2) {r};  % v_7v_0
\draw [-,very thick,red] (0,2.414) -- (1,2.414); %% RED
\node [above] at (0.5, 2.514) {r};
\draw [-,very thick] (1.707,1.707) -- (1.707,0.707);
\draw [-,very thick] (-0.707,1.707) -- (-0.707,0.707);
\draw [-,very thick] (0,0) -- (1,0);
\draw [-,very thick,dashed] (0,0) -- (-0.707,0.707);
\draw [-,very thick,dashed,red] (-0.707,1.707) -- (0,2.414); %% RED
\draw [-,very thick,dashed,red] (1,2.414) -- (1.707,1.707); %% RED
\draw [-,very thick,dashed] (1.707,0.707) -- (1,0);
\draw [-,very thick,dashed] (1.707,0.707) -- (1,0);
\draw [fill] (0,0) circle (0.1); \draw [fill] (1,0) circle (0.1);
\draw [fill] (1.707,0.707) circle (0.1); \draw [fill] (1.707,1.707) circle (0.1);
\draw [fill] (1,2.414) circle (0.1); \draw [fill] (0,2.414) circle (0.1);
\draw [fill] (-0.707,1.707) circle (0.1); \draw [fill] (-0.707,0.707) circle (0.1);
\draw [->,line width = 1mm] (3,0.866) -- (4,0.866);
\begin{scope}[shift={(6,0)}]
\draw [-,very thick,dashed,red] (0,2.414) -- (1,2.414); %% RED
\node [above] at (0.5, 2.514) {r};
\node [right] at (1.3,2.2) {r};  % v_1v_2
\node [left] at (-0.3,2.2) {r};  % v_7v_0
\draw [-,very thick,dashed] (1.707,1.707) -- (1.707,0.707);
\draw [-,very thick,dashed] (-0.707,1.707) -- (-0.707,0.707);
\draw [-,very thick,dashed] (0,0) -- (1,0);
\draw [-,very thick] (0,0) -- (-0.707,0.707);
\draw [-,very thick,red] (-0.707,1.707) -- (0,2.414); %% RED
\draw [-,very thick,red] (1,2.414) -- (1.707,1.707); %% RED
\draw [-,very thick] (1.707,0.707) -- (1,0);
\draw [fill] (0,0) circle (0.1); \draw [fill] (1,0) circle (0.1);
\draw [fill] (1.707,0.707) circle (0.1); \draw [fill] (1.707,1.707) circle (0.1);
\draw [fill] (1,2.414) circle (0.1); \draw [fill] (0,2.414) circle (0.1);
\draw [fill] (-0.707,1.707) circle (0.1); \draw [fill] (-0.707,0.707) circle (0.1);
\end{scope}
\end{tikzpicture}\\
\hline & & \\
%%%%%%%% Type IVa, Class B1 and mirror
\textIIbpm,\, \textBipm  & $\mathcal{S}_{i-2} \rightarrow \mathcal{S}_i$ &
\begin{tikzpicture}[scale=0.7]
\draw [-,very thick,dashed,red] (0,2.414) -- (1,2.414); %% RED
\node [below] at (0.5, 2.314) {r};
\draw [-,very thick,dashed] (1.707,1.707) -- (1.707,0.707); 
\node [right] at (1.3,2.2) {r};  % v_1v_2
\draw [-,very thick,dashed] (-0.707,1.707) -- (-0.707,0.707);
\draw [-,very thick,dashed] (0,0) -- (1,0);
\draw [-,very thick] (0,0) -- (-0.707,0.707);
\draw [-,very thick] (-0.707,1.707) -- (0,2.414);
\draw [-,dashed,very thick,red] (1,2.414) -- (1.707,1.707); %% RED
\draw [-,very thick] (1.707,0.707) -- (1,0);
\draw [fill] (0,0) circle (0.1); \draw [fill] (1,0) circle (0.1);
\draw [fill] (1.707,0.707) circle (0.1); \draw [fill] (1.707,1.707) circle (0.1);
\draw [fill] (1,2.414) circle (0.1); \draw [fill] (0,2.414) circle (0.1);
\draw [fill] (-0.707,1.707) circle (0.1); \draw [fill] (-0.707,0.707) circle (0.1);
\draw [-,very thick] (0,2.414) -- (0,3.214);
\draw [-,very thick,dashed] (0,3.214) -- (0,4.014);
\draw [-,very thick] (0,4.014) -- (1,4.014);
\draw [-,very thick,dashed] (1,3.214) -- (1,4.014);
\draw [-,very thick] (1,2.414) -- (1,3.214);
\draw [fill] (0,3.214) circle (0.1); \draw [fill] (1,3.214) circle (0.1);
\draw [fill] (0,4.014) circle (0.1); \draw [fill] (1,4.014) circle (0.1);
\draw [-,very thick] (1,2.414) -- (1.707,3.121);
\draw [-,very thick,dashed] (2.414,3.808) -- (1.707,3.121);
\draw [-,very thick] (2.414,2.414) -- (1.707,1.707);
\draw [-,very thick,dashed] (2.414,2.414) -- (3.121,3.121);
\draw [-,very thick] (2.414,3.808) -- (3.121,3.121);
\draw [fill] (2.414,3.808) circle (0.1); \draw [fill] (1.707,3.121) circle (0.1);
\draw [fill] (2.414,2.414) circle (0.1); \draw [fill] (3.121,3.121) circle (0.1);
\draw [->,line width = 1mm] (3,0.866) -- (4,0.866);
\begin{scope}[shift={(6,0)}]
\draw [-,very thick,red] (0,2.414) -- (1,2.414); %% RED
\node [below] at (0.5, 2.314) {r};
\draw [-,very thick,dashed] (1.707,1.707) -- (1.707,0.707);
\node [right] at (1.3,2.2) {r};  % v_1v_2
\draw [-,very thick,dashed] (-0.707,1.707) -- (-0.707,0.707);
\draw [-,very thick,dashed] (0,0) -- (1,0);
\draw [-,very thick] (0,0) -- (-0.707,0.707);
\draw [-,very thick] (-0.707,1.707) -- (0,2.414);
\draw [-,very thick,red] (1,2.414) -- (1.707,1.707); %% RED
\draw [-,very thick] (1.707,0.707) -- (1,0);
\draw [fill] (0,0) circle (0.1); \draw [fill] (1,0) circle (0.1);
\draw [fill] (1.707,0.707) circle (0.1); \draw [fill] (1.707,1.707) circle (0.1);
\draw [fill] (1,2.414) circle (0.1); \draw [fill] (0,2.414) circle (0.1);
\draw [fill] (-0.707,1.707) circle (0.1); \draw [fill] (-0.707,0.707) circle (0.1);
\draw [-,very thick,dashed] (0,2.414) -- (0,3.214);
\draw [-,very thick] (0,3.214) -- (0,4.014);
\draw [-,very thick,dashed] (0,4.014) -- (1,4.014);
\draw [-,very thick] (1,3.214) -- (1,4.014);
\draw [-,very thick,dashed] (1,2.414) -- (1,3.214);
\draw [fill] (0,3.214) circle (0.1); \draw [fill] (1,3.214) circle (0.1);
\draw [fill] (0,4.014) circle (0.1); \draw [fill] (1,4.014) circle (0.1);
\draw [-,very thick,dashed] (1,2.414) -- (1.707,3.121);
\draw [-,very thick] (2.414,3.808) -- (1.707,3.121);
\draw [-,very thick,dashed] (2.414,2.414) -- (1.707,1.707);
\draw [-,very thick] (2.414,2.414) -- (3.121,3.121);
\draw [-,very thick,dashed] (2.414,3.808) -- (3.121,3.121);
\draw [fill] (2.414,3.808) circle (0.1); \draw [fill] (1.707,3.121) circle (0.1);
\draw [fill] (2.414,2.414) circle (0.1); \draw [fill] (3.121,3.121) circle (0.1);
\end{scope}
\end{tikzpicture}\\
%%%%
\hline & & \\
\textIIcpm,\, \textBipm & $\mathcal{S}_{i-3} \rightarrow \mathcal{S}_i$ &
\begin{tikzpicture}[scale=0.7]
\draw [-,very thick,dashed,red] (0,2.414) -- (1,2.414); %% RED
\node [below] at (0.5, 2.314) {r};
\draw [-,very thick,dashed] (1.707,1.707) -- (1.707,0.707);
\node [right] at (1.3,2.2) {r};  % v_1v_2
\node [left] at (-0.3,2.2) {r};  % v_7v_0
\draw [-,very thick,dashed] (-0.707,1.707) -- (-0.707,0.707);
\draw [-,very thick,dashed] (0,0) -- (1,0);
\draw [-,very thick] (0,0) -- (-0.707,0.707);
\draw [-,very thick,dashed,red] (-0.707,1.707) -- (0,2.414); %% RED
\draw [-,dashed,very thick,red] (1,2.414) -- (1.707,1.707); %% RED
\draw [-,very thick] (1.707,0.707) -- (1,0);
\draw [fill] (0,0) circle (0.1); \draw [fill] (1,0) circle (0.1);
\draw [fill] (1.707,0.707) circle (0.1); \draw [fill] (1.707,1.707) circle (0.1);
\draw [fill] (1,2.414) circle (0.1); \draw [fill] (0,2.414) circle (0.1);
\draw [fill] (-0.707,1.707) circle (0.1); \draw [fill] (-0.707,0.707) circle (0.1);
\draw [-,very thick] (0,2.414) -- (0,3.214);
\draw [-,very thick,dashed] (0,3.214) -- (0,4.014);
\draw [-,very thick] (0,4.014) -- (1,4.014);
\draw [-,very thick,dashed] (1,3.214) -- (1,4.014);
\draw [-,very thick] (1,2.414) -- (1,3.214);
\draw [fill] (0,3.214) circle (0.1); \draw [fill] (1,3.214) circle (0.1);
\draw [fill] (0,4.014) circle (0.1); \draw [fill] (1,4.014) circle (0.1);
\draw [-,very thick] (1,2.414) -- (1.707,3.121);
\draw [-,very thick,dashed] (2.414,3.808) -- (1.707,3.121);
\draw [-,very thick] (2.414,2.414) -- (1.707,1.707);
\draw [-,very thick,dashed] (2.414,2.414) -- (3.121,3.121);
\draw [-,very thick] (2.414,3.808) -- (3.121,3.121);
\draw [fill] (2.414,3.808) circle (0.1); \draw [fill] (1.707,3.121) circle (0.1);
\draw [fill] (2.414,2.414) circle (0.1); \draw [fill] (3.121,3.121) circle (0.1);
\draw [-,very thick] (0,2.414) -- (-0.707,3.121);
\draw [-,very thick,dashed] (-1.414,3.808) -- (-0.707,3.121);
\draw [-,very thick] (-1.414,2.414) -- (-0.707,1.707);
\draw [-,very thick,dashed] (-1.414,2.414) -- (-2.121,3.121);
\draw [-,very thick] (-1.414,3.808) -- (-2.121,3.121);
\draw [fill] (-1.414,3.808) circle (0.1); \draw [fill] (-0.707,3.121) circle (0.1);
\draw [fill] (-1.414,2.414) circle (0.1); \draw [fill] (-2.121,3.121) circle (0.1);
\draw [->,line width = 1mm] (3,0.866) -- (4,0.866);
\begin{scope}[shift={(6,0)}]
\draw [-,very thick,red] (0,2.414) -- (1,2.414);  %% RED
\node [below] at (0.5, 2.314) {r};
\draw [-,very thick,dashed] (1.707,1.707) -- (1.707,0.707);
\node [right] at (1.3,2.2) {r};  % v_1v_2
\node [left] at (-0.3,2.2) {r};  % v_7v_0
\draw [-,very thick,dashed] (-0.707,1.707) -- (-0.707,0.707);
\draw [-,very thick,dashed] (0,0) -- (1,0);
\draw [-,very thick] (0,0) -- (-0.707,0.707);
\draw [-,very thick,red] (-0.707,1.707) -- (0,2.414); %% RED
\draw [-,very thick,red] (1,2.414) -- (1.707,1.707); %% RED
\draw [-,very thick] (1.707,0.707) -- (1,0);
\draw [fill] (0,0) circle (0.1); \draw [fill] (1,0) circle (0.1);
\draw [fill] (1.707,0.707) circle (0.1); \draw [fill] (1.707,1.707) circle (0.1);
\draw [fill] (1,2.414) circle (0.1); \draw [fill] (0,2.414) circle (0.1);
\draw [fill] (-0.707,1.707) circle (0.1); \draw [fill] (-0.707,0.707) circle (0.1);
\draw [-,very thick,dashed] (0,2.414) -- (0,3.214);
\draw [-,very thick] (0,3.214) -- (0,4.014);
\draw [-,very thick,dashed] (0,4.014) -- (1,4.014);
\draw [-,very thick] (1,3.214) -- (1,4.014);
\draw [-,very thick,dashed] (1,2.414) -- (1,3.214);
\draw [fill] (0,3.214) circle (0.1); \draw [fill] (1,3.214) circle (0.1);
\draw [fill] (0,4.014) circle (0.1); \draw [fill] (1,4.014) circle (0.1);
\draw [-,very thick,dashed] (1,2.414) -- (1.707,3.121);
\draw [-,very thick] (2.414,3.808) -- (1.707,3.121);
\draw [-,very thick,dashed] (2.414,2.414) -- (1.707,1.707);
\draw [-,very thick] (2.414,2.414) -- (3.121,3.121);
\draw [-,very thick,dashed] (2.414,3.808) -- (3.121,3.121);
\draw [fill] (2.414,3.808) circle (0.1); \draw [fill] (1.707,3.121) circle (0.1);
\draw [fill] (2.414,2.414) circle (0.1); \draw [fill] (3.121,3.121) circle (0.1);
\draw [-,very thick,dashed] (0,2.414) -- (-0.707,3.121);
\draw [-,very thick] (-1.414,3.808) -- (-0.707,3.121);
\draw [-,very thick,dashed] (-1.414,2.414) -- (-0.707,1.707);
\draw [-,very thick] (-1.414,2.414) -- (-2.121,3.121);
\draw [-,very thick,dashed] (-1.414,3.808) -- (-2.121,3.121);
\draw [fill] (-1.414,3.808) circle (0.1); \draw [fill] (-0.707,3.121) circle (0.1);
\draw [fill] (-1.414,2.414) circle (0.1); \draw [fill] (-2.121,3.121) circle (0.1);
\end{scope}
\end{tikzpicture}\\
\hline
\end{tabular}
\caption{The booster switchings for classes \textBipm}
\label{tab:TypeIV}
\end{center}
\end{table}
%%%%%%%%%%%%%%%%%%%%%%%%%%%%%%%%%%%%%%%%%%%%%%%%%
%%  End of the second biggish table
%%%%%%%%%%%%%%%%%%%%%%%%%%%%%%%%%%%%%%%%%%%%%%%%%

As shown in Table~\ref{tab:TypeIV}, Type~\textIIapm\ switchings are
described by an 8-tuple $\boldsymbol{v}=(v_0,\ldots, v_7)$, while Type~\textIIbpm\
switchings are described by 8-tuple $(v_0,\ldots, v_7)$ together with 8 additional
vertices, and Type~\textIIcpm\ switchings are described by an 8-tuple
$(v_0,\ldots, v_7)$ together with 12 additional vertices, providing the
additional edges used to perform the switching.
We denote the sequence of these additional vertices 
by $\boldsymbol{y}$, where the vertices are arranged in some prescribed
order: see Figure~\ref{whynot} for Type~\IVc.
An inverse Type~\textIIbpm\ switching is described by choosing the
8-tuple $\boldsymbol{v}$ and an 8-tuple $\boldsymbol{y}$ of additional
vertices,
while an inverse Type~\textIIcpm\ switching is described by choosing the
8-tuple $\boldsymbol{v}$ and a 12-tuple $\boldsymbol{y}$ of additional
vertices.

Suppose that a Type~\textIIbpm\ or Type~\textIIcpm\ switching based on the
8-tuple $\boldsymbol{v}$ creates a graph $G'$.
We refer to the subgraph of $G'$ formed by vertices in $\boldsymbol{v}$
as an \emph{octagon}.   
If an 8-tuple $\boldsymbol{v}=(v_0,\ldots,v_7)$ in $G$ can be combined 
with an $8$-tuple (respectively, $12$-tuple) 
of additional vertices $\boldsymbol{y}$
on which an inverse Type IIb$\pm$, 
(respectively, inverse Type IIc$\pm$) switching can be performed, then
we call $\boldsymbol{v}$ an octagon of Type IIb$\pm$ 
(respectively, Type IIc$\pm$). 
The switching operation is denoted by $(G,\boldsymbol{v},\boldsymbol{y})\mapsto G'$.
Note that octagons of different types induce different subgraph structures and 
(non-)edge colour restrictions.
Each octagon which can result from a Type~\textIIbpm\ 
(respectively, Type~\textIIcpm)
switching is not created equally often, due to the varying number of ways
to select the additional vertices needed to perform
the inverse switching.
Thus we introduce another sort of rejection, called \emph{pre-b-rejection}, 
to equalise the frequency of the creation of each octagon, given a switching 
type $\tau\in\{\IIbpm,\, \IIcpm\}$.

Given $G\in \strata_i$, $\tau\in \{ \IIbpm,\, \IIcpm\}$ 
and an octagon induced by the 8-tuple $\boldsymbol{v}$,
let ${\widehat b_{\tau}}(G,\boldsymbol{v})$ be the number of ways to choose 
the sequence of additional vertices $\boldsymbol{y}$ (the length of
$\boldsymbol{y}$ depends on $\tau$) so that an inverse Type~$\tau$ switching 
can be performed using $\boldsymbol{v}$ and $\boldsymbol{y}$.
Define
\begin{eqnarray}
{\widehat \LB}_{\IIbpm}(i)&=&(dn-2i-12)^4-6(dn)^3d^2-6(dn)^3\Delta d;\\
{\widehat \LB}_{\IIcpm}(i)&=&(dn-2i-14)^6-9(dn)^5d^2-9(dn)^5\Delta d.
\end{eqnarray}

%We will prove the following lemma in the appendix.

\begin{lemma}\label{lem:IVa}
Let \emph{$\tau\in\{\IVa,\, \IVc\}$}  and $G\in \strata_i$.
For any octagon $\boldsymbol{v} = (v_0,\ldots,v_7)$ in $G$ that can be 
created by a type $\tau$ switching,
\[
{\widehat \LB}_{\tau}(i) \le {\widehat b_{\tau}}(G, \boldsymbol{v}) ={\widehat \LB}_{\tau}(i) (1+O((d+\Delta)/n)).
\]
\end{lemma}

\begin{proof}
We only prove the result for $\tau=\IVa$, as the argument for $\tau=\IVc$ is similar
and by symmetry, the same bounds will hold for $\tau=\IVd$ and $\tau=\IVf$, 
respectively.

Let $\boldsymbol{v}=(v_0,\ldots, v_7)$ be a fixed 8-tuple which gives
rise to the octagon shown in Figure~\ref{whynot}.
We bound the number of ways to choose an 8-tuple 
$\boldsymbol{y} = (y_1,y_2,y_3,y_4,y_5,y_6,y_7,y_8)$ of
additional vertices
so that dashed lines in Figure~\ref{whynot} correspond to black non-edges in $G$.
\begin{figure}[ht!]
\begin{center}
\begin{tikzpicture}[scale=0.9]
\draw [-,very thick,red] (0,2.414) -- (1,2.414); %% RED
\node [above] at (0.5, 2.414) {r};
\draw [-,very thick,dashed] (1.707,1.707) -- (1.707,0.707);
\node [above] at (1.5,2.0) {r};  % v_1v_2
\draw [-,very thick,dashed] (-0.707,1.707) -- (-0.707,0.707);
\draw [-,very thick,dashed] (0,0) -- (1,0);
\draw [-,very thick] (0,0) -- (-0.707,0.707);
\draw [-,very thick] (-0.707,1.707) -- (0,2.414);
\draw [-,very thick,red] (1,2.414) -- (1.707,1.707); %% RED
\draw [-,very thick] (1.707,0.707) -- (1,0);
\draw [fill] (0,0) circle (0.1); \draw [fill] (1,0) circle (0.1);
\draw [fill] (1.707,0.707) circle (0.1); \draw [fill] (1.707,1.707) circle (0.1);
\draw [fill] (1,2.414) circle (0.1); \draw [fill] (0,2.414) circle (0.1);
\draw [fill] (-0.707,1.707) circle (0.1); \draw [fill] (-0.707,0.707) circle (0.1);
\draw [-,very thick,dashed] (0,2.414) -- (0,3.214);
\draw [-,very thick] (0,3.214) -- (0,4.014);
\draw [-,very thick,dashed] (0,4.014) -- (1,4.014);
\draw [-,very thick] (1,3.214) -- (1,4.014);
\draw [-,very thick,dashed] (1,2.414) -- (1,3.214);
\draw [fill] (0,3.214) circle (0.1); \draw [fill] (1,3.214) circle (0.1);
\draw [fill] (0,4.014) circle (0.1); \draw [fill] (1,4.014) circle (0.1);
\draw [-,very thick,dashed] (1,2.414) -- (1.707,3.121);
\draw [-,very thick] (2.414,3.808) -- (1.707,3.121);
\draw [-,very thick,dashed] (2.414,2.414) -- (1.707,1.707);
\draw [-,very thick] (2.414,2.414) -- (3.121,3.121);
\draw [-,very thick,dashed] (2.414,3.808) -- (3.121,3.121);
\draw [fill] (2.414,3.808) circle (0.1); \draw [fill] (1.707,3.121) circle (0.1);
\draw [fill] (2.414,2.414) circle (0.1); \draw [fill] (3.121,3.121) circle (0.1);
%%%
\node [left] at (0.0,2.514) {$v_0$};
\node [below] at (1,2.314) {$v_1$};
\node [right] at (1.807,1.507) {$v_2$};
\node [right] at (1.807,0.507) {$v_3$};
\node [below] at (1,-0.1) {$v_4$};
\node [below] at (0,-0.1) {$v_5$};
\node [left] at (-0.807,0.707) {$v_6$};
\node [left] at (-0.807,1.707) {$v_7$};
%%%%
\node [left] at (0.0,3.314) {$y_1$};
\node [left] at (0.0,4.114) {$y_2$};
\node [left] at (1.0,3.314) {$y_3$};
\node [right] at (1.0,4.114) {$y_4$};
\node [right] at (1.7,2.814) {$y_5$};
\node [left] at (2.5,4.014) {$y_6$};
\node [right] at (2.3,2.114) {$y_7$};
\node [right] at (3.1,3.114) {$y_8$};
\end{tikzpicture}
\caption{Choosing the additional vertices for an inverse Type~\IVc\ switching }
\label{whynot}
\end{center}
\end{figure}
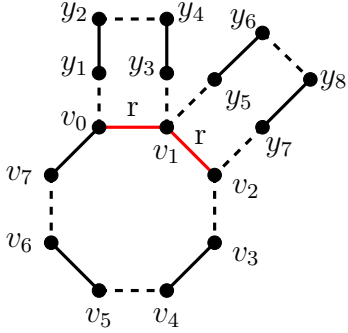

The upper bound $(dn)^4$ is obvious. For the lower bound, first notice that this 
number is at least $(dn-2i - 12)^4$, as the 4 extra edges involved in the 
inverse 
switching are black, and are distinct from each other and from the 
3 black edges in the octagon
Further, we need to 
subtract the number of choices where at least one defect appears. 
There are at most $6\cdot (dn)^3d^2$ choices where
one designated non-edge (such as $v_0y_1$ or $y_2y_4$)
is actually present,  and at most
$6\cdot (dn)^3 d\Delta$ choices where one designated non-edge is red. 
Subtracting these counts gives the required lower bound.  
\end{proof}

We will specify $\UB_{\tau}(i)$ for $\tau\in\{\IIapm,\,\IIbpm,\IIcpm\}$ 
in~(\ref{UBIIa})--(\ref{UBIIc}). It is trivial to see that 
$f_{\tau}(G)\le \UB_{\tau}(i)$  for each such $\tau$ and for $G\in \strata_i$. 
These types of switchings are performed so rarely that the trivial lower bound 
$f_{\tau}(G)\ge 0$ is sufficient for our analysis: see the proof of Lemma~\ref{lem:rejections}.

\bigskip

\noindent {\bf Pre-b-rejection} \label{Pre-state}

When we count the number of inverse Class~\textBipm\ switchings applicable to 
$G\in\strata_i$, we count the number of choices of $(v_0,\ldots,v_7)$ that are 
allowed to be created by a Class~\textBipm\ switching. However, some types of 
switchings create structures with more vertices than those in an octagon. 
For instance, let $\boldsymbol{v}$ be an octagon of type IIb+ and let
$\boldsymbol{x}=(x_1,\ldots,x_4)$ be the four extra edges that are created by a 
Type \IVa\ switching.  We can consider $(G, \boldsymbol{v}, \boldsymbol{x})$ as
a {\em pre-state} of $(G,\boldsymbol{v})$. Each octagon $\boldsymbol{v}$ in $G$ corresponds to 
exactly $ {\widehat b_{\tau}}(G, \boldsymbol{v})$ pre-states, and 
${\widehat b_{\tau}}(G, \boldsymbol{v}) \approx {\widehat \LB}_{\tau}(i)$ by 
Lemma~\ref{lem:IVa}. 
By carefully designing the pre-b-rejection scheme, we can ensure that each 
octagon $\boldsymbol{v}$ in $G$ is created equally often if each of 
its pre-states are created equally often.

When a Type~$\tau$ switching converting $G$ to $G'$ is chosen, 
corresponding to a valid 8-tuple $\boldsymbol{v}$, we will 
reject the algorithm and restart with probability
\[
1-\frac{{\widehat \LB}_{\tau}(i)}{{\widehat b_{\tau}}(G, \boldsymbol{v})}.
\]
This restart will be called a \emph{pre-b-rejection}.

The pre-b-rejection is incorporated in the formal definition of the algorithm in Section~\ref{sec:alg}. We close this section by bounding $b_{\alpha}(G)$ for 
$\alpha\in\{\Bipm\}$.

Define
\begin{equation}
\LB_{\alpha}(i)=2i(\Delta-1)(d-1) (dn-2i-10d)^2 - 6i(\Delta-1) d^3n(d+\Delta) \quad \mbox{for $\alpha\in\{\Bipm\}$.} \label{LBB1}
\end{equation}

\begin{lemma}\label{lem:LBB1}
For any $G\in\strata_i$, and \emph{$\alpha\in\{\Bipm\}$},
\[
\LB_{\alpha}(i)\le b_{\alpha}(G)= \LB_{\alpha}(i)\left(1+O\left(\frac{1}{d}+\frac{d+\Delta}{n}\right)\right).
\]
\end{lemma}

\begin{proof}
The number of ways to choose $v_7$, $v_0$, $v_1$ and $v_2$ is between $2i(\Delta-1)(d-1)$ and $2i(\Delta-1)d$. The number of ways to choose the other four vertices is at least $(dn-2i-10d)^2$ so that there are no vertex coincidence and both $v_3v_4$ and $v_5v_6$ are black. We subtract the choices where $v_2v_3$ or $v_4v_5$ or $v_6v_7$ is present in $G$, or is a red non-edge. There are at most $3\cdot 2i(\Delta-1)d (d^2 (dn) + \Delta d(dn))$ such choices. This verifies the desired lower bound. The upper bound $2i(\Delta-1)d(dn)^2$ is trivial, which yields the required relative error.
\end{proof}

\subsubsection{Classes \textCpm}

Now consider Class \textCpm. As we will show later, the probability that a Type I switching is in Class \textCpm\ is very small. Thus, we only need a rather rough lower bound on the number of inverse Class \textCpm\ switchings, so that the probability of a b-rejection is not too close to~1. However, there are very rare graphs in $\strata_i$ that cannot be created by a Type I Class \textCpm\ switching. For instance, this
may occur if $d=\Delta$ and the set of all red edges in $G$ form a red $d$-regular subgraph. Then $G$ does not contain the following structure, 
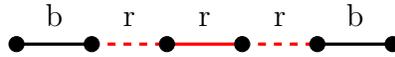
\begin{center}
\begin{tikzpicture}
\draw [very thick,-,dashed,red] (0,0) -- (1,0);  %% RED
\draw [very thick,-,dashed,red] (2,0) -- (3,0); \draw [very thick,-,red] (1,0) -- (2,0); %% RED
\draw [very thick,-] (-1,0) -- (0,0); \draw [very thick,-] (3,0) -- (4,0);
\node [above] at (0.5,0.1) {r};  
\node [above] at (1.5,0.1) {r};  
\node [above] at (2.5,0.1) {r};  
\node [above] at (-0.5,0.1) {b};  
\node [above] at (3.5,0.1) {b};  
\draw [fill] (0,0) circle (0.1); \draw [fill] (1,0) circle (0.1);
\draw [fill] (2,0) circle (0.1); \draw [fill] (3,0) circle (0.1);
\draw [fill] (-1,0) circle (0.1); \draw [fill] (4,0) circle (0.1);
\end{tikzpicture}
\end{center}
and thus cannot be created by a Type~I switching. In this case, that the probability of a b-rejection 
would equal~1 due to the existence of such graphs.
In order to reduce the probability of a b-rejection, we introduce a new 
type of switching, namely Type \Va\ for Class C+ and Type \textVb\ 
for Class C$-$,  that boost the probability of graphs which
contain the following structure:
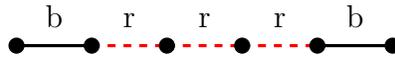
\begin{center}
\begin{tikzpicture}
\draw [very thick,-,dashed,red] (0,0) -- (1,0); \draw [very thick,-,dashed,red] (2,0) -- (3,0); %% RED
\draw [very thick,-,dashed,red] (1,0) -- (2,0); %% RED
\draw [very thick,-] (-1,0) -- (0,0); \draw [very thick,-] (3,0) -- (4,0);
\node [above] at (0.5,0.1) {r};  
\node [above] at (1.5,0.1) {r};  
\node [above] at (2.5,0.1) {r};  
\node [above] at (-0.5,0.1) {b};  
\node [above] at (3.5,0.1) {b};  
\draw [fill] (0,0) circle (0.1); \draw [fill] (1,0) circle (0.1);
\draw [fill] (2,0) circle (0.1); \draw [fill] (3,0) circle (0.1);
\draw [fill] (-1,0) circle (0.1); \draw [fill] (4,0) circle (0.1);
\end{tikzpicture}
\end{center}
It turns out that for any $G\in\strata_i$, the number of choices of 6-tuples of vertices $(x_1,\ldots,x_6)$ such that $x_1x_2$ and $x_5x_6$ are black edges in $G$, $x_2x_3$ and $x_4x_5$ are red non-edges, and $x_3x_4$ is either a red edge or a red non-edge, is always sufficiently concentrated. See Lemma~\ref{lem:LBC} for a precise bound. This is why we boost the second structure, 
to transform a highly varying count into a well-concentrated count.

The Type~\textIIIpm, Class~\textCpm\ 
switchings are shown in Table~\ref{tab:TypeV}.
\begin{table}[ht!]
\begin{center}
\renewcommand{\arraystretch}{1.2}
\begin{tabular}{|r|r|c|}
\hline
type, class &  action &  the switching \\
\hline
%%%%%%%% Type IVa, Class B1
%%%%%
%%  Type IVb, Class B1 (Type IVd, Class B7)
\textIIIpm,\, \textCpm  & $\mathcal{S}_{i} \rightarrow \mathcal{S}_i$ &
\begin{tikzpicture}[scale=0.7]
\node [right] at (1.3,2.2) {r};  % v_1v_2
\node [left] at (-0.3,2.2) {r};  % v_7v_0
\draw [-,very thick,dashed,red] (0,2.414) -- (1,2.414); %% RED
\node [above] at (0.5, 2.514) {r};
\draw [-,very thick] (1.707,1.707) -- (1.707,0.707);
\draw [-,very thick] (-0.707,1.707) -- (-0.707,0.707);
\draw [-,very thick] (0,0) -- (1,0);
\draw [-,very thick,dashed] (0,0) -- (-0.707,0.707);
\draw [-,very thick,dashed,red] (-0.707,1.707) -- (0,2.414); %% RED
\draw [-,very thick,dashed,red] (1,2.414) -- (1.707,1.707); %% RED
\draw [-,very thick,dashed] (1.707,0.707) -- (1,0);
\draw [-,very thick,dashed] (1.707,0.707) -- (1,0);
\draw [fill] (0,0) circle (0.1); \draw [fill] (1,0) circle (0.1);
\draw [fill] (1.707,0.707) circle (0.1); \draw [fill] (1.707,1.707) circle (0.1);
\draw [fill] (1,2.414) circle (0.1); \draw [fill] (0,2.414) circle (0.1);
\draw [fill] (-0.707,1.707) circle (0.1); \draw [fill] (-0.707,0.707) circle (0.1);
\draw [->,line width = 1mm] (3,0.866) -- (4,0.866);
\begin{scope}[shift={(6,0)}]
\draw [-,very thick,dashed,red] (0,2.414) -- (1,2.414); %% RED
\node [above] at (0.5, 2.514) {r};
\node [right] at (1.3,2.2) {r};  % v_1v_2
\node [left] at (-0.3,2.2) {r};  % v_7v_0
\draw [-,very thick] (1.707,1.707) -- (1.707,0.707);
\draw [-,very thick] (-0.707,1.707) -- (-0.707,0.707);
\draw [-,very thick] (0,0) -- (1,0);
\draw [-,very thick,dashed] (0,0) -- (-0.707,0.707);
\draw [-,very thick,dashed,red] (-0.707,1.707) -- (0,2.414); %% RED
\draw [-,very thick,dashed,red] (1,2.414) -- (1.707,1.707); %% RED
\draw [-,very thick,dashed] (1.707,0.707) -- (1,0);
\draw [fill] (0,0) circle (0.1); \draw [fill] (1,0) circle (0.1);
\draw [fill] (1.707,0.707) circle (0.1); \draw [fill] (1.707,1.707) circle (0.1);
\draw [fill] (1,2.414) circle (0.1); \draw [fill] (0,2.414) circle (0.1);
\draw [fill] (-0.707,1.707) circle (0.1); \draw [fill] (-0.707,0.707) circle (0.1);
\end{scope}
\end{tikzpicture}\\
%%%%
\hline
\end{tabular}
\caption{The booster switchings for Classes \textCpm}
\label{tab:TypeV}
\end{center}
\end{table}

Unusually, the Type~\textIIIpm\ switchings do not perform any switch of edges,
 except for designating an 8-tuple of vertices satisfying certain constraints, 
as shown in Table~\ref{tab:TypeV}.
They can be viewed as adding a small ``do nothing'' probability
to the algorithm. As we will see later, the probability of ever
performing a Class~\textCpm\ switching is extremely small.

As before, although type III$+$ and III$-$ switchings have the same definition, they are booster switchings for classes C$+$ and C$-$ respectively, and thus have to be categorised into different types.
 %--- types III$+$ and III$-$ here.

Define
\begin{align}
\qquad \LB_{\alpha}(i)&=d^3\Delta^3 n^2 (1- 8(d+\Delta)/n) 
    & 
\text{for $\alpha\in\{\Cpm\}$},\qquad  \label{LBC}\\
\qquad \UB_{\tau}(i)&=\Delta^3 d^3 n^2  & \text{for $\tau\in\{\IIIpm\}$}.
\qquad
\end{align}

\begin{lemma}\label{lem:LBC}
For each $G\in \strata_i$ and for \emph{$\tau\in \{ \IIIpm\}$},
\[
\UB_{\tau}(i)(1-8(d+\Delta)/n)\le f_{\tau}(G)\le \UB_{\tau}(i).
\]
For each $G\in\strata_i$ and for \emph{$\alpha\in\{\Cpm\}$},
\[
\LB_{\alpha}(i) \le b_{\alpha}(G) \le d^3\Delta^3 n^2.
\]
\end{lemma}

\begin{proof}
 We only discuss the case $\tau=\Va$, as the case $\tau=\Vb$ is symmetric. 
There are at most $\Delta n-2i\le \Delta n$ ways to choose $(v_0,v_1)$, 
and then at most $\Delta^2$ ways to choose $(v_2,v_7)$. 
Then there are at most $d^2$ ways to choose $(v_3,v_6)$. 
Finally, there are at most $dn-2i\le dn$ ways to choose $(v_4,v_5)$. 
This gives the required upper bound for $f_{\Va}(G)$. To obtain the lower bound, 
we need to subtract from these $d^2\Delta^2 (\Delta n-2i)(dn-2i)$ choices 
of $(v_0,\ldots, v_7)$ the following cases:
\begin{enumerate}
\item[(a)] $v_1 v_2$ or $v_0v_7$ is a red edge in $G$;
\item[(b)] $v_2v_3$ or $v_6v_7$ is a red edge in $G$;
\item[(c)] $v_3v_4$ or $v_5v_6$ is a edge in $G$.
\end{enumerate} 
The number of choices for (a) is at most $2\cdot 2i d^2 \Delta^2 dn=4id^3\Delta^2 n$.
To see this, there are at most $2i$ ways to fix $v_1$ and $v_2$ if $v_1v_2$ is a red
edge in $G$; then at most $d$ ways to fix $v_3$, 
at most $\Delta^2$ ways to fix $v_0$ and $v_7$, at most $d$ ways to fix $v_6$ 
and finally at most $dn$ ways to fix $v_4$ and $v_5$. The factor of $2$ 
covers the case that $v_2v_7$ is a red edge present in $G$.

The number of choices for (b) is at most $2\cdot 2i \Delta^3 d dn=4i d^2 \Delta^3 n$.

The number of choices for (c) is at most $2\cdot \Delta n \Delta^2 d^4=2\Delta^3 d^4 n$. 

\noindent Hence, we have
\[
f_{\Va}(G)\ge d^2\Delta^2 (\Delta n-2i)(dn-2i)-4id^3\Delta^2 n-4i d^2 \Delta^3 n-2\Delta^3 d^4 n \ge d^3\Delta^3 n^2(1-8(d+\Delta)/n),
\]
as required, since $i<d\Delta$ by~(\ref{def-imaxb}).

Next we bound $b_{\alpha}(G)$ for $\alpha=\Ci$, as the case $\alpha=\Cii$ is 
symmetric. To perform an inverse Class \Ci\ switching, we need to designate 
an 8-tuple $\boldsymbol{v}=(v_0,\ldots, v_7)$ such that either an inverse Type I Class \Ci\ switching can be performed on $\boldsymbol{v}$, or an inverse Type Va switching can be performed on $\boldsymbol{v}$. Note that an inverse type \Ci\ switching is just the same as a type \Ci\ switching. The lower bound on $f_{\Va}(G)$ naturally is a lower bound for $b_{\Ci}(G)$. So immediately we have
$b_{\Ci}(G)\ge \LB_{\Ci}(i)$ as specified in~(\ref{LBC}). It is not hard to see that $d^3\Delta^3n^2$ is an upper bound for $b_{\Ci}(G)$, because there are at most $\Delta n$ ways to fix $v_0$ and $v_1$ (either $v_0v_1$ is present or not present in $G$), and at most $\Delta^2d^2$ ways to fix $v_2,v_3,v_7,v_6$ and then at most $dn$ ways to fix $v_4$ and $v_5$.
\end{proof}

\subsection{The algorithm: \algB}\label{sec:alg}

Now we have defined all types and classes of switchings involved in 
\algB. Figure~\ref{fig:all-switchings} depicts all switchings which produce
an element of $\strata_i$, labelled by their type and class.

\begin{figure}
\begin{center}
\begin{tikzpicture}[scale=0.80]
\draw [very thick] (-0.5,1.5) circle (1.0 and 1.5);
\draw [very thick] (2.0,1.5) circle (1.0 and 1.5);
\draw [very thick] (6,2.0) circle (1.5 and 2.0);
\draw [very thick] (11.0,1.5) circle (1.0 and 1.5);
\draw [very thick] (13.5,1.5) circle (1.0 and 1.5);
\draw [very thick] (16.0,1.5) circle (1.0 and 1.5);
\draw [->, line width=3pt] (2.5,1.3) -- (5.2,1.3);
\node [above] at (3.75,1.3) {I,A};
\node [below] at (3.75,1.2) {I,$\Cpm$};
\draw [->, line width=3pt] (10.5,1.3) -- (6.8,1.3);
\node [above] at (8.75,1.3) {$\IIapm$,$\Bipm$};
\draw [->, line width=3pt] (13.5,2.3) arc (10:175:3.4 and 1.25);
\node [above] at (11.0,3.3) {$\IIbpm$,$\Bipm$};
\draw [<-, line width=3pt] (5.6,2.3) arc (5:170:2.8 and 1.25);
\node [above] at (2.5,3.6) {I,$\Biipm$};
\draw [->, line width=3pt] (16.0,2.5) arc (10:175:4.9 and 2.5);
\node [above] at (12.0,4.6) {$\IIcpm$,$\Bipm$};
\draw [<-, line width=3pt] (6.2,3.3) arc (-70:245:0.7 and 0.9);
\node [above] at (6.0,5.8) {I,$\Bipm$};
\node [above] at (6.0,5.1) {$\IIIpm$,$\Cpm$};
\node [below] at (-0.5,-0.1) {$\strata_{i+2}$};
\node [below] at (2.0,-0.1) {$\strata_{i+1}$};
\node [below] at (6,-0.1) {$\strata_{i}$};
\node [below] at (11.0,-0.1) {$\strata_{i-1}$};
\node [below] at (13.5,-0.1) {$\strata_{i-2}$};
\node [below] at (16.0,-0.1) {$\strata_{i-3}$};
\end{tikzpicture}
\caption{All switchings into $\strata_i$, labelled by type and class}
\label{fig:all-switchings}
\end{center}
\end{figure}
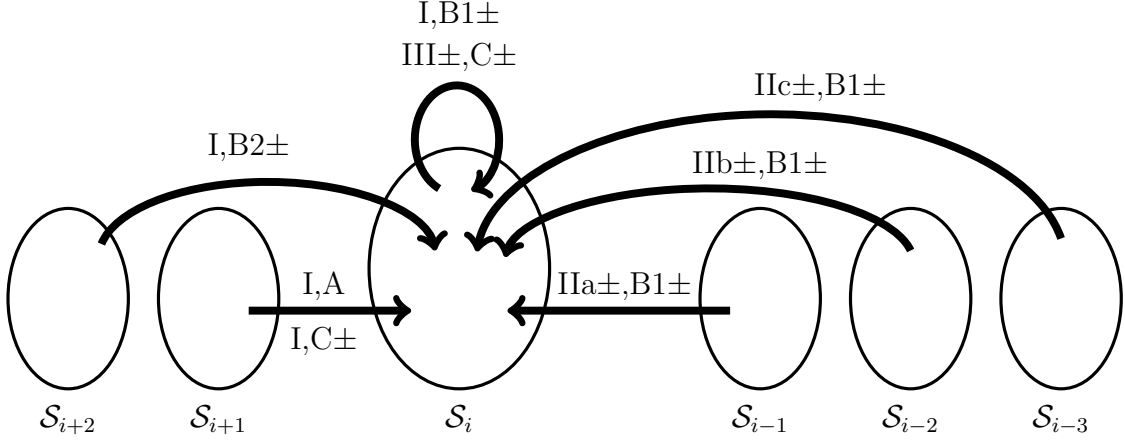

%%%%%%
%%%%%

We now describe the algorithm \algB\ formally.
First, \algB\ calls REG to generate a uniformly random 
$d$-regular graph $G_0$ on $\{ 1,2,\ldots, n\}$. If $G_0$ contains more than $\imax$ red 
edges then \algB\ restarts. Otherwise, \algB\ iteratively performs a sequence of 
switching steps. In each switching step, if the current graph $G$ is in $\strata_i$ then
\algB\ chooses a switching type $\tau$ from a set of types $\Gamma$, with 
probability $\rho_{\tau}(i)$. 
Here 
\[ \Gamma=\{\text{I},\,\, \IIapm,\,\,  \IIbpm, \,\, \IIcpm, \,\, \IIIpm\}\]
and we insist only that $\sum_{\tau\in\Gamma}\, \rho_{\tau}(i) \leq 1$ for
all $i\leq i_1$.  With probability $1-\sum_{\tau\in\Gamma}\, \rho_{\tau}(i)$
we perform a rejection called ``t-rejection'' instead of choosing a switching type. 
If no t-rejection is performed then \algB\ chooses a random Type~$\tau$ switching, 
and either performs this chosen switching, or restarts with a small probability 
(the sum of the f-rejection, pre-b-rejection and b-rejection probabilities).
The parameters $\rho_{\tau}(i)$ will be specified in the next section. 
If $i=0$ then a Type I switching is interpreted as outputting the current graph.

To be more specific, let $G_t$ be the graph obtained after $t$ switching
steps, and suppose that $G_t=G\in\strata_i$.
The $(t+1)$'th switching step is composed 
of the following substeps:
\begin{enumerate}
\item[(i)] Choose switching type $\tau\in\Gamma$ with probability $\rho_{\tau}(i)$. If no type is chosen, perform a t-rejection.
\item[(ii)] Assume that no t-rejection was performed. If $i=0$ and $\tau=I$ then output 
the current graph $G$.  Otherwise, choose a random Type~$\tau$ switching $S$ for $G$
and let $G'$ be the graph obtained from $G$ by performing $S$. 
Perform an f-rejection with probability $1-f_{\tau}(G)/\UB_{\tau}(i)$.
\item[(iii)] 
If no f-rejection is performed then perform a pre-b-rejection, if applicable. 
(See the description given above Lemma~\ref{lem:LBB1}.) 
\item[(iv)] If no f-rejection or pre-b-rejection is performed then 
let $\alpha$ be the class of $S$ and suppose that $G'\in\strata_{i'}$. 
Perform a b-rejection with probability 
$1-\LB_{\alpha}(i')/b_{\alpha}(G')$.
\item[(v)] If no b-rejection is performed then set $G_{t+1} = G'$. 
\end{enumerate} 
If any rejection occurs then \algB\ restarts.

Again, this is a Las Vegas algorithm with no deterministic
upper bound on the running time of the algorithm, due to the chance of rejections.
But we will show that the probability of a rejection occurring is small, under the 
assumptions of Theorem~\ref{thm:B}.

\subsection{Uniformity: fixing $\rho_{\tau}(i)$}\label{ss:equalise}

We complete the definition of \algB\ by specifying the parameters $\rho_{\tau}(i)$, 
for $\tau\in \Gamma$.  
Let $\sigma(G)$ denote the expected number of times that $G$ is 
reached by \algB. We will design $\rho_{\tau}(i)$ such that $\sigma(G)=\sigma_i$ for some $\sigma_i$, for every $G\in \strata_i$ and every $0\le i\le \imax$.
A method of designing these parameters is discussed in~\cite[Section 5]{GWSIAM} 
in a general setting. In the rest of this section, we carry out this method and apply 
it to our specific problem.

For convenience, we summarise which switching types occur for each class
in Table~\ref{tab:type-class}.
As usual, a type ending in ``+'' goes with a class ending in ``+'', and
similarly for those ending in ``$-$''.
%\begin{table}[ht!]
%\renewcommand{\arraystretch}{1.2}
%\begin{center}
%\begin{tabular}{|l|c|}
%\hline
%type & corresponding classes\\
%\hline
%I & A,\, \textBipm,\, \textBiipm,\, \textCpm\\
%\textIIapm,\, \textIIbpm,\, \textIIcpm & \textBipm \\
%\textIIIpm & \textCpm\\
%\hline
%\end{tabular}
%\caption{Reference table for types and classes.}
%\label{tab:type-class}
%\end{center}
%\end{table}
%
\begin{table}[ht!]
\renewcommand{\arraystretch}{1.2}
\begin{center}
\begin{tabular}{|l|c|}
\hline
Class & Types associated with the given class\\
\hline
A & I\\
\textBipm & I,\, \textIIapm,\, \textIIbpm,\, \textIIcpm \\
\textBiipm & I\\
\textCpm & I, \textIIIpm\\
\hline
\end{tabular}
\caption{Reference table for types and classes.}
\label{tab:type-class}
%\label{tab:class-type}
\end{center}
\end{table}

Below is a list of parameters $\UB_{\tau}(i)$, which are upper bounds on the 
number of ways to perform a switching of each type on a given $G\in\mathcal{S}_i$:
\begin{align}
 \overline{m}_{\text{I}}(i) &= 
  %\left(2i(dn)^3 - 8i(d-1)^2 (dn)^2-4i\Delta d^3n^2-4i^2(dn)^2\right)\left( 1 + 40((d+\Delta)^2/n^2 + 1/n)\right),\nonumber\\
 2i(dn)^3\left(1 + 28\left(\frac{(\Delta+d)^2}{n^2} + \frac{1}{n}\right)\right)  - 8i(d-1)^2d^2n^2 - 4i\Delta d^3 n^2 - 4i^2(dn)^2, \nonumber \\
\overline{m}_{\tau}(i) &= 2i \Delta^2 d^3 n, \hspace*{3.9cm}
    \tau\in\{  \IIapm\}, \label{UBIIa}\\
\overline{m}_{\tau}(i) &= \Delta^2 d^9 n^5, \hspace*{4cm}
    \tau\in\{ \IIbpm\},\label{UBIIb}\\
\overline{m}_{\tau}(i) &= \Delta^3 d^{11} n^6, \hspace*{3.9cm}
    \tau\in\{ \IIcpm\},\label{UBIIc}\\
\overline{m}_{\tau}(i) &= \Delta^3 d^3 n^2, \hspace*{4.0cm}
    \tau\in\{ \IIIpm\}.\nonumber
\end{align}

Next we list parameters $\LB_{\alpha}(i)$, which are lower bounds on the number of ways to perform switchings of
each class to produce a given $G\in\mathcal{S}_i$:
\begin{align*}
\underline{m}_A(i) &=
 (\Delta n - 2i) d^2 (dn)^2\left(1 - \frac{30}{n}\right) - 3 d^5 \Delta n^2-8i\Delta d^3 n^2-3 \Delta^2 d^4 n^2,\\
\underline{m}_{\alpha}(i) &= 2i(\Delta-1)(d-1) (dn-2i-10d)^2 - 6i(\Delta-1) d^3n(d+\Delta),\hspace*{15mm}
   \alpha\in \{  \Bipm\},\\
\underline{m}_{\alpha}(i)  &=
 (\Delta n-2i)\Delta d^4n - 8i\Delta^2 d^3 n -2i \Delta d^4 n -2 \Delta^2 d^5 n -12\Delta^2d^4n, \hspace*{14mm}
   \alpha\in \{ \Biipm\},\\
\underline{m}_{\alpha}(i) &=d^3\Delta^3 n^2 (1- 8(d+\Delta)/n),
 \hspace*{73mm}
  \alpha\in \{\Cpm\}.
\end{align*}

Fix a class $\alpha$ and let $\tau$ be a type such that class
$\alpha$ and type $\tau$ appear together in some row of 
Table~\ref{tab:type-class}.  For each relevant $i\leq \imax$, 
let $q_\alpha^\tau(i)$ denote the expected number of times that
an element of $\strata_i$ is reached by a Type~$\tau$, Class~$\alpha$ 
switching.  We will choose our parameters to ensure 
that the value of $q_\alpha^\tau(j)$ \emph{does not depend on $\tau$},
for any type $\tau$ associated with class $\alpha$.
This common value is denoted by $q_\alpha(i)$; that is,
$q_\alpha(i) = q_\alpha^\tau(i)$ for any type $\tau$ associated
with class $\alpha$.

It follows then that
\begin{align}\label{sigma_i}
\sigma_i = \frac{1}{|\accept|}+\sum_{\alpha} q_{\alpha}(i) \, \LB_{\alpha}(i), \quad \mbox{for every $0\le i\le \imax$.}
\end{align}
This equation holds because every $G\in \strata_i$ can be chosen as the
initial graph, if not initially rejected; or is reached via some 
switching. The probability that $G$ is the graph obtained at Step 0 is 
$1/|\accept|$, since $G_0$ is chosen uniformly, and by our design
of the algorithm, $q_{\alpha}(i) \, \LB_{\alpha}(i)$ is exactly the expected 
number of times that $G$ is reached via some class $\alpha$ switching and is not 
t-rejected, f-rejected, pre-b-rejected or b-rejected.

Immediately we have
\begin{align}
q_A(i) = \frac{\sigma_{i+1}\, \rho_{\text{I}}(i+1)}{\overline{m}_{\text{I}}(i+1)}.
\label{qA}
\end{align} 
This is because a graph $G\in\strata_i$ can be created via a 
Class A switching only via a Type I switching on a graph $G'\in\strata_{i+1}$. 
(See the first line of Table~\ref{tab:type-class}.)
Every graph in $\strata_{i+1}$ is visited $\sigma_{i+1}$ times in expectation, 
and given any $G'\in \strata_{i+1}$ such that $S=(G', G)$ is a valid Type~I Class~A 
switching, the probability that \algB\ chooses Type~I is $\rho_{\text{I}}(i+1)$, and 
the probability that \algB\ chooses the particular switching $S$ is 
$1/\UB_{\text{I}}(i+1)$. 

Next, consider $\alpha\in\{\Bipm\}$. A Class \textBipm\ switching can be of 
Type~I, \textIIapm, \textIIbpm or \textIIcpm\ (see the second line of
Table~\ref{tab:type-class}). 
Now $G\in \strata_i$ might be created from some
$G'\in\strata_i$ via a Type I Class \textBipm\ switching. Thus,
arguing as above, we have
\begin{equation}
q_{\Bi}(i) = q_{\Bvii}(i) =\frac{\sigma_i\, \rho_{\text{I}}(i)}{\overline{m}_{\text{I}}(i)}.\label{qB}
\end{equation}
To ensure that the expected number of times $G$ is visited via Type $\tau$
Class \textBipm\ switchings does not depend on $\tau$, for 
$\tau\in\{I,\IIapm,\IIbpm,\IIcpm\}$,
we must choose $\rho_{\tau}(\cdot)$ for $\tau\in\{\IIapm,\IIbpm,\IIcpm\}$ such that
\begin{align}
 \frac{\sigma_i\, \rho_{\text{I}}(i)}{\overline{m}_{\text{I}}(i)} =q_{\Bi}(i)
          &= \frac{\sigma_{i-2}\, \rho_{\IVa}(i-2)}{\overline{m}_{\IVa}(i-2)}\cdot {\widehat \LB_{\IVa}}(i) \label{eq:IVa}\\
          &= \frac{\sigma_{i-1}\, \rho_{\IVb}(i-1)}{\overline{m}_{\IVb}(i-1)}\\
          &= \frac{\sigma_{i-3}\, \rho_{\IVc}(i-3)}{\overline{m}_{\IVc}(i-3)}\cdot {\widehat \LB_{\IVc}(i)}, 
\end{align}
and 
\begin{align}
 \frac{\sigma_i\, \rho_{\text{I}}(i)}{\overline{m}_{\text{I}}(i)} =q_{\Bvii}(i)
          &= \frac{\sigma_{i-2}\, \rho_{\IVd}(i-2)}{\overline{m}_{\IVd}(i-2)}\cdot {\widehat \LB_{\IVd}}(i) \\
          &= \frac{\sigma_{i-1}\, \rho_{\IVe}(i-1)}{\overline{m}_{\IVe}(i-1)}\\
          &= \frac{\sigma_{i-3}\, \rho_{\IVf}(i-3)}{\overline{m}_{\IVf}(i-3)} \cdot {\widehat \LB_{\IVf}(i)}.\label{eq:IVf}
\end{align}

%\medskip

Next we consider $\alpha\in\{\Biipm\}$. A Class \textBiipm\ switching can only be of 
Type~I (see the third line of Table~\ref{tab:type-class}). Thus we immediately have 
\begin{equation}
 q_{\Bii}(i) = q_{\Bvi}(i) = \frac{\sigma_{i+2}\, \rho_{\text{I}}(i+2)}{\overline{m}_{\text{I}}(i+2)}.
\end{equation}

%\medskip

Finally, consider $\alpha\in\{\Cpm\}$. A Class \textCpm\ switching can be of 
Type~I or Type~\textIIIpm (see the fourth row of Table~\ref{tab:type-class}). 
Considering switchings of Type~I and Class \textCpm, we have 
\begin{equation}
q_{\Ci}(i)=q_{\Cii}(i)=\frac{\sigma_{i+1}\, \rho_{\text{I}}(i+1)}{\LB_{\text{I}}(i+1)}. 
\end{equation}
To ensure that the expected number of times $G$ is visited by
Type~\textIIIpm\ Class~\textCpm\ switchings, we must choose
$\rho_{\tau}(\cdot)$ for $\tau\in\{\IIIpm\}$ such that
\begin{align}
\frac{\sigma_{i+1}\, \rho_{\text{I}}(i+1)}{\LB_{\text{I}}(i+1)} =q_{\Ci}(i)
          &= \frac{\sigma_{i}\, \rho_{\Va}(i)}{\overline{m}_{\Va}(i)},\label{qC1}
        \end{align}
and 
\begin{align}
 \frac{\sigma_{i+1}\, \rho_{\text{I}}(i+1)}{\LB_{\text{I}}(i+1)} =q_{\Cii}(i)
          &= \frac{\sigma_{i}\, \rho_{\Vb}(i)}{\overline{m}_{\Vb}(i)}.\label{qC}
\end{align}
Combining~(\ref{qA})--(\ref{qC}), and using~(\ref{sigma_i}), we deduce that
\begin{equation}
\sigma_{i} = \frac{1}{|\accept|} + \sum_{\alpha} q_\alpha(i)\, \underline{m}_\alpha(i)\label{sigma}
\end{equation}
for $i=0,\ldots, i_1$, where $\alpha$ ranges over all possible classes 
$\{A,\, \Bipm,\Biipm,\, \Cpm\}$.
Using the change of variables $x_i = \sigma_i\, |\accept|$, for $i=0,\ldots, i_1$,
as in~\cite{GWSIAM}, we rewrite this as
\begin{equation}
x_i = 1 + \sum_{\alpha} q_\alpha(i)\, \underline{m}_\alpha(i)\, |\accept|. 
  \label{rec}
\end{equation}

\subsubsection{Boundary conditions and solving the system} 

Recall that \algB\ rejects the initial $d$-regular graph if it contains more than 
$\imax$ red edges. Thus, we set $\rho_{\tau}(i)=0$ for all $i>\imax$. 
The parameter $\rho_{\text{I}}(i)$ is already defined for all $0\le i\le \imax$, 
recalling 
that $\rho_{\text{I}}(0)$ is interpreted as the probability of outputting the 
current graph. The Type I switchings may be of Class \textBipm, \textBiipm\ or
\textCpm, as shown in Figure~\ref{fig:all-switchings}. 
A Type~I Class~\textBipm\ switching converts a graph from $\strata_i$ to 
$\strata_{i}$, and the booster switchings for Class \textBipm\ are of Type \textIIapm, \textIIbpm, \textIIcpm.  
By Table~\ref{tab:TypeIV}, we must define $\rho_{\tau}(i)$ for all
\[
\left\{
\begin{array}{ll}
0\le i\le \imax-1 & \mbox{if}\ \tau\in\{\IIapm\},\\
0\le i\le \imax-2 & \mbox{if}\ \tau\in\{\IIbpm\},\\
0\le i\le \imax-3 & \mbox{if}\ \tau\in\{\IIcpm\}.
\end{array}
\right.
\]
Similarly, consider the booster switchings for classes \textCpm. 
When $\tau\in \{\IIIpm\}$
we must define $\rho_{\tau}(i)$ for all $0\le i\le\imax$.

Note that in general there are switchings converting graphs in 
$\strata_j$ to graphs in $\strata_i$ for $i-3\le j\le i+2$, as shown in 
Figure~\ref{fig:all-switchings}. For $i$ close to $\imax$ or 0 there are fewer
switching types involved. For instance, graphs in $\strata_{\imax}$ cannot be reached by a Type~I Class~A switching, because the boundary conditions will be set so that no graphs in strata $j>\imax$ will ever be reached. 
Similarly, no graphs in $\strata_0$ can be reached by a Type~\textIIapm\ switching, because any such switching
increases the number of red edges in the graph, and this number can never be negative.

Let $\epsilon$ be a function of $n$, $d$ and $\Delta$ to be specified later. We first give a computation scheme that determines $\rho_{\text{I}}(i)$, $\rho_{\Va}(i)$, $\rho_{\Vb}(i)$ and $x_i$ for all $0\le i\le \imax$ such that
\begin{align}
\rho_{\text{I}}(\imax)=1-\epsilon;& \qquad \rho_{\text{III}}(\imax)=0; \label{cond1}\\
\rho_{\text{I}}(i)+\rho_{\Va}(i)+\rho_{\Vb}(i)=1-\epsilon &\quad\mbox{for all $0\le i\le\imax-1$}.\label{cond2}
\end{align}
By symmetry, let $\rho_{\text{III}}(i)=\rho_{\Va}(i)=\rho_{\Vb}(i)$. 
We determine $\rho_{\text{I}}(i)$, $\rho_{\text{III}}(i)$ and $x_i$ recursively for 
$i$ in descending order. 

Substituting
\begin{eqnarray}
q_A(i) &=& \frac{\sigma_{i+1}\, \rho_{\text{I}}(i+1)}{\overline{m}_{\text{I}}(i+1)};\\
q_{\Bi}(i)=q_{\Bvii}(i) &=& \frac{\sigma_i\, \rho_{\text{I}}(i)}{\overline{m}_{\text{I}}(i)};\\ 
 q_{\Bii}(i) = q_{\Bvi}(i) &=& \frac{\sigma_{i+2}\, \rho_{\text{I}}(i+2)}{\overline{m}_{\text{I}}(i+2)};\\
 q_{\Ci}(i)=q_{\Cii}(i)&=&\frac{\sigma_{i+1}\, \rho_{\text{I}}(i+1)}{\LB_{\text{I}}(i+1)}
\end{eqnarray}
into~(\ref{rec}), we have
\begin{align}
 x_i &= 1+ \frac{x_{i+1}\, \rho_{\text{I}}(i+1)}{\overline{m}_{\text{I}}(i+1)} \LB_A(i)+2\cdot\frac{x_i\, \rho_{\text{I}}(i)}{\overline{m}_{\text{I}}(i)} \LB_{\Bi}(i) \nonumber\\
  & \hspace*{4cm} {} +2\cdot\frac{x_{i+2}\, \rho_{\text{I}}(i+2)}{\overline{m}_{\text{I}}(i+2)} \LB_{\Bii}(i)+2\cdot \frac{x_{i+1}\, \rho_{\text{I}}(i+1)}{\UB_{\text{I}}(i+1)}\LB_{\Ci}(i).\label{rec2}
\end{align}
Here we used the fact that $\LB_{\Bi}(i)=\LB_{\Bvii}(i)$, $\LB_{\Bii}(i)=\LB_{\Bvii}(i)$ and $\LB_{\Ci}(i)=\LB_{\Cii}(i)$.
Similarly, by~(\ref{qC1}) and~(\ref{qC}), and using the fact that $\UB_{\Va}(i)=\UB_{\Vb}(i)$,
\begin{equation}
\rho_{\text{III}}(i)=\frac{x_{i+1}}{x_{i}} \,\frac{ \overline{m}_{\Va}(i)}{\UB_{\text{I}}(i+1)}\, \rho_{\text{I}}(i+1). \label{rhoV}
\end{equation}

\bigskip

\noindent {\em Base case} (a): $i=\imax$. We have set $\rho_{\text{I}}(\imax)=1-\epsilon$, $\rho_{\text{III}}(\imax)=0$ and $\rho_{\text{I}}(i)=0$ for all $i>\imax$. 
Hence, by (\ref{rec2}),
\[
\left(1-2\cdot\frac{\rho_{\text{I}}(\imax)}{\overline{m}_{\text{I}}(\imax)} \LB_{\Bi}(\imax)\right)x_{\imax} = 1
\]
which determines $x_{\imax}$.

\bigskip

\noindent {\em Base case} (b): $i=\imax-1$. 
We have
\begin{equation}
\rho_{\text{I}}(\imax-1)+2\rho_{\text{III}}(\imax-1)=1-\epsilon,\label{eq:b1}
\end{equation}
and
\begin{equation}
\left(1-\kappa_1\, \rho_{\text{I}}(\imax-1)\right)x_{\imax-1} = \kappa_2, \label{eq:b2}
\end{equation}
where
\[
\kappa_1= \frac{2\, \LB_{\Bi}(\imax-1)}{\overline{m}_{\text{I}}(\imax-1)},\qquad 
\kappa_2=1+ \frac{\rho_{\text{I}}(\imax)}{\UB_{\text{I}}(\imax)} \big(\LB_A(\imax-1)+2\, \LB_{\Ci}(\imax-1)\big)\, x_{\imax}.
\]
Hence, by (\ref{rhoV}), 
\begin{equation}
\rho_{\text{III}}(\imax-1) \, x_{\imax-1}=\kappa_3 \label{eq:b3}
\end{equation}
where
\[
\kappa_3=x_{\imax} \,\frac{ \overline{m}_{\Va}(\imax-1)\, \rho_{\text{I}}(\imax)}{\UB_{\text{I}}(\imax)}.
\]
Solving~(\ref{eq:b1}),~(\ref{eq:b2}) and~(\ref{eq:b3}) gives
\[
x_{\imax-1}=\frac{\kappa_2 -2\, \kappa_1\kappa_3}{1-\kappa_1(1-\epsilon)},
\]
and substituting this into~(\ref{eq:b3}) and then into~(\ref{eq:b1}) yields $\rho_{\text{III}}(\imax-1)$ and $\rho_{\text{I}}(\imax-1)$.

\bigskip

\noindent \emph{Inductive step.}\
Now assume that $i\le \imax-2$ and that $\rho_{\text{I}}(j)$, $\rho_{\text{III}}(j)$ and $x_j$ have been 
determined for all $j>i$. By~(\ref{rec2}) we have
\begin{equation}
\left(1-\kappa_1\rho_{\text{I}}(i)\right)x_{i} = \kappa_2, \label{eq:c2}
\end{equation}
where
\begin{eqnarray*}
\kappa_1&=&\frac{2\, \LB_{\Bi}(i)}{\UB_{\text{I}}(i)},\\
\kappa_2&=&1+\frac{\rho_{\text{I}}(i+1)}{\UB_{\text{I}}(i+1)} \big(\LB_A(i)+2\, \LB_{\Ci}(i)\big)\, x_{i+1}+
  2\cdot\frac{\rho_{\text{I}}(i+2) \,  \LB_{\Bii}(i)}{\UB_{\text{I}}(i+2)}\,x_{i+2}.
\end{eqnarray*}
We also have
\begin{eqnarray}
\rho_{\text{I}}(i)+2\rho_{\text{III}}(i)&=&1-\epsilon, \label{eq:sum}\\
\rho_{\text{III}}(i) x_{i}&=&\kappa_3,\quad \mbox{where}\ \kappa_3=x_{i+1} \,\frac{ \overline{m}_{\Va}(i)\, \rho_{\text{I}}(i+1)}{\UB_{\text{I}}(i+1)}.\label{eq:rhoV}
\end{eqnarray}
As in base case (b), we obtain 
\begin{equation}
x_{i}=\frac{\kappa_2 - 2\kappa_1\kappa_3}{1-\kappa_1(1-\epsilon)}, \label{xi}
\end{equation}
and immediately this gives $\rho_{\text{III}}(i)$ and $\rho_{\text{I}}(i)$. Thus we have uniquely determined $\rho_{\text{I}}(i)$, $\rho_{\text{III}}(i)$ and $x_i$ that satisfy~(\ref{cond1}) and~(\ref{cond2}).

Next we specify $\epsilon$. Define
\begin{equation}
\epsilon=5\left(\frac{\Delta+d}{n}\right)^2.
\label{epsilon}
\end{equation}

\begin{lemma}\label{lem:solution}
For $\epsilon$ as defined in \emph{(\ref{epsilon})}, 
equations~\emph{(\ref{qA})--(\ref{rec})} have a unique solution 
$(x^*_i,\rho^*_{\tau}(i):\tau\in \Gamma, i\geq 0)$  which
satisfies~\emph{(\ref{cond1})} 
and~\emph{(\ref{cond2})}.  Moreover,  for all $0\le i\le \imax$, 
\begin{align}
x^*_i> 0;&\qquad\qquad \frac{x^*_{i}}{x^*_{i-1}}=O(i/d\Delta) \,\,\, \text{ if \,\, $i\ge 1$}; \label{condx}\\
0\le \rho^*_{\tau}(i)\le 1\,\, \text{ for all \,\, $\tau\in\Gamma$};& 
\qquad\qquad \sum_{\tau\in\Gamma} \rho^*_{\tau}(i)\le 1.\label{condrho}
\end{align}
\end{lemma}

\bigskip

\begin{proof}
We have already shown that there are unique $x^*_i$, $\rho^*_{\text{I}}(i)$ and $\rho^*_{\Va}(i)$, $\rho^*_{\Vb}(i)$ satisfying~(\ref{cond1}) and~(\ref{cond2}). Substituting these values into~(\ref{eq:IVa})--(\ref{eq:IVf}) uniquely determines $\rho^*_{\tau}(i)$ for all $\tau\in\{\IIapm,\, \IIbpm,\, \IIcpm\}$
and all relevant values of $i$.

It is easy to verify that $x^*_{\imax}>0$ and $x^*_{\imax-1}>0$ and that~(\ref{condrho}) is satisfied for $i=\{\imax,\imax-1\}$. We will prove by induction on $i$, for all $0\le i\le \imax-1$,  that~(\ref{condrho}) is satisfied as well as the following strengthening condition of~(\ref{condx}).
\begin{equation}
x^*_{i+1}>0;\quad x^*_{i}>0,\quad \frac{x^*_{i+1}}{x^*_{i}}\le 2.3(i+1)/d\Delta. \label{condx2}
\end{equation}

 By~(\ref{rec2}), for every $i\le\imax-1$, provided that $x_{i+2}\ge 0$ and $\rho_{\text{I}}(i+2)\ge 0$,
we have
\begin{equation}
\frac{x_{i+1}}{x_{i}}\le \frac{1.1\, \UB_{\text{I}}(i+1)}{\LB_A(i)}\le 2.3(i+1)/d\Delta,\label{ratio}
\end{equation}
since
\[
x_{i}\left(1-2\cdot\frac{\rho_{\text{I}}(i)}{\overline{m}_{\text{I}}(i)} \LB_{\Bi}(i)\right)>\frac{x_{i+1}\, \rho_{\text{I}}(i+1)}{\overline{m}_{\text{I}}(i+1)} \Big( \LB_A(i)+2\LB_{\Ci}(i)\Big),
\]
and $\LB_{\Bi}(i)/\UB_{\text{I}}(i)=o(1)$ and $\LB_{\Ci}(i)=o(\LB_A(i))$. Hence~(\ref{condx2}) and~(\ref{condrho}) are satisfied for $i=\imax-1$. 

Now assume $i\le \imax-1$.  By induction, $x^*_{j}>0$  and $\rho^*_{\text{I}}(j)\ge 0$ for all $j\ge i+1$. 
Thus, $\kappa_2>0$ where $\kappa_2$ is as below~(\ref{eq:c2}). It is easy to see that $\kappa_1<1$  by the definition of $\LB_{\Bi}(i)$ and $\UB_{\text{I}}(i)$. Thus, $x^*_i>0$ by~(\ref{xi}). Since $\rho_{\text{I}}(i+2)\ge 0$ and $x^*_{i+2}\ge 0$ by induction,~(\ref{ratio}) follows and thus~(\ref{condx2}) holds for $i$.

Substituting~(\ref{xi}) to~(\ref{eq:rhoV}) we have
\begin{equation}
\rho_{\text{III}}(i) = \frac{x_{i+1}}{x_i} \frac{\UB_{\Va}(i)}{\UB_{\text{I}}(i+1)} \rho_{\text{I}}(i+1)\le \frac{2.3(i+1)}{d\Delta} \, \frac{\UB_{\Va}(i)}{\UB_{\text{I}}(i+1)}\le \frac{1.2\Delta^2}{dn}. \label{eq:IIIprob}
\end{equation}
as we have already established~(\ref{condx2}) for $i$, and $\rho_{\text{I}}(i+1)\le 1$ by induction. 
Substituting the bound to~(\ref{eq:sum}), we have $\rho_{\text{I}}(i)=1-\epsilon-2\rho_{\text{III}}(i)=1-o(1)\ge 0$.

Finally, by~(\ref{eq:IVa})--(\ref{eq:IVf}) and using (\ref{def-imaxb}), we have
\begin{eqnarray}
\rho_{\IVa}(i)&=&\rho_{\IVd}(i)=\frac{x_{i+2}}{x_i} \frac{\rho_{\text{I}}(i+2)\, \UB_{\IVa}(i)}{\UB_{\text{I}}(i+2)\, {\widehat \LB_{\IVa}}(i+2)}<4d\Delta/n^2,\label{eq:IIbprob}\\
\rho_{\IVb}(i)&=&\rho_{\IVe}(i)=\frac{x_{i+1}}{x_i} \frac{\rho_{\text{I}}(i+1)\, \UB_{\IVb}(i)}{\UB_{\text{I}}(i+1)}\le 2\Delta^2/n^2,\label{eq:IIaprob}\\
\rho_{\IVc}(i)&=&\rho_{\IVf}(i)=\frac{x_{i+3}}{x_i} \frac{\rho_{\text{I}}(i+3)\, \UB_{\IVc}(i)}{\UB_{\text{I}}(i+3)\, {\widehat \LB_{\IVc}}(i+3)}=O(\Delta^2 d/n^3)<
\nfrac{1}{2}\, \Delta^2/n^2. \label{eq:IIcprob}
\end{eqnarray}
With these bounds on $\rho_{\tau}(i)$ for $\tau\in\{\IIapm, \IIbpm, \IIcpm\}$, 
we can easily verify that 
\[ \sum_{\tau\in \{\IIapm, \IIbpm, \IIcpm\}} \, \rho^*_{\tau}(i)<\epsilon.\] 
This verifies~(\ref{condrho}), completing the proof.
\end{proof}

Now we define parameters $\rho_{\tau}(i)$ in \algB\ by 
$\rho_{\tau}(i) = \rho^*_{\tau}(i)$, where $\big(\rho^*_\tau(i)\big)$ 
is the unique solution
guaranteed by Lemma~\ref{lem:solution}. This completes the definition of \algB. Next we show that the output of \algB\ is correct.

\begin{lemma}\label{lem:uniform}
The output of \algB\ is a uniformly random $d$-factor of $H_n$.
\end{lemma}

\begin{proof} By the definition of \algB, the output graph
contains no red edges. Thus it is a $d$-factor of $H_n$. Recall that the parameters $\rho_{\tau}(i)$ are set according to the solution of (\ref{qA})--(\ref{rec}). Hence, the expected number of times that a graph in $\strata_0$ is visited equals $\sigma_0$, by~(\ref{sigma}), which is the same for every graph in $\strata_0$. It follows that
for every $d$-factor $G$ of $H_n$, the probability that $G$ is the output of \algB\ is
$
\sigma_0\rho_{\text{I}}(0)$, which, again, is independent of $G$. Hence, \algB\ is a
uniform sampler for $d$-factors of $H_n$. 
\end{proof}

\subsection{Rejection probability and number of switching steps}

We bound the number of switching steps performed by \algB\ and 
the probability of any rejection in \algB.
The proofs of Lemmas~\ref{lem:iterations} and~\ref{lem:rejections} are standard, and very similar to those in~\cite{GWSIAM}. 

\begin{lemma}\label{lem:iterations}
\algB\ performs at most $O(\imax)$ switching steps 
in expectation and with high probability.
\end{lemma}

\begin{proof}
The proof is almost identical to~\cite[Lemma 8]{GWSIAM}. We omit the details 
here and only sketch the main idea. It is easy to verify that with probability 
$1-o(1)$, a Type~I Class~A switching is performed in each step. Thus, we can easily bound the probability by $o(1)$ that more than 5\% of steps are implemented with a switching that is not of Type~I Class~A. With a Type I~Class~A switching, the number of red edges reduces by exactly one. On the other hand, the number of red edges can increase by at most~3 in each step. Since initially there are at most $\imax$ red edges, it follows immediately that with high probability the number of switching steps performed by
\algB\ is $O(\imax)$.
\end{proof}

\begin{lemma}\label{lem:rejections}
Assume that $\epsilon \, \imax=o(1)$. Then the probability of a t-rejection, or an f-rejection, or a b-rejection, or pre-b-rejection occurring in \algB\ is $o(1)$.
\end{lemma}

\begin{proof}
By Lemma~\ref{lem:solution}, the probability of a t-rejection in each step is at most $\epsilon$. Thus, the probability of a t-rejection in \algB\ is $O(\epsilon \imax)=o(1)$ by Lemma~\ref{lem:iterations}.

Next, consider f-rejections. Let $G_t$ be the graph obtained after $t$ 
switching steps. Given $G_t=G\in\strata_i$, the probability that an f-rejection occurs at step $t+1$ is
\[
\sum_{\tau\in\Gamma}\rho_{\tau}(i)\left(1-\frac{f_{\tau}(G)}{\UB_{\tau}(i)}\right)
\]
Summing over all $G\in\strata_i$ and summing over all steps $t$, 
the probability of an f-rejection in \algB\ is at most
\[
\sum_{t\ge 0} \, \sum_{0\le i\le \imax}\, \sum_{G\in\strata_i}\pr(G_t=G)\sum_{\tau\in\Gamma}\rho_{\tau}(i)\left(1-\frac{f_{\tau}(G)}{\UB_{\tau}(i)}\right).
\]
Next, we verify that for every $\tau\in\Gamma$ and for every $i\le\imax$, 
\[
\rho_{\tau}(i)\left(1-\frac{f_{\tau}(G)}{\UB_{\tau}(i)}\right)=O\left(\frac{d^2+\Delta^2}{n^2}+\frac1n+\frac{\Delta^2(d+\Delta)}{dn^2}\right).
\]
For $\tau\notin \Gamma\setminus\{\text{I,III}\pm\}$, we use the bounds in~(\ref{eq:IIbprob})--(\ref{eq:IIcprob}) for $\rho_{\tau}(i)$ and the trivial bound 1 for $1-f_{\tau}(G)/\UB_{\tau}(i)$. For $\tau=\text{III}\pm$, we use~(\ref{eq:IIIprob}) for $\rho_{\tau}(i)$ and Lemma~\ref{lem:LBC} for $1-f_{\tau}(G)/\UB_{\tau}(i)$. Lastly, for $\tau=I$, we use the trivial bound 1 for $\rho_{\text{I}}(i)$ and Lemma~\ref{lem:ff} for $1-f_{\text{I}}(G)/\UB_{\text{I}}(i)$. These yield the desired bound above.

We also have
\[
\sum_{t\ge 0} \, \sum_{0\le i\le \imax}\, \sum_{G\in\strata_i}\pr(G_t=G) = 
\sum_{0\le i\le \imax}\, \sum_{G\in\strata_i}\, \sum_{t\ge 0}\pr(G_t=G) =  
  \sum_{0\le i\le \imax} \sigma_i|\strata_i|,
\]
which is the number of switching steps in \algB. By Lemma~\ref{lem:iterations}, 
this is $O(\imax)$ in expectation and with high probability. 
Thus, the probability of an f-rejection is at most
\[
O\left(\frac{d^2+\Delta^2}{n^2}+\frac1n+\frac{\Delta^2(d+\Delta)}{dn^2}\right)\, O(\imax) =O\left(\frac{(d^2+\Delta^2)d\Delta}{n^2}+\frac{d\Delta}{n}+\frac{\Delta^3(d+\Delta)}{n^2}\right),
\]
and this is $o(1)$ when $d^2+\Delta^2=o(n)$.

Next, consider pre-b-rejections. Pre-b-rejections can happen when a switching
of type 
$\tau\in\{\IIbpm,\, \IIcpm\}$ is performed. Given $G_t=G\in\strata_i$, the probability that a pre-b-rejection occurs at step $t+1$ is
\[
\sum_{\tau\in\{\IIbpm,\, \IIcpm\}} \, \rho_{\tau}(i) \left(1-\frac{{\widehat \LB}_{\tau}(i)}{{\widehat b}_{\tau}(G, \boldsymbol{v})}\right).
\]
By Lemma~\ref{lem:IVa}, 
\[
\left(1-\frac{{\widehat \LB}_{\tau}(i)}{{\widehat b}_{\tau}(G, \boldsymbol{v})}\right) =O((d+\Delta)/n).
\]
Thus, by Lemma~\ref{lem:solution},  the probability of a pre-b-rejection occurring at step $t+1$, given $G_t=G$, is
\[
\sum_{\tau\in\{\IIbpm,\, \IIcpm\}}  O(\epsilon\, (d+\Delta)/n) = O\left(\frac{d^3+\Delta^3}{n^3}\right).
\]
Arguing as above, the probability of any pre-b-rejection during \algB\ is at most
\begin{align*}
\sum_{t\ge 0}\, \sum_{i\le \imax} \, \sum_{G\in\strata_i} \pr(G_t=G) \cdot  
 O\left(\frac{d^3+\Delta^3}{n^3}\right) 
  &=  O\left(\frac{d^3+\Delta^3}{n^3}\right)\, \imax\\
 &=  O\left(\frac{(d^3+\Delta^3)d\Delta}{n^3}\right), 
\end{align*}
and this is $o(1)$ when $d^2+\Delta^2 = o(n)$.

Finally, we consider b-rejections. Let $\Psi_{\tau,\alpha}(G,G')$ denote the set of switchings of 
Type~$\tau$ and Class~$\alpha$ which convert $G$ to $G'$.
An element of $\Psi_{\tau,\alpha}(G,G')$  is either an 8-tuple $\boldsymbol{v}$ such that
$(G,\boldsymbol{v})\mapsto G'$ is a switching of Type~$\tau$ and Class~$\alpha$,
if $(\tau,\alpha)\not\in \{\IIbpm,\Bipm),\, (\IIcpm,\Bipm)\}$  
or a pair $(\boldsymbol{v},\boldsymbol{y})$ which
determines a switching $(G,\boldsymbol{v},\boldsymbol{y})\mapsto G'$ of Type $\tau$ and Class $\Bipm$,
where $\tau\in \{ \IIbpm,\, \IIcpm\}$.
Let
\[
\Psi_{\tau,\alpha}(G')=\bigcup_{G} \Psi_{\tau,\alpha}(G,G')
\]
be the set of switchings into $G'$ which are of Type~$\tau$ and Class~$\alpha$.  
For any $S\in\Psi_{\tau,\alpha}(G,G')$ where $G\in \strata_i$ and $G'\in\strata_{i'}$ and 
$\Psi_{\tau,\alpha}(G,G')\neq \emptyset$, 
the probability that $S$ is performed and b-rejected, given $G_t=G$, is
\[
\frac{\rho_{\tau}(i)}{\UB_{\tau}(i)} \left(1-\frac{\LB_{\alpha}(i')}{b_{\alpha}(G')}\right).
\]
Then the probability that a b-rejection ever occurs in \algB\ is
\[
\sum_{\alpha}\, \sum_{\tau}\, \sum_{t\ge 1}\, \sum_{i'\le \imax} \,
\sum_{G'\in \strata_{i'}}\, \sum_{G} \,
  \pr(G_{t-1}=G)\, |\Psi_{\tau,\alpha}(G,G')|\, 
     \frac{\rho_{\tau}(i) }{\UB_{\tau}(i)} \left(1-\frac{\LB_{\alpha}(i')}{b_{\alpha}(G')}\right).
\]
Here $i$ is the index of the strata which contains $G$, which is determined by 
$\tau$, $\alpha$ and $i'$.  Hence, the above sum is
\begin{eqnarray*}
&&\sum_{\alpha}\, \sum_{i'\le \imax}\, \sum_{\tau}\, \frac{\rho_{\tau}(i)}{\UB_{\tau}(i)} \, \sum_{G'\in \strata_{i'}}  \frac{b_{\alpha}(G')-\LB_{\alpha}(i')}{b_{\alpha}(G')}\, \sum_{G}\, |\Psi_{\tau,\alpha}(G,G')| \, 
    \sum_{t\ge 1}\pr(G_{t-1}=G)\\
&&=\hspace{0.5cm}\sum_{\alpha}\, \sum_{i'\le \imax}\, \sum_{\tau}\,
\frac{\rho_{\tau}(i) \sigma_i}{\UB_{\tau}(i)} \, \sum_{G'\in \strata_{i'}} 
\,  \frac{b_{\alpha}(G')-\LB_{\alpha}(i')}{b_{\alpha}(G')} |\Psi_{\tau,\alpha}(G')|,
\end{eqnarray*}
since $\sum_{t\ge 1}\pr(G_{t-1}=G)=\sigma_{i}$. 
By the design of the algorithm, the expected number of times that $G'$ is reached via 
a Type~$\tau$, Class~$\alpha$ switching equals $q_{\alpha}(i')$ for all relevant
$\tau$, as displayed in~(\ref{qA})--(\ref{qC}).  Therefore, for every $\tau$,
\[
\frac{\rho_{\tau}(i) \sigma_i}{\UB_{\tau}(i)} = q_{\alpha}(i') = \frac{\rho_{\text{I}}(i'+1) \sigma_{i'+1}}{\UB_{\tau}(i'+1)}.
\]
Thus the above summation is
\[
\sum_{\alpha}\, \sum_{i'\le \imax}\, \frac{\rho_{\text{I}}(i'+1) \, 
  \sigma_{i'+1}}{\UB_{\text{I}}(i'+1)} \, \sum_{G'\in \strata_{i'}}  \,
  \frac{b_{\alpha}(G')-\LB_{\alpha}(i')}{b_{\alpha}(G')}\,
   \sum_{\tau}\, |\Psi_{\tau,\alpha}(G')|.
\]
Since $\sum_{\tau}\, |\Psi_{\tau,\alpha}(G')| = b_{\alpha}(G')$,
 by definition of $b_{\alpha}(G')$, the probability of a b-rejection in \algB\ is
\begin{eqnarray}
&&\sum_{\alpha}\, \sum_{i'\le \imax}\, \frac{\rho_{\text{I}}(i'+1) \sigma_{i'+1}}{\UB_{\text{I}}(i'+1)} \, \sum_{G'\in \strata_{i'}}  \Big(b_{\alpha}(G')-\LB_{\alpha}(i')\Big) \nonumber \\
&&= \sum_{\alpha}\, \sum_{i'\le \imax}\, \frac{\rho_{\text{I}}(i'+1) 
   \sigma_{i'+1}}{\UB_{\text{I}}(i'+1)} \, |\strata_{i'}| \,
    \Big(\ex b_{\alpha}(G')-\LB_{\alpha}(i')\Big)\nonumber\\
&&\le  \sum_{\alpha}\, \sum_{i'\le \imax}\,
   \frac{(i'+1)\sigma_{i'}}{d\Delta\cdot\UB_{\text{I}}(i'+1)} \, |\strata_{i'}|
\,  \Big(\ex b_{\alpha}(G')-\LB_{\alpha}(i')\Big) \nonumber\\
 &&\hspace{2cm}\mbox{(since $\rho_{\text{I}}(i'+1)\le 1$ and $\sigma_{i'+1}=O((i+1)\sigma_{i'}/d\Delta)$ by Lemma~\ref{lem:solution})}\nonumber\\
&&\hspace{0.5cm} \le  \frac{1}{d\Delta(dn)^3}\, \sum_{\alpha}\,
   \sum_{i'\le \imax}\, 
  \sigma_{i'}\, |\strata_{i'}|  \Big(\ex b_{\alpha}(G')-\LB_{\alpha}(i')\Big),  \label{eq:b-rej}
\end{eqnarray}
where $\ex b_{\alpha}(G')$ is the expectation of $b_{\alpha}(G')$ on a uniformly random $G'\in\strata_{i'}$.

\medskip

For $\alpha=\text{A}$, by Lemma~\ref{lem:bb}, 
\[
\ex b_{\text{A}}(G') - \LB_{\text{A}}(i') = O((d^2+\Delta^2)/n^2+1/n) \,
  \LB_{\text{A}}(i') = O((d^2+\Delta^2)/n^2+1/n)\Delta n d^2 (dn)^2.
\] 
Thus, the contribution to~(\ref{eq:b-rej}) from $\alpha=\text{A}$ is
\[
O\left(\Big(\frac{d^2+\Delta^2}{n^2}+\frac{1}{n}\Big)\frac{\Delta n d^2 (dn)^2}{d\Delta(dn)^3}\right)\sum_{i'\le \imax}\sigma_{i'}|\strata_{i'}|=o(1),
\]
as $\sum_{i'\le \imax}\sigma_{i'}|\strata_{i'}|=O(\imax)=O(d\Delta)$.

\medskip

For $\alpha\in\{\Bipm\}$, by Lemma~\ref{lem:LBB1}, 
\begin{align*}
\ex b_{\alpha}(G') - \LB_{\alpha}(i') &= O(1/d+(d+\Delta)/n)\, \LB_{\alpha}(i') 
=  O(1/d+(d+\Delta)/n) \imax \Delta d(dn)^2\\
&=O(1/d+(d+\Delta)/n) d^4 \Delta^2 n^2.
\end{align*} 
Thus, the contribution to~(\ref{eq:b-rej}) from $\alpha\in\{\Bipm\}$ is
\[
\frac{d^4 \Delta^2 n^2}{d\Delta(dn)^3}\cdot O(1/d+(d+\Delta)/n) \, \imax 
= O\left(\frac{d^5 \Delta^3 n^2}{d\Delta(dn)^3}\cdot 
  \left(\frac1d+\frac{d+\Delta}{n}\right)\right)=o(1).
\]
For $\alpha\in\{\Biipm\}$, by Lemma~\ref{lem:bB2}, 
\[
\ex b_{\alpha}(G') - \LB_{\alpha}(i') = O((d+\Delta)/n)\, \LB_{\alpha}(i') 
   =  O((d+\Delta)/n) \Delta^2d^4 n^2.
\]
Thus, the contribution to~(\ref{eq:b-rej}) from $\alpha\in\{\Biipm\}$ is
\[
\frac{\Delta^2d^4 n^2}{d\Delta(dn)^3}\cdot O((d+\Delta)/n)  \imax = O\left(\frac{d^5 \Delta^3 n^2}{d\Delta(dn)^3}\cdot \left(\frac{d+\Delta}{n} \right)\right)=o(1).
\]
For $\alpha=\Cpm$, by Lemma~\ref{lem:LBC},
\[
\ex b_{\alpha}(G') - \LB_{\alpha}(i') = O((d+\Delta)/n)\,
    \LB_{\alpha}(i') =  O((d+\Delta)/n) \Delta^3d^3 n^2.
\]
Thus, the contribution to~(\ref{eq:b-rej}) from $\alpha\in\{ \Cpm\}$ is
\[
\frac{\Delta^3d^3 n^2}{d\Delta(dn)^3}\cdot O((d+\Delta)/n)\,  \imax = O\left(\frac{d^4 \Delta^4 n^2}{d\Delta(dn)^3}\cdot \left(\frac{d+\Delta}{n} \right)\right)=o(1).
\]
This completes the proof that the probability of any b-rejection in \algB\ is $o(1)$.
\end{proof}

\section{Deferred analysis of time complexity and distance from uniform}\label{sec:runtime}

In this section we present the deferred analysis of the time complexity of
algorithms \algA\ and \algB, and the proof that
the output of \algC\ is within $o(1)$ of uniform.
As usual, asymptotics are as $n\to\infty$
where $d=d(n)$ and $\Delta=\Delta(n)$ satisfy the assumptions of the relevant theorem.

\bigskip

\Arun*
\begin{proof}
In each iteration, the switching of a bounded number
of edges can be done in $O(1)$ time. The time-consuming part is to compute $b(G)$ to 
determine the probability of a b-rejection. 

First consider the initial graph $G$ produced by REG.
To compute $b(G)$ for the initial graph, we use brute force to search for all possible choices of 
$v_5,v_0,v_1,v_2$. This can be done in time $O(d^2\Delta n)$. Denote the number 
of choices by $X$. Multiplying $X$ by $dn-2i$ gives the first estimate for
$b(G)$. Next we accurately compute $b(G)$ using inclusion-exclusion. The choices of $v_3$ and $v_4$ must satisfy a set $U$ of constraints. We can count the number
of choices which satisfy all constraints using inclusion-exclusion. 
The inclusion-exclusion argument involves a bounded number of terms counting 
choices where a subset 
$W\subseteq U$ of constraints are violated. We will show that each such term 
can be computed in time $O((d+\Delta)^3n)$ with the aid of a
proper data structure.   

Let $\widetilde{G}$ be the supergraph of $G$ consisting of $G$ together
with all red edges in $K_{n}\setminus G$. We use $\widetilde{\text{red}}$ for
the colour of edges in $\widetilde{G}$ which are not in $G$.
When computing $X$ using brute force search, we can record the 
number of 3-paths in $\widetilde{G}$ between any two vertices of any type 
(for example, red-black-red, or black-red-black, or 
black-$\widetilde{\text{red}}$-black);
we can also record the number of 
3-paths and 2-paths in $\widetilde{G}$ of any given type starting from any given 
vertex.  The time complexity for computing all these numbers is $O((d+\Delta)^3 n)$ 
since the maximum degree in $\widetilde{G}$ is bounded above by $d+\Delta$. 
The number of other local structures of at most 4 vertices can be computed 
and recorded within this time complexity bound,  
that is, triangles, 4-cycles, etc. We can also record lists of pairs of vertices 
which are joined by a 3-path, or 2-path, or an edge, within the same time complexity. 

Given $W$, let $b_W$ be the choices of the sequence of the six vertices that violate 
constraints in $W$ (here constraints in $U\setminus W$ may or may not be violated).
Since the structure counted by $b_W$ uses up to six vertices, it is easy to see that $b_W$ can be computed using the numbers we have recorded.  For instance, if $W=\emptyset$ then $b_W=X(dn-2i)$. If $W=\{\mbox{$v_4v_5$ is a black non-edge}\}$, then $b_W=b_{W,1}+b_{W,2}+b_{W,3}$, where $b_{W,1}$ counts those choices where $W$ is violated by 
taking $v_4v_5$ as a red edge; $b_{W,2}$ counts those choices where $v_4v_5$ is a black edge, and $b_{W,3}$ counts those choices where $v_4v_5$ is a red non-edge. 
In each case, we can run through $n$ choices for $v_5$, and compute $b_{W,i}$ using the number of 3-paths and 2-paths starting from $v_5$ that have been recorded. The time complexity is then $O(n)$. For every other $W$ it is easy to check that a similar scheme works. Thus, it takes $O((d+\Delta)^3 n)$ time to compute $b(G)$ for the 
%first switching step of \algA. 
initial graph $G$.

Next suppose that $G$ is produced by a switching step, during the run of \algA.
%For the subsequent switching steps, 
We do not need to recompute $b(G)$ from scratch: instead, we can update the data recorded in our 
data structure very efficiently, because only 3 new edges are added, and 3 edges are 
deleted. Since the data we store are counts of structures involving only up to 
4 vertices, changing each edge will alter at most $O((d+\Delta)^2)$ entries. 
For each entry change, we can update $b(G)$ by updating the corresponding
$b_W$ terms in the 
inclusion-exclusion formula. Thus the time complexity for computing $b(G)$ is 
$O((d+\Delta)^2)$ after each subsequent switching step and there are 
$O(d\Delta)$ switching steps in expectation, since \algA\ restarts $O(1)$ times
in expectation. 
Thus the total time complexity for \algA\ is 
\[ O((d+\Delta)^3 n+d\Delta(d+\Delta)^2)=O((d+\Delta)^3 n)\] 
in expectation, completing the proof.
\end{proof}

\bigskip

For convenience, we restate Theorem~\ref{thm:B} below.

%\noindent{\em Proof of Theorem~\ref{thm:B}.} 
\thmB*
\begin{proof}
In Lemma~\ref{lem:uniform} we have shown that \algB\ is a uniform sampler.
It only remains to prove the efficiency. By Corollary~\ref{cor3:imax}, \algB\ restarts only $O(1)$ times in expectation and $O(\log n)$ times a.a.s.\ before finding a $d$-regular graph containing at most $\imax$ red edges. The total time complexity for finding such a graph is $O(d^3n)$ in expectation, and $O(d^3n\log n)$ a.a.s..
By Lemma~\ref{lem:rejections}, the probability that \algB\ restarts afterwards is $o(1)$. Thus, we only need to bound the remaining runtime of \algB\ assuming no rejections. By Lemma~\ref{lem:iterations}, after finding a $d$-regular graph with at most $\imax$ red edges, \algB\ will perform $O(\imax)$ switching steps in expectation 
and with high probability. In each switching step, the most time-consuming part is to compute $f_{\tau}(G)$, ${\widehat b}_{\tau}(G,\boldsymbol{v})$, and $b_{\alpha}(G)$ for $\tau\in\{\text{I}, \IIapm, \IIbpm, \IIcpm, \IIIpm\}$ and $\alpha\in\{\text{A},\Bipm,\Biipm,\Cpm\}$.

We first bound the a.a.s.\ time complexity. Note that $f_{\tau}(G)$ and 
${\widehat b}_{\tau}(G,\boldsymbol{v})$ will only need to be evaluated once a type 
$\tau$ switching is performed. By~(\ref{cond2}) and~(\ref{epsilon}), the 
probability that a type $\tau$ switching is ever performed in \algB\ for any 
$\tau\notin\{\text{I},\IIIpm\}$ is $O(\epsilon \imax)=o(1)$. It follows immediately that a.a.s., only $f_{\tau}(G)$ for $\tau\in\{\text{I},\IIIpm\}$ and 
$b_{\alpha}(G)$, $\alpha\in\{\text{A},\Bipm,\Biipm,\Cpm\}$, will ever be computed during the implementation of \algB.

First consider $f_{\tau}(G)$ for $\tau=\text{I}$. We want to count choices of 
$(v_0,\ldots, v_7)$ such that $v_0v_1$ is a red edge in $G$, $v_2v_3$ and $v_6v_7$ 
are edges (red or black) in $G$, and $v_4v_5$ is a black edge in $G$, satisfying a set $U$ of constraints (that is, no vertex collision except for $v_2=v_7$ and certain edges are forbidden in $G$ and must be with certain colour in $K_n$). As before, using inclusion-exclusion, we can express this number by $b_W$, $W\subseteq U$, where $b_W$ is the number of choices where the conditions in $W$ are violated. 
We use similar data structures as in Theorem~\ref{thm:A}, but we record counts of 
structures containing up to 5 vertices. 
Thus the time complexity for constructing the data structures is 
$O((d+\Delta)^4n)$ in the first iteration. It is easy to see, as in the proof of
Theorem~\ref{thm:A}, that all  terms $b_W$ in the inclusion-exclusion
can be computed using the recorded data. 
Moreover, it takes $O((d+\Delta)^3)$ time to update the data structure after each 
subsequent switching step. Thus, the total time complexity for computing 
$f_{\text{I}}(G)$ throughout \algB\ is 
$O((d+\Delta)^4n+d\Delta(d+\Delta)^3)=O((d+\Delta)^4n)$. The same
time complexity bound holds for computing $f_{\text{III}+}(G)$ 
and $f_{\text{III}-}(G)$ throughout \algB,
as the same number of vertices are involved in a Type~\textIIIpm\ switching
as in a Type~I switching.

Next we consider $b_{\alpha}(G)$. Similar arguments as for $f_{\text{I}}(G)$ 
show that the time complexity of computing
$b_{\alpha}(G)$ for every $\alpha$ throughout \algB\ is at most
$O((d+\Delta)^4n)$. 
The Gao--Wormald algorithm~\cite{GWSIAM}, used to produce
the initial $d$-regular graph,
has time complexity $O(d^3 n)$ in expectation. Hence a.a.s.\ the time
complexity to produce the initial $d$-regular graph is $O(d^3 n\log n)$.
Therefore, the a.a.s.\ time complexity bound for \algB\ is
$O(d^3n\log n+(d+\Delta)^4n)$. 

%recalling that $O(d^3n\log n)$ is an a.a.s.\ upper bound for finding an initial $d$-regular graph with at most $\imax$ red edges using the
%Gao--Wormald algorithm~\cite{GWSIAM}.

Now we consider the time complexity in expectation. To do this, we obtain an 
upper bound for the time complexity of computing $f_{\tau}(G)$ and 
${\widehat b}_{\tau}(G,\boldsymbol{v})$, $\tau\notin\{\text{I},\IIIpm\}$, and then 
multiply by the probability that the switching type $\tau$ is chosen in a single 
step, and finally multiply by $O(\imax)=O(d\Delta)$, which is an upper bound for the 
expected
number of switching steps performed by \algB. These switchings are performed rarely,
so we do not attempt to update the data after every switching step. 
Instead we simply reconstruct the data structure whenever it is needed.  
Obviously for $d$ and $\Delta$ in different ranges, different counting schemes can be 
used to optimise the runtime: we have not attempted this. 
Here we simply use the scheme which naturally extends that given
in Theorem~\ref{thm:A}.

For each $\tau$, we use data structures to record counts of connected small structures up to $j$ vertices, where 
\[  j = 
\begin{cases}
5 & \text{ for $\tau\in \{\IIapm\}$,}\\
9 & \text{ for $\tau\in \{\IIbpm\}$,}\\ 
11 & \text{ for $\tau\in \{\IIcpm\}$. }
\end{cases}
\]
This leads to the following bounds on the complexity of computing $f_{\tau}(G)$
for a particular $G$,
in these cases: 
\begin{eqnarray*}
&&
\IIapm:\ O((d+\Delta)^4n);
\qquad 
\IIbpm:\ O((d+\Delta)^8n); 
\qquad \IIcpm:\ O((d+\Delta)^{10}n).
\end{eqnarray*}
Now for $\tau\neq \text{I}$, the Type $\tau$ switchings are only 
implemented occasionally. 
Let $\rho_{\tau}^*=\max_{0\le i\le \imax} \rho_{\tau}(i)$. Multiplying the 
above bounds by $O(\imax) \rho_{\tau}^*$, using
(\ref{eq:IIIprob})--(\ref{eq:IIcprob}), yields the following overall bounds on the 
expected time complexity for computing $f_{\tau}(G)$ during \algB:
\begin{eqnarray*}
&&
\IIapm:\quad O((d+\Delta)^4 d\Delta^3/n);
\qquad 
\IIbpm:\quad \ O((d+\Delta)^8 d^2\Delta^2/n ); \\
&& \IIcpm:\quad O((d+\Delta)^{10} d^2\Delta^3 /n^2);\qquad
\IIIpm:\quad  O((d+\Delta)^{4}\Delta^3 ).
\end{eqnarray*}

Combining the contribution from every type $\tau$, the expected time complexity for computing $f_{\tau}(G)$ 
throughout  \algB\ is bounded above by
\begin{equation}
O\Big((d+\Delta)^4(n+\Delta^3)+(d+\Delta)^8d^2\Delta^2/n +(d+\Delta)^{10} d^2\Delta^3 /n^2\Big). \label{final}
\end{equation}

Finally, we consider computation of 
${\widehat b}_{\tau}(G,\boldsymbol{v})$ for $\tau\in\{\IIbpm,\, \IIcpm\}$,
for a given 8-tuple~$\boldsymbol{v}$. 
(Recall that these are the only switching types which have pre-b-rejections.)
Due to~(\ref{final}) we only need rough bounds for 
the time complexity of computing ${\widehat b}_{\tau}(G,\boldsymbol{v})$. 
For Type $\IIbpm$, 
it is sufficient to use a data structure to record counts of connected structures 
involving up to 5 vertices, or up to 9 vertices, one of which belongs
to $\boldsymbol{v}$.
The runtime to construct data structures of the first type is $O((d+\Delta)^4n)$; 
and $O((d+\Delta)^8)$ for the second type. For Type $\IIcpm$,
similar arguments
give an upper bound of 
\[ O\big((d+\Delta)^6n+(d+\Delta)^{12}\big)\]
 on the complexity
of constructing the data structure. Multiplying these bounds 
by $O(\imax) \rho_{\tau}^*$, using (\ref{eq:IIbprob}) and (\ref{eq:IIcprob}),
yields the following upper bound on the
complexity of computing ${\widehat b}_{\tau}(G,\boldsymbol{v})$ during the 
implementation of \algB:
\begin{align*}
\IIbpm: & \qquad O\left(\frac{d^2\Delta^2(d+\Delta)^4}{n} +\frac{d^2\Delta^2(d+\Delta)^8}{n^2} \right); \\
\IIcpm: &  \qquad O\left(\frac{d^2\Delta^3(d+\Delta)^6}{n^2} 
  +\frac{d^2\Delta^3(d+\Delta)^{12}}{n^3} \right).
\end{align*}
Since $d^2+\Delta^2 = o(n)$, the above terms are dominated by~(\ref{final}). 
Thus,
combining everything together, the expected time complexity of \algB\ is 
bounded by~(\ref{final}).
\end{proof}

\bigskip

Finally, we restate and prove Lemma~\ref{lem:algC-distribution}. 

%\noindent {\em Proof of Theorem~\ref{thm:C}.} 
\Cdist*
\begin{proof}
Recall that \algB\ calls REG to generate a uniformly random $d$-regular graph on $[n]$, whereas \algC\ calls REG* which generates an approximately uniformly random $d$-regular graph. It was proved in~\cite[Section 10]{GWSIAM} that REG and REG* can be coupled so that with probability $1-o(1)$ they have the same output. Assume REG and REG* both output $G$ which contains at most $\imax$ red edges.  Consider continuing the run of \algB\ but
restricting the choice of type to $\tau\in\{\text{I},\IIIpm\}$ and
ignoring all t-rejections, f-rejections, b-rejections and pre-b-rejections. 
Recall that Type \textIIIpm\ switchings do not switch any edges. If we ignore the 
f-rejections in the implementation of Type~\textIIIpm\ switchings then the switching 
has no effect and we may simply skip that step. Hence this
modification of \algB\ behaves identically to \algC, that is,
by repeatedly performing valid Type~I switchings. 

Therefore, we can couple the implementation of \algB\ and \algC\ such that \algB\ and 
\algC\ output the same graph $G$ as long as no rejections occur in \algB\ and no 
types of switchings other than $\text{I}, \IIIpm$ are chosen in \algB.  
By Lemma~\ref{lem:rejections}, the probability of performing any rejection in 
\algB\ is $o(1)$. By~(\ref{cond2}) and~(\ref{epsilon}), the probability of 
performing any switchings in \algB\ of type other than $\text{I}, \IIIpm$ is 
$O(\epsilon)=o(1)$.  Hence, the total variation distance between the output of 
\algC, and that of \algB, which is uniform, is $o(1)$. 
Here $o(1)$ also accounts for the probability that the coupled REG and REG* have distinct outputs. 
\end{proof}

\subsection*{Acknowledgements} We thank Brendan McKay for his idea of permitting $v_2=v_5$ in a 3-edge-switching and the similar treatment in a 4-edge-switching. This crucial idea significantly simplified our algorithms and analysis.
We also thank the referee for their helpful comments.

\end{document}